\documentclass[11pt,a4paper]{amsart}

\usepackage{amssymb,amsmath,epsfig,graphics,mathrsfs}
\usepackage{graphicx}

\usepackage{fancyhdr}
\usepackage{hyperref}
\hypersetup{
 colorlinks   = true,
 urlcolor     = blue,
 linkcolor    = blue,
 citecolor   = red ,
 bookmarksopen=true
}

\newtheorem{thrm}{Theorem}[section]
\newtheorem{lemma}[thrm]{Lemma}
\newtheorem{prop}[thrm]{Proposition}
\newtheorem{cor}[thrm]{Corollary}
\newtheorem{rmrk}[thrm]{Remark}
\newtheorem{dfn}[thrm]{Definition}

\begin{document}

\newcommand{\SL}{\mathcal L^{1,p}( D)}
\newcommand{\Lp}{L^p( Dega)}
\newcommand{\CO}{C^\infty_0( \Omega)}
\newcommand{\Rn}{\mathbb R^n}
\newcommand{\Rm}{\mathbb R^m}
\newcommand{\R}{\mathbb R}
\newcommand{\Om}{\Omega}
\newcommand{\Hn}{\mathbb H^n}
\newcommand{\aB}{\alpha B}
\newcommand{\eps}{\ve}
\newcommand{\BVX}{BV_X(\Omega)}
\newcommand{\p}{\partial}
\newcommand{\IO}{\int_\Omega}
\newcommand{\bG}{\boldsymbol{G}}
\newcommand{\bg}{\mathfrak g}
\newcommand{\bz}{\mathfrak z}
\newcommand{\bv}{\mathfrak v}
\newcommand{\Bux}{\mbox{Box}}
\newcommand{\e}{\ve}
\newcommand{\X}{\mathcal X}
\newcommand{\Y}{\mathcal Y}
\newcommand{\W}{\mathcal W}
\newcommand{\la}{\lambda}
\newcommand{\vf}{\varphi}
\newcommand{\rhh}{|\nabla_H \rho|}
\newcommand{\Ba}{\mathcal{B}_\beta}
\newcommand{\Za}{Z_\beta}
\newcommand{\ra}{\rho_\beta}
\newcommand{\na}{\nabla_\beta}
\newcommand{\vt}{\vartheta}
\newcommand{\G}{\Gamma}
\newcommand{\Ga}{\overline{G}}
\newcommand{\La}{\mathscr L_a}
\newcommand{\Gb}{\overline{\mathscr G}_a}
\newcommand{\Gbb}{\overline{\mathscr G}}

\numberwithin{equation}{section}

\newcommand{\RN} {\mathbb{R}^N}
\newcommand{\Sob}{S^{1,p}(\Omega)}
\newcommand{\Dxk}{\frac{\partial}{\partial x_k}}
\newcommand{\Co}{C^\infty_0(\Omega)}
\newcommand{\Je}{J_\ve}
\newcommand{\beq}{\begin{equation}}
\newcommand{\bea}[1]{\begin{array}{#1} }
\newcommand{\eeq}{ \end{equation}}
\newcommand{\ea}{ \end{array}}
\newcommand{\eh}{\ve h}
\newcommand{\Dxi}{\frac{\partial}{\partial x_{i}}}
\newcommand{\Dyi}{\frac{\partial}{\partial y_{i}}}
\newcommand{\Dt}{\frac{\partial}{\partial t}}
\newcommand{\aBa}{(\alpha+1)B}
\newcommand{\GF}{\psi^{1+\frac{1}{2\alpha}}}
\newcommand{\GS}{\psi^{\frac12}}
\newcommand{\HFF}{\frac{\psi}{\rho}}
\newcommand{\HSS}{\frac{\psi}{\rho}}
\newcommand{\HFS}{\rho\psi^{\frac12-\frac{1}{2\alpha}}}
\newcommand{\HSF}{\frac{\psi^{\frac32+\frac{1}{2\alpha}}}{\rho}}
\newcommand{\AF}{\rho}
\newcommand{\AR}{\rho{\psi}^{\frac{1}{2}+\frac{1}{2\alpha}}}
\newcommand{\PF}{\alpha\frac{\psi}{|x|}}
\newcommand{\PS}{\alpha\frac{\psi}{\rho}}
\newcommand{\ds}{\displaystyle}
\newcommand{\Zt}{{\mathcal Z}^{t}}
\newcommand{\XPSI}{2\alpha\psi \begin{pmatrix} \frac{x}{|x|^2}\\ 0 \end{pmatrix} - 2\alpha\frac{{\psi}^2}{\rho^2}\begin{pmatrix} x \\ (\alpha +1)|x|^{-\alpha}y \end{pmatrix}}
\newcommand{\Z}{ \begin{pmatrix} x \\ (\alpha + 1)|x|^{-\alpha}y \end{pmatrix} }
\newcommand{\ZZ}{ \begin{pmatrix} xx^{t} & (\alpha + 1)|x|^{-\alpha}x y^{t}\\
     (\alpha + 1)|x|^{-\alpha}x^{t} y &   (\alpha + 1)^2  |x|^{-2\alpha}yy^{t}\end{pmatrix}}
\newcommand{\norm}[1]{\lVert#1 \rVert}
\newcommand{\ve}{\varepsilon}
\newcommand{\Rnn}{\mathbb R^{n+1}}
\newcommand{\Rnp}{\mathbb R^{n+1}_+}
\newcommand{\B}{\mathbb{B}}
\newcommand{\Ha}{\mathbb{H}}
\newcommand{\xa}{X}
\newcommand{\Sa}{\mathbb{S}}

\title[Monotonicity of generalized frequencies, etc.]{Monotonicity of generalized frequencies and the strong unique continuation property for fractional parabolic equations}

\author{Agnid Banerjee}
\address{TIFR CAM, Bangalore-560065} \email[Agnid Banerjee]{agnidban@gmail.com}

\author{Nicola Garofalo}
\address{Dipartimento di Ingegneria Civile, Edile e Ambientale (DICEA) \\ Universit\`a di Padova\\ 35131 Padova, ITALY}
\email[Nicola Garofalo]{rembdrandt54@gmail.com}

\thanks{Second author supported in part by a grant ``Progetti d'Ateneo, 2013,'' University of Padova.}

\maketitle

\tableofcontents

\begin{abstract}
We  study the strong unique continuation property backwards in time for the nonlocal equation in $\Rnn$
\begin{equation}\label{one}
(\partial_t - \Delta)^{s} u = V(x,t)u,\ \ \ \ \ \ \ \ \ \ s \in (0,1).
\end{equation}
Our main  result Theorem \ref{main} can be thought of as the nonlocal counterpart  of the result obtained in \cite{Po}  for the case when $s=1$. In order to prove Theorem \ref{main} we develop the regularity theory of the extension problem for the equation \eqref{one}. With such theory in hands we establish:
\begin{itemize}
\item[(i)] a basic monotonicity result for an adjusted frequency function which plays a central role in this paper, see Theorem \ref{mon1} below;
\item[(ii)] an extensive blowup analysis of the so-called \emph{Almgren rescalings} associated with the extension problem.
\end{itemize}
We feel that our work will also be of interest e.g. in the study of certain basic open questions in free boundary problems, as well as in nonlocal segregation problems. 
\end{abstract}


\section{Introduction}\label{S:Intro}

In his visionary papers \cite{R1} and \cite{R2} Marcel Riesz introduced the fractional powers of the Laplacian in Euclidean and Lorentzian space, developed the calculus of these nonlocal operators and studied the Dirichlet and Cauchy problems for respectively $(-\Delta)^s$ and $(\p_{tt} - \Delta)^s$. These pseudo-differential operators play an important role in many branches of the applied sciences ranging from elasticity, to geophysical fluid dynamics and to quantum mechanics. But they also appear prominently in other branches of mathematics, such as e.g. geometry, probability and financial mathematics. 

Although the introduction of \cite{R1} reads:...``On peut en particulier consid\'erer certains proc\'ed\'es d'int\'egration de charact\`er elliptique, hyperbolique et parabolique respectivement", M. Riesz' works  do not directly encompass the fractional powers $H^s = (\p_t - \Delta)^s$ of the third fundamental operator of mathematical physics, the heat operator $H = \p_t - \Delta$. Yet, such pseudo-differential  operator plays an important role in many contexts. For instance, when studying the phenomenon of osmosis one is led to consider an obstacle type problem for  $H^{1/2}$, see \cite{DL} and the recent monograph \cite{DGPT}.  

In this paper, given a number $0<s<1$, we are concerned with the basic question of the space-time strong unique continuation (backward in time) for the following nonlocal equation
\begin{equation}\label{e0}
H^s u(x,t)  = V(x,t) u(x,t),
\end{equation}
where $x\in \Rn,\ t\in \R$, and we have indicated by $(x,t)$ the generic point in $\R^{n+1}$. As it is well-known, the problem of uniqueness is central to the analysis of pde's, and in the context of parabolic equations the subject has been extensively developed in the local case $s=1$ in \eqref{e0}. The nonlocal case $0<s<1$ instead is completely open, and our paper constitutes a first basic contribution in this direction.

 To state our main result we
introduce the relevant notion of vanishing to infinite order of  a function at a point. Such notion is by now quite standard in the literature for parabolic differential equations. For $x_0\in \Rn$ and $r>0$ we indicate with $B_r(x_0)$ the Euclidean ball in $\R^{n}$ centered at $x_0$ with radius $r$. Given $(x_0, t_0)\in \Rnn$ we denote by $Q_{r}(x_0, t_0)= B_r (x_0) \times (t_0-r^2, t_0]$ the lower half of the parabolic cylinder in $\Rnn$ with radius $r$ centered at $(x_0,t_0)$. 

\begin{dfn}\label{d11}
We say that a function $u\in L^\infty_{loc}(\Rnn)$ \emph{vanishes to infinite order backward in time} at a point $(x_0,t_0)\in \Rnn$ if as $r\to 0^+$ one has 
\begin{equation*}
\underset{Q_{r}(x_0, t_0)}{\operatorname{essup}}\ |u| = O(r^N),
\end{equation*}
for all $N>0$. 
\end{dfn}
We note that if $u$ is smooth, Definition \ref{d11} is equivalent to saying that all the derivatives of $u$ vanish at $(x_0, t_0)$. This latter fact follows from Taylor's theorem.

Henceforth, given $0<s<1$ we will denote by  Dom$(H^s)\subset L^2(\Rnn)$ the domain of the nonlocal operator $H^s$, see \eqref{dom} below. 
Next, we describe our assumptions on the potential $V$ in \eqref{e0}. We indicate with $C^{k}(\Rnn)$ the Banach space of $C^k$ functions $f:\Rnn\to \R$ for which the norm
\[
||f||_{C^{k}(\Rnn)} = \sum_{|\alpha| + j  \leq k} \underset{(x,t)\in \Rnn}{\sup}\ |D^{\alpha} D_t^{j} f(x,t)| < \infty.
\]
Notice that the finiteness of $||f||_{C^{k}(\Rnn)}$ imposes, in particular, that $D^{\alpha} D_t^{j} f\in L^\infty(\Rnn)$ for $|\alpha| + j  \leq k$.
We assume that $V:\Rnn\to \R$ is such that for some $K>0$ 
\begin{equation}\label{vasump}
\begin{cases}
||V||_{C^{1}(\Rnn)}    \le K,\ \ \ \ \ \text{if}\ 1/2\le s < 1,
\\
\\
||V||_{C^{2}(\Rnn)},\ \ \ \ \   ||<\nabla_x V, x>||_{L^\infty(\Rnn)} \le K,\ \ \ \ \  \text{if}\ 0< s < 1/2.
\end{cases}
\end{equation}

Our main result can now be stated as follows.

\begin{thrm}[Space-time strong unique continuation property]\label{main}
Let $u\in  \operatorname{Dom}(H^{s})$ be a solution to \eqref{e0} with $V$ satisfying \eqref{vasump}. If $u$ vanishes to infinite order backward in time at some point $(x_0, t_0)$ in $\R^{n+1}$ in the sense of  Definition \ref{d11}, then $u(\cdot,t) \equiv 0$ for all $t \leq t_0$.  
\end{thrm}

We mention that in the local case, i.e., when $s=1$, a related backward  uniqueness result was established by Poon in \cite{Po} by adapting to the parabolic setting the approach of the second named author and F. H. Lin in \cite{GL1}, \cite{GL2}, based on the almost monotonicity of a generalized frequency function. In the present paper, we show that such method can be suitably extended to the nonlocal parabolic setting of \eqref{e0}. We consider a solution $u\in \operatorname{Dom}(H^s)$ to the equation \eqref{e0}, and denote by $U$ the corresponding solution 
of the extension problem \eqref{ext} below. Given such $U$, we introduce the height  $H(U,r)$ of $U$, its energy $I(U,r)$ and the \emph{frequency} of $U$,
\[
N(U,r) = \frac{I(U,r)}{H(U,r)},
\]
see Definition \ref{D:fre} below.
A second central result in this paper is the following. 

\begin{thrm}[Monotonicity of the adjusted frequency]\label{mon1}
Let $u\in  \operatorname{Dom}(H^{s})$ be a solution to \eqref{e0} with $V$ satisfying \eqref{vasump}. There exist universal constants $C, t_0>0$, depending only on $n, s$ and the number $K$ in \eqref{vasump}, such that with $r_0= \sqrt{t_0}$ and $a = 1-2s$, under the assumption that  $H(U,r) \neq 0$  for all $0< r \leq r_0$, then
\begin{equation}\label{mon}
r  \to \exp\left\{C\int_{0}^r t^{-a} dt\right\} \left(N(r)  +  C \int_{0}^{r} t^{-a} dt\right)
\end{equation}
is monotone increasing on $(0,r_0)$. Furthermore, when the potential $V\equiv 0$, then the constant $C$ in \eqref{mon} can be taken equal to zero and we have pure monotonicity of the function $r\to N(r)$. In such case, $N(r) \equiv \kappa$ for $0<r<R$ if and only if $U$ is homogeneous of degree $2\kappa$ in $\Sa_{R}^{+}$ with respect to the parabolic dilations \eqref{pardil}.
\end{thrm}

Theorem \ref{mon1} can be thought of as the nonlocal counterpart of the monotonicity formula first proved by Poon in \cite{Po} for the local case $s=1$. In connection with the local case we also refer to the  interesting works \cite{EF}, \cite{EFV}, where a space-like unique continuation  property for local solutions has been established for parabolic equations with sufficiently regular coefficients. 


For nonlocal elliptic equations with principal part $(-\Delta)^s$ a strong unique continuation theorem was obtained by Fall and Felli, see Theorem 1.3 in \cite{FF}. Their analysis combines the approach in \cite{GL1}, \cite{GL2} with the Caffarelli-Silvestre extension method in \cite{CSi}.  We also mention the interesting work of Ruland \cite{Ru}, \cite{Ru2}, where the Carleman method has been used, together with \cite{CSi}, to obtain results similar to those in \cite{FF} but with weaker assumptions on the potential $V$. Finally, we mention the recent paper \cite{Yu} where the case of nonlocal variable coefficient elliptic equations has been studied.

In our work,  similarly to the elliptic case in \cite{FF}, we combine the approach of Poon with the generalization for the heat equation of the Caffarelli-Silvestre extension method that  has been developed independently by Nystr\"om and Sande in  \cite{NS} and by Stinga and Torrea in \cite{ST}. We emphasize that the problem of space-time unique continuation backward in time is a global one, and of a somewhat  different nature compared to the unique continuation problem studied in \cite{FF} and \cite{Ru}, \cite{Ru2}.
We emphasize that the problem of space-time unique continuation backward in time is a global one, and of a somewhat  different nature compared to the unique continuation problem studied in \cite{FF} and \cite{Ru}, \cite{Ru2}. For parabolic  equations even in the case $s=1$ the space-time unique continuation is not true without certain global  assumptions on the solution as well as certain decay assumptions on the derivatives of the  principal  part. This follows from an example of F. Jones in \cite{Fr}, where a caloric function is constructed which  is supported in a space-time strip. We also refer to the interesting  paper \cite{WZ}, where it is shown that the decay assumption on  the derivatives of the principal part is somewhat  optimal. 

The present paper is organized  as follows:
\begin{itemize}
\item In Section \ref{S:fp} we include a brief discussion on the fractional heat operator $H^{s}$ and describe the pointwise formula \eqref{sH} for $H^{s}$ first found in \cite{SKM}. Here, Bochner's subordination principle plays a key role. Such principle was first used in a general framework by Balakrishnan in \cite{B}, where he defined the fractional powers $(-L)^s$ of a closed operator $L$ between Banach spaces, and it is also central to the works \cite{NS} and \cite{ST}.  
The reader should note that it  follows from the pointwise formula \eqref{sH} that the pseudo-differential operator $(\partial_t - \Delta)^s$ is a special case of the master equation introduced by Kenkre, Montroll and Shlesinger in \cite{KMS}.  Such equations are presently receiving increasing attention by mathematicians also thanks to the work of Caffarelli and Silvestre in \cite{CSi2} in which the authors proved H\"older continuity of viscosity solutions of generalized masters equations. We also mention the paper \cite{ACM}, where the obstacle problem for $H^s$ has been studied. 
\item In Section \ref{S:ep} we provide for completeness a detailed discussion of the extension problem for $(\p_t - \Delta)^s$ in \cite{NS} and \cite{ST}, see Theorem \ref{T:st} and Corollary \ref{C:st} below. The reason for doing this is that the main representation formulas in those papers are given by fiat, and do not explicitly mention the analysis on the Fourier transform side (in terms of Macdonald's functions) of the relevant Bessel process. On the other hand, such explicit tools are the main starting point of our analysis.

\item Starting from the representation formulas in these results, in Section \ref{S:nltol} we prove Lemma \ref{regU1} and Corollary \ref{C:ws} which allow us to convert the study of the strong unique continuation for the nonlocal equation \eqref{e00} in $\Rnn$ into a related problem for the local extension operator in $\Rnp\times \R$. 
\item Section \ref{S:dgn} plays a key role in our work. It is devoted to develop the basic regularity theory for the extension problem \eqref{ext} which is essential to the proof of Theorems \ref{main} and \ref{mon1}. Most part of the section is devoted to proving Theorem \ref{reg5} below, a result of De Giorgi-Nash-Moser type for the extension problem. Theorem \ref{reg5}, and the ensuing Lemmas \ref{reg2} and \ref{regU} are essential to our work in the remainder of the paper, but they are also interesting in their own right. For instance, we obtain a new scale invariant  Harnack inequality for the nonlocal equation $H^s u = V u + \psi$, see Theorem \ref{T:harnack} below. 
\item Section \ref{S:mono} is central to the whole paper. In it we establish the basic monotonicity property of the generalized frequency for the extension problem corresponding to \eqref{e0} in Theorem \ref{mon1} above. This result is the keystone to our analysis. As the reader will see, its proof is quite delicate and involved. 
\item In Section \ref{S:comp} we introduce the parabolic \emph{Almgren rescalings} (see Definition \ref{D:abups}) and,  similarly to \cite{FF}, \cite{Ru} and \cite{DGPT}, we perform a  blow-up  analysis of such rescalings which crucially rests on Theorem \ref{mon1}. The essential point is that these rescalings converge to a globally defined function $U_0$, that we call an \emph{Almgren blow-up}, which is homogeneous with respect to the non-isotropic parabolic scalings, see Proposition \ref{homg1} below. At the end of the section from the homogeneity of $U_0$, and the equation satisfied by it, we finally prove our main result Theorem \ref{main}. 
\item The paper ends with an Appendix, Section \ref{S:app}, in which we develop the regularity theory of the Almgren blowups which allows us to rigorously justify the results in Section \ref{S:comp}.
\end{itemize}

Some final remarks are in order:
\begin{itemize}
\item[1)] The reader should be aware that in the present work we have not  tried to push for maximal generality. Instead, we have chosen to keep a framework where the key ideas are properly highlighted. For instance, we are not concerned with whether the assumption on $V$  in \eqref{vasump} can be further weakened.  In this work we need \eqref{vasump} for the regularity result in Lemma \ref{reg2} which is essential for justifying  the computations in Section \ref{S:mono}. Having said this, it however  remains an interesting open problem whether in the supercritical case $0<s < \frac{1}{2}$ one can dispense with the assumption that $|<\nabla_{x} V, x>|$ be bounded.

\item[2)] We mention that in the case $V \equiv 0$ a Poon type monotonicity formula different from our Theorem \ref{mon1} has been formally derived in Theorem 1.15 in \cite{ST}. Our proof of Theorem \ref{mon1} however involves a substantial new delicate analysis of the regularity properties of the solution of the extension problem which is further complicated by the presence of the potential $V$. In addition, similarly to what was done in  \cite{DGPT} in the analysis of the parabolic Signorini problem in the case  $s=1/2$, several technical obstructions force us to work with certain averaged versions $H(U,r)$ and $I(U,r)$ of the functionals $h(U,t)$ and $i(U,t)$ in \eqref{h1} and \eqref{i} below. This allows for instance to get appropriate apriori estimates for the rescalings $U_r$'s as in Lemma \ref{convergence}. We then study monotonicity properties of such averaged  functionals as this is essential in Section \ref{S:comp}, where we adapt some ideas from \cite{DGPT} to establish uniform estimates in Gaussian space for the Almgren blow-up's $U_0$.

\item[3)] We finally recover our unique continuation property from these  uniform estimates of $U_0$. It turns out that, unlike the elliptic case in \cite{FF}, \cite{Ru} and \cite{Yu}, in our parabolic setting the strong unique continuation property does not follow directly from the homogeneity of the Almgren blow-up $U_0$ (see Proposition \ref{hom} below)  and from the information on its vanishing order at certain points (see Remark \ref{bl1} for  a detailed discussion on this aspect). This is caused by the fact that, due to the nature of the frequency, the ensuing estimates are only in  Gaussian space. In order to overcome this aspect, in addition to the homogeneity of $U_0$ we need to further utilize the equation \eqref{hom} satisfied by it. It turns out that this aspect is somewhat subtle and makes our proof quite  different from the elliptic case, where Gaussian spaces are not involved.   
\end{itemize} 

In closing, we mention that the work in this paper will also prove of interest e.g. in the study of certain basic open questions in free boundary problems, as well as in nonlocal parabolic segregation problems. We plan to come back to these aspects in future works. 
 
\medskip

\textbf{Acknowledgment:} We would like to thank the anonymous referee for suggestions and comments which helped improving the presentation of the paper.




\section{The fractional powers of $H = \p_t - \Delta$}\label{S:fp}

As already mentioned in Section \ref{S:Intro}, we  will  denote a typical point in $\R^{n} \times \R$ by $(x,t)$. 
Given a function $f\in L^1(\Rn)$, we denote by $\hat f$ its Fourier transform defined by
\[
\hat f(\xi) = \mathcal F_{x\to \xi}(f) = \int_{\Rn} e^{-2\pi i<\xi,x>} f(x) dx.
\]
We recall that if $h\in \Rn$, then $\mathcal F_{x\to \xi}(f(\cdot + h)) = e^{2\pi i<h,\xi>} \hat f(\xi)$. We will indicate with $\G(z) = \int_0^\infty t^{z-1} e^{-t} dt$, Euler's gamma function, and recall that for every $z\in \mathbb C$ such that $\Re z > 0$ we have
\begin{equation}\label{fact}
\G(z+1) = z \G(z).
\end{equation}

Let $H u = (\p_t  - \Delta) u = 0$ be the heat equation in $\R^{n+1}$. Its fundamental solution will be denoted by 
\[
G(x,t) = \begin{cases} (4\pi t)^{-\frac n2} e^{-\frac{|x|^2}{4t}}, \ \ \ \ \ t>0,
\\
0,\ \ \ \ \ \ \ \ \ \ \ \ \ \ \ \ \ \ \ \ \ \ t\le 0.
\end{cases}
\]
We recall that $\mathcal F_{x\to \xi}(G(\cdot,t)) = e^{-t(2\pi|\xi|)^2}$ for every $t>0$. The heat semigroup on $\Rn$ with generator $-\Delta$ will be denoted by 
\[
P_t f(x) = e^{-t \Delta} f(x) = G(\cdot,t) \star f(x) = \int_{\Rn} G(x-y,t) f(y) dy.
\]
Obviously, on the Fourier transform side we have
\[
\widehat{P_t f}(\xi) = e^{-t(2\pi|\xi|)^2} \hat f(\xi).
\]
We next define the semigroup $P^H_\tau = \{e^{-\tau H}\}_{\tau >0}$ on $\Rnn$ with generator $H$ by the formula
\begin{equation}\label{Hsg}
\widehat{P^H_\tau u}(\xi,\sigma) = \widehat{e^{-\tau H} u}(\xi,\sigma) = e^{-\tau(2\pi i \sigma + (2\pi |\xi|)^2)} \hat u(\xi,\sigma).
\end{equation}
We notice that by Plancherel's theorem we have for $f\in L^2(\R^{n+1})$ and any $\tau>0$ 
\begin{equation}\label{hl2}
||e^{-\tau H} f||_{L^2(\R^{n+1})} \le ||f||_{L^2(\R^{n+1})}.
\end{equation}
Furthermore, one has the following representation formula
\begin{equation}\label{Hsg2}
P_\tau^H u(x,t) = \left(G(\cdot,\tau) \star \Lambda_{-\tau} u\right)(x,t) = \int_{\Rn} G(x-z,\tau) u(t-\tau,z) dz,
\end{equation}
where we have let $\Lambda_h u(x,t) = u(x,t+h)$. One can easily check the identity between \eqref{Hsg} and \eqref{Hsg2} on the Fourier transform side. From \eqref{Hsg2} it is immediate to verify that for $f\in L^\infty(\Rnn)$ one has
\begin{equation}\label{linfty}
||e^{-\tau H} f||_{L^\infty(\R^{n+1})} \le ||f||_{L^\infty(\R^{n+1})}.
\end{equation}
 
With these preliminaries in place, we next introduce the definition of the fractional powers of the heat operator $H$. 

\begin{dfn}[Fractional heat operator]\label{D:hs}
Let $0<s<1$, then the nonlocal operator $H^s$ is defined on functions $u\in \mathcal S(\Rnn)$ by the formula 
\begin{equation}\label{sHft}
\widehat{H^s u}(\xi,\sigma) = ((2\pi|\xi|)^2 + 2\pi i \sigma)^s\  \hat u(\xi,\sigma).
\end{equation} 
\end{dfn}

Henceforth, given $\xi\in \Rn$, and $\sigma\in \R$, we denote by $L(\xi,\sigma)$ the complex number defined by the equation
\begin{equation}\label{lsigma}
L(\xi,\sigma)^2 \overset{def}{=} (2\pi |\xi|)^2 + 2\pi i \sigma,
\end{equation}
with the understanding that we have chosen the principal branch of the complex square root.
With this notation, we define
\begin{align}\label{dom}
\operatorname{Dom}(H^s) & = \{u\in L^2(\Rnn)\mid (\xi,\sigma) \to ((2\pi|\xi|)^2 + 2\pi i \sigma)^s\  \hat u(\xi,\sigma)\in L^2(\Rnn)\}
\\
& = \{u\in L^2(\Rnn)\mid (\xi,\sigma) \to L(\xi,\sigma)^{2s}  \hat u(\xi,\sigma)\in L^2(\Rnn)\}.
\notag
\end{align}
We will denote by $\mathcal H^{2s}_{\mathcal P}(\Rnn)$ the subspace $\operatorname{Dom}(H^s)\subset L^2(\Rnn)$ endowed with the norm
\begin{equation}\label{norm}
||u||_{\mathcal H^{2s}_{\mathcal P}(\Rnn)} = \left(\int_{\Rnn} (1+ |L(\xi,\sigma)|^2)^{2s}) |\hat{u}(\xi,\sigma)|^2 d\xi d\sigma\right)^{1/2} < \infty.
\end{equation}

The next formula provides a motivation for the approach in the paper by Stinga and Torrea \cite{ST} that is the basis for the present work. It is a classical fact that, if $L>0$, then for any $0<s<1$ one has
\begin{equation}\label{Ls}
- \frac{s}{\G(1-s)}\int_0^\infty \tau^{-s-1} (e^{-\tau L} - 1) d\tau  =  L^s.
\end{equation}
The proof of \eqref{Ls} follows easily by a change of variable and integration by parts.
Formula \eqref{Ls} is the keystone to Bochner's subordination principle used by A. V. Balakrishnan in a general framework, see formula (2.1) in \cite{B}. When applied to a closed operator $L$ on a Banach space of functions, Balakrishnan's result gives a way to define the fractional powers of $L$ as 
\begin{equation}\label{Ls2}
(L)^s f(x) = -  \frac{s}{\G(1-s)} \int_0^\infty \tau^{-s-1} \left[P_\tau f (x) - f(x)\right] d\tau,
\end{equation}
where $P_\tau = e^{-\tau L}$. For instance, one might take $L = - \Delta$ in $\Rn$, but one can allow for more general elliptic operators $L$. The idea in \cite{ST} is to use a suitable adaptation of Balakrishnan's formula \eqref{Ls2} to capture the fractional powers of $H$. 

The implementation of this idea is based on the observation that the formula \eqref{Ls} can be continued into $\mathbb C$ by substituting $L>0$ with $L\in \mathbb C$, as long as $\Re L >0$. Starting from this observation the following pointwise formula is proved in \cite{ST}.

\begin{thrm}\label{T:st1}
For any $0<s<1$, one has for $u\in\mathcal S(\Rnn)$
\begin{equation}\label{sH}
H^s u(x,t) = - \frac{s}{\G(1-s)} \int_0^\infty \tau^{-s-1} \left[P_\tau(\Lambda_{-\tau} u)(x,t) - u(x,t)\right] d\tau,
\end{equation}
where we have now denoted $\Lambda_h u(x,t) = u(x,t+h)$. 
\end{thrm}

\begin{proof}

To verify \eqref{sH} let us observe that, if $u\in \mathcal S(\Rnn)$, then 
\begin{align}\label{heatsg}
& \mathcal F_{(x,t)\to (\xi,\sigma)}(P_\tau(\Lambda_{-\tau} u)) = \mathcal F_{x\to \xi}(P_\tau(\mathcal F_{t\to\sigma}(\Lambda_{-\tau} u))) 
\\
& = \mathcal F_{x\to \xi}(P_\tau(e^{-2\pi i \sigma \tau} \mathcal F_{t\to\sigma}(u))) = e^{-\tau(2\pi|\xi|)^2} e^{-2\pi i \sigma \tau} \mathcal F_{(x,t)\to (\xi,\sigma)}(u)
\notag\\
& = e^{-\tau(2\pi i \sigma +(2\pi|\xi|)^2)} \mathcal F_{(x,t)\to (\xi,\sigma)}(u).
\notag
\end{align}
Therefore,  the Fourier transform of the right-hand side of \eqref{sH} is given by
\begin{align*}
& - \frac{s}{\G(1-s)} \int_0^\infty \tau^{-s-1} \mathcal F_{(x,t)\to (\xi,\sigma)} \left[P_\tau(\Lambda_{-\tau} u)(x,t) - u(x,t)\right] d\tau 
\\
& = - \frac{s}{\G(1-s)} \int_0^\infty \tau^{-s-1} \left[e^{-\tau (2\pi i \sigma + (2\pi|\xi|)^2)} - 1\right] d\tau \  \hat u(\xi,\sigma).
\end{align*}
If we now use \eqref{Ls} with $L = L(\xi,\sigma)$ given by \eqref{lsigma} (assuming that $\xi\not= 0$), we finally obtain 
\begin{align*}
& \mathcal F_{(x,t)\to (\xi,\sigma)} \left(- \frac{s}{\G(1-s)} \int_0^\infty \tau^{-s-1} \left[P_\tau(\Lambda_{-\tau} u)(x,t) - u(x,t)\right] d\tau\right)  
\\
& = ((2\pi|\xi|)^2 + 2\pi i \sigma)^s\  \hat u(\xi,\sigma).
\end{align*}
Keeping in mind \eqref{sHft}, we have established  \eqref{sH}.

\end{proof}

\begin{rmrk}\label{R:st}  
Although the right-hand side of \eqref{sHft} seems different from that of formula (1.2) in \cite{ST}, they are in fact identical. This is easily seen by observing that
\begin{align*}
& \int_0^\infty \tau^{-s-1} \left[P_\tau(\Lambda_{-\tau} u)(x,t) - u(x,t)\right] d\tau = 
\int_0^\infty \tau^{-s-1} \int_{\Rn} G(x-y,\tau) [u(y,t-\tau) - u(x,t)] dy d\tau
\\
& = \int_0^\infty  \int_{\Rn} \tau^{-s-1} G(z,\tau) [u(x-z,t-\tau) - u(x,t)] dz d\tau,
\end{align*}
which gives exactly the right-hand side of (1.2) in \cite{ST} if one uses \eqref{fact} above that gives $\G(1-s) = - s \G(-s)$ for $0<s<1$.
\end{rmrk}

\begin{rmrk}\label{R:st}  
Theorem 1.1 in \cite{ST} is stated not just for $u\in \mathcal S(\Rnn)$, but for functions in the parabolic H\"older space $C_{x,t}^{2s+\ve,s+\ve}(\Rnn)$, for some $\ve>0$. This is completely analogous to what happens in the time-independent case for $(-\Delta)^s$. It is easily checked that the right-hand side of \eqref{sH} continues to be finite not only on rapidly decreasing functions, but also when $u \in C_{x,t}^{2s+\ve,s+\ve}(\Rnn)$ for some $\ve>0$.  
\end{rmrk}


\section{The extension problem}\label{S:ep}

We recall that our main objective is establishing Theorem \ref{main}, i.e., the space-time strong unique continuation property for solutions of the nonlocal equation in $\R^{n+1}$
\begin{equation}\label{e00}
H^s u(x,t) = (\partial_t - \Delta)^{s} u(x,t) = \frac{\Gamma(1-s)}{2^{2s-1} \Gamma(s)}  V(x,t) u(x,t),
\end{equation}
where $u\in \operatorname{Dom}(H^s)$.
The reader will have noticed that we have written \eqref{e00} differently from \eqref{e0} in the introduction. We stress that the constant $\frac{\Gamma(1-s)}{2^{2s-1} \Gamma(s)}>0$ in the right-hand side of \eqref{e00} serves a purely normalization purpose motivated by \eqref{dtn} below, its presence being otherwise immaterial.

Due to the nonlocal nature of the operator $H^s$ in \eqref{e00} proving directly Theorem \ref{main} is a difficult task. On the other hand, in Theorem 1.3 in \cite{NS} and Theorem 1.7 in \cite{ST} the authors have generalized to the heat equation the well-known extension procedure of Caffarelli and Silvestre for the fractional powers of the Laplacian. Given this fact, like most works on nonlocal operators, and similarly to what was done in \cite{FF} and \cite{Ru} in the stationary case, our approach consists in using the extension procedure in \cite{NS}, \cite{ST} and thus work with a local operator. However, as we have already mentioned in Section \ref{S:Intro}, things are not so straightforward and we need to develop several auxiliary tools which also have an independent interest. This will be done in the subsequent sections.

The approach to the extension problem for the nonlocal operator $(\p_t - \Delta)^s$ in \cite{ST} is based on Bochner's subordination. As a help to the reader in this section we provide a complete account of the construction in Theorem 1.3 in \cite{NS} and Theorem 1.7 in \cite{ST} since this tool will be the starting point of our analysis. 

For $x\in \Rn$ and $y>0$, we will indicate with $X=(x,y)$ the corresponding point in $\Rnp$. Whenever convenient, points in $\Rnp\times \R$ will be indicated with $(X,t)$, instead of $(x,y,t)$.  
Given a solution $u$ of \eqref{e00}, we introduce the constant $a = 1-2s$ and we consider the following \emph{extension problem} for the function $U = U(X,t)$, where $(X,t) = (x,y,t)\in \Rnp\times\R$:
\begin{equation}\label{ep44}
\begin{cases}
y^a \frac{\p U}{\p t} = \operatorname{div}_{X}(y^a \nabla_{X} U),
\\
U(x,0,t) = u(x,t).
\end{cases}
\end{equation}
The equation in \eqref{ep44} is a special case of the following class of degenerate parabolic equations 
\[
\frac{\p (\omega(X) f)}{\p t} = \operatorname{div}_X(A(X) \nabla_X f)
\]
first studied by Chiarenza and Serapioni in \cite{CSe}. In that paper the authors assumed that $\omega\in L^1_{loc}(\Rnn)$ is a Muckenhoupt $A_2$-weight, and that the symmetric matrix-valued function $X\to A(X)$ verifies the following degenerate ellipticity assumption for a.e. $X\in \Om \subset \Rnn$, and for every $\xi\in \Rnn$:
\[
\la \omega(X) |\xi|^2 \le <A(X)\xi,\xi> \le \la^{-1} \omega(X) |\xi|^2,
\]
for some $\la >0$.
Under such hypothesis they established a parabolic strong Harnack inequality, and therefore the local H\"older continuity of the weak solutions. The extension equation in \eqref{ep44} is a special case of those treated in \cite{CSe} since, given that $a = 1-2s\in (-1,1)$, the function $\omega(X) = \omega(x,y) = y^a$ is an $A_2$-weight in $\Om = \Rnp$.  

The second order degenerate parabolic equation in \eqref{ep44} can also be written in nondivergence form in the following way
\begin{equation}\label{ep2}
\begin{cases}
\frac{\p U}{\p t} - \Delta_x U = \mathcal B_a U,\ \ \ \ (x,y,t)\in \Rnp\times\R,
\\
U(x,0,t) = u(x,t),\ \ \ \ \ \ \ \ \ \ \  \ \ \ \ \ \ (x,t) \in \Rnn,
\\
U(x,y,t) \to 0,\ \text{as}\ y \to \infty,\ \ \ \ \ \ \ \ (x,t)\in \Rnn,
\end{cases}
\end{equation}  
where we have denoted by $\mathcal B_a = \frac{\p^2}{\p y^2} + \frac ay \frac{\p }{\p y}$ the generator of the Bessel semigroup on $(\R^+,y^a dy)$.

In order to solve the problem \eqref{ep2} we assume momentarily that $u\in \mathcal S(\Rnn)$. We take the Fourier transform of \eqref{ep2} with respect to the variable $(x,t)$ and denote for a fixed $(\xi,\sigma)\in \Rnn$ 
\[
Y(y) \overset{def}{=} \hat U(\xi,y,\sigma) = \int_\R \int_{\Rn} e^{-2\pi i(\sigma t + <\xi,x>)} U(x,y,t) dx dt.
\]
In so doing,  for every fixed $(\xi,\sigma)\in \Rnn$, $\xi\not= 0$, problem \eqref{ep2} is converted into the following one on $\R^+$
\begin{equation}\label{ep3}
\begin{cases}
y^2 Y''(y) + a y Y'(y) - L(\xi,\sigma)^2 y^2 Y(y) = 0,\ \ \ \ y\in \R^+,
\\
Y(0) = \hat u(\xi,\sigma),
\\
Y(y) \to 0,\ \text{as}\ y \to \infty,
\end{cases}
\end{equation}  
where we have defined the complex number $L(\xi,\sigma)^2$ by the equation \eqref{lsigma} above. 
Comparing \eqref{ep3} with the generalized modified Bessel equation
\begin{equation}\label{genbessel}
y^2 u''(y) + (1 - 2\alpha) y u'(y) + \left[(\alpha^2 - \nu^2 \gamma^2) - \beta^2 \gamma^2
y^{2\gamma}\right] u(y) = 0,
\end{equation}
we see that we must have
\[
1-2\alpha = a = 1-2s,\ \ \ \gamma = 1,\ \ \ \beta = L(\xi,\sigma),\ \ \ \nu = \pm \alpha.
\]
This gives 
\[
\alpha = s,\ \ \ \ \gamma = 1,\ \ \ \ \beta = L(\xi,\sigma),\ \ \ \ \nu = \pm s.
\]
Therefore, the general solution of \eqref{ep3} is 
\[
Y(y) = A y^s I_s(L(\xi,\sigma) y) + B y^s K_s(L(\xi,\sigma)  y),
\]
where $I_\nu(z)$ and $K_\nu(z)$ respectively denote the modified Bessel function of the first kind and the Macdonald's function.
The condition $Y(y) \to 0$ as $y \to \infty$ forces $A = 0$ (see e.g. formulas (5.11.9) and (5.11.10) on p. 123 of \cite{Le} for the asymptotic behavior at $\infty$ of $I_\nu$ and $K_\nu$), and thus 
\begin{equation}\label{Y}
Y(y) = B y^s K_s(L(\xi,\sigma) y).
\end{equation}
Next, we use the condition $Y(0) = \hat u(\xi,\sigma)$ to determine the constant $B$. Recalling that 
\[
K_\nu(z) = \frac{\pi}{2} \frac{I_{-\nu}(z) - I_\nu(z)}{\sin \pi \nu},
\]
see (5.7.2) on p. 108 in \cite{Le}, and the asymptotics
\[
z^{\nu} I_{-\nu}(z)\ \cong\ \frac{2^{\nu}}{\Gamma(1-\nu)},\ \ \ \ z^{\nu} I_{\nu}(z)\cong 0, \quad\text{as }z\to 0,
\]
we see that as $y\to 0^+$ we have
\[
Y(y)  = B y^s K_s(L(\xi,\sigma) y) = B \frac \pi{2} \frac{y^s I_{-s}(L(\xi,\sigma) y) - y^s I_s(L(\xi,\sigma) y)}{\sin \pi s} \ \longrightarrow\ \frac{B 2^{s-1} \pi}{\G(1-s) \sin \pi s} L(\xi,\sigma)^{-s}.
\]
Using this asymptotic, along with the formula
\[
\G(s) \G(1-s) = \frac{\pi}{\sin \pi s},
\]
we find that as $y\to 0^+$,
\[
Y(y) \ \longrightarrow\  B 2^{s-1} \G(s) L(\xi,\sigma)^{-s},
\]
and the right-hand side equals $\hat u(\xi,\sigma)$ provided that
\[
B = \frac{L(\xi,\sigma)^{s} \hat u(\xi,\sigma)}{2^{s-1} \G(s)}.
\]
Returning to \eqref{Y} we conclude that the solution of \eqref{ep3} is given by
\begin{equation}\label{hatU}
\hat U(\xi,y,\sigma) = \frac{y^s}{2^{s-1} \G(s)} L(\xi,\sigma)^{s}  K_s(L(\xi,\sigma) y)\ \hat u(\xi,\sigma).
\end{equation}

Before proceeding we pause to note a fundamental fact about the solution $U(x,y,t)$ of \eqref{ep2}  identified by \eqref{hatU}: it satisfies the \emph{weighted Neumann condition}
\begin{equation}\label{dtn}
H^s u(x,t) = - \frac{2^{2s-1} \G(s)}{\G(1-s)} \underset{y\to 0^+}{\lim} y^a \frac{\p U}{\p y}(x,y,t).
\end{equation}
To see this notice that, in view of \eqref{sHft}, proving 
\eqref{dtn} is equivalent to showing 
\begin{equation}\label{ep4}
 - \frac{2^{2s-1} \G(s)}{\G(1-s)} \underset{y\to 0^+}{\lim} y^a \frac{\p \hat U}{\p y}(\xi,y,\sigma) = L(\xi,\sigma)^{2s}\  \hat u(\xi,\sigma).
\end{equation}
In the following computation we will write for brevity $L = L(\xi,\sigma)$. Keeping in mind that $a = 1-2s$, and using the formula
 \[
 K_s'(z) = \frac sz K_s(z) - K_{s+1}(z),
 \]
see (5.7.9) on p. 110 of \cite{Le}, we obtain from \eqref{hatU}
 \[
y^a \frac{\p \hat U}{\p y}(\xi,y,\sigma) = \frac{L^{s+1} \hat u(\xi,\sigma)}{2^{s-1} \G(s)} y^{1-s} \left\{\frac{2s}{L y} K_s(L y) - K_{s+1}(L y)\right\}.
\]
Since
\[
\frac{2s}{z} K_s(z) - K_{s+1}(z) = - K_{s-1}(z) = - K_{1-s}(z),
\]
 (again, by (5.7.9) on p. 110 of \cite{Le}) we finally have
\[
y^a \frac{\p \hat U}{\p y}(\xi,y,\sigma) = -  \frac{L^{s+1} \hat u(\xi,\sigma)}{2^{s-1} \G(s)} y^{1-s} K_{1-s}(L y).
\]
Now, as before, we have as $y\to 0^+$,
\[
y^{1-s} K_{1-s}(L y)\ \longrightarrow\ 2^{-s} \G(1-s) L^{s-1} .
\]
We finally reach the conclusion that as $y\to 0^+$,
\[
y^a \frac{\p \hat U}{\p y}(\xi,y,\sigma)\ \longrightarrow\ -  \frac{\G(1-s)}{2^{2s-1} \G(s)}  L^{2s} \hat u(\xi,\sigma).
\]
This proves \eqref{ep4}, and therefore \eqref{dtn}, when $u\in \mathcal S(\Rnn)$. However, the above considerations extend to functions $u\in \operatorname{Dom}(H^s)$ if interpreted in the sense of $L^2(\Rnn)$.

Returning to \eqref{hatU}, we finally see that the solution $U(x,y,t)$ of \eqref{ep} is given by
\begin{align}\label{fst}
U(x,y,t) & = \frac{y^s}{2^{s-1} \G(s)}  \mathcal F^{-1}_{(\xi,\sigma)\to (x,t)}\left[L(\xi,\sigma)^{s}  K_s(L(\xi,\sigma) y)\ \hat u(\xi,\sigma)\right] 
\\
& = \frac{y^s}{2^{s-1} \G(s)} \int_{\Rnn} e^{2\pi i (<x,\xi> + t \sigma)} L(\xi,\sigma)^{s} K_s(y L(\xi,\sigma)) \hat u(\xi,\sigma) d\xi d\sigma.
\notag
\end{align}
The key to computing the Fourier transform in the right-hand side of \eqref{fst} is the following important formula that can be found e.g. in 9. on p. 340 of \cite{GR}
\begin{equation}\label{f9GR}
\int_0^\infty \tau^{\nu-1} e^{-(\frac{\beta}{\tau} + \gamma \tau)} d\tau = 2 \left(\frac \beta{\gamma}\right)^{\frac \nu{2}} K_\nu(2 \sqrt{\beta \gamma}),
\end{equation}
provided $\Re \beta, \Re \gamma >0$. Applying \eqref{f9GR} with
\[
\nu = - s,\ \ \beta = \frac{y^2}{4},\  \ \gamma = L(\xi,\sigma)^2,
\]
and keeping in mind that $K_\nu = K_{-\nu}$ (see (5.7.10) in \cite{Le}), we find
\begin{equation}\label{fst2}
L(\xi,\sigma)^{s} K_s(y L(\xi,\sigma)) = \frac{y^s}{2^{s+1}} \int_0^\infty \tau^{-(1+s)} e^{-\frac{y^2}{4\tau}} e^{-L(\xi,\sigma)^2 \tau} d\tau.
\end{equation}
Substituting \eqref{fst2} into \eqref{fst}, and exchanging the order of integration, we finally obtain
\begin{align}\label{fst3}
U(x,y,t) & = \frac{y^s}{2^{s-1} \G(s)}  \mathcal F^{-1}_{(\xi,\sigma)\to (x,t)}\left[L(\xi,\sigma)^{s}  K_s(L(\xi,\sigma) y)\ \hat u(\xi,\sigma)\right]  
\\
& = \frac{y^{2s}}{2^{2s} \G(s)} \int_{\Rnn} e^{2\pi i (<x,\xi> + t \sigma)} \left(\int_0^\infty \tau^{-(1+s)} e^{-\frac{y^2}{4\tau}} e^{-L(\xi,\sigma)^2 \tau} d\tau\right) \hat u(\xi,\sigma) d\xi d\sigma
\notag
\\
& = \frac{y^{2s}}{2^{2s} \G(s)} \int_0^\infty \tau^{-(1+s)} e^{-\frac{y^2}{4\tau}} \left(\int_{\Rnn} e^{2\pi i (<x,\xi> + t \sigma)} e^{-L(\xi,\sigma)^2 \tau} \hat u(\xi,\sigma)  d\xi d\sigma\right) d\tau
\notag\\ 
& = \frac{y^{2s}}{2^{2s} \G(s)} \int_0^\infty \tau^{-(1+s)} e^{-\frac{y^2}{4\tau}} e^{-\tau H} u(x,t) d\tau,
\notag
\end{align}
where in the last equality we have used the definition \eqref{Hsg} of the semigroup $e^{-\tau H}$. The above calculations are performed under the assumption that $u\in \mathcal S(\Rnn)$. They can be extended in a standard fashion to functions $u\in \operatorname{Dom}(H^s)$.

We have thus proved the following result which represents the first part of formula (1.5) plus formula (1.6) in Theorem 1.7 in \cite{ST}, see also Theorem 1 in \cite{NS}.

\begin{thrm}\label{T:st} Let $0<s<1$ and assume that $u\in \operatorname{Dom}(H^s)$. For any $(x,t)\in \Rnn$ and $y>0$ the function
\begin{equation}\label{UU}
U(x,y,t) = \frac{y^{2s}}{2^{2s} \G(s)} \int_0^\infty \tau^{-(1+s)} e^{-\frac{y^2}{4\tau}} e^{-\tau H} u(x,t) d\tau
\end{equation}
solves the extension problem \eqref{ep2} above, with the boundary condition $U(x,0,t) = u(x,t)$ understood in the sense of $L^2(\Rnn)$. Furthermore, one has in $L^2(\Rnn)$
\begin{equation}\label{dtn2}
- \frac{2^{2s-1} \G(s)}{\G(1-s)} \underset{y\to 0^+}{\lim} y^a \frac{\p U}{\p y}(x,y,t) = H^s u(x,t).
\end{equation}
\end{thrm}

Once Theorem \ref{T:st} is established it is easy to prove the following corollary that represents the remaining part of (1.5) in Theorem 1.7 in \cite{ST}.

\begin{cor}\label{C:st}
Let $u$ and $U$ be as in Theorem \ref{T:st}. Then, one has
\begin{equation}\label{stbis}
U(x,y,t) = \frac{1}{\G(s)} \int_0^\infty \tau^{-(1-s)} e^{-\frac{y^2}{4\tau}} e^{-\tau H} (H^s u)(x,t) d\tau.
\end{equation}
Moreover, one has the following Poisson representation formula 
\begin{equation}\label{pos}
U(x,y,t) = \int_{0}^{\infty} \int_{\R^{n}} P^{s}_{y}(z,\tau) u(x-z,t- \tau) dz d\tau,
  \end{equation}
  where
\begin{equation}\label{poisson}
P^{s}_{y}(z,\tau) = \frac{1}{4^{\frac n2 + s} \pi^{\frac n2} \Gamma(s)}  \frac{y^{2s}}{\tau^{n/2 + 1+ s}} e^{-\frac{|z|^2 + y^2}{4 \tau}}.
  \end{equation}
\end{cor}


\section{From nonlocal to local}\label{S:nltol}

In this section we use Theorem \ref{T:st} and Corollary \ref{C:st} to convert the study of the strong unique continuation property for the nonlocal equation \eqref{e00} in $\Rnn$ into a related problem for the local extension operator in  $\Rnp\times \R$. Precisely, for a given $0<s<1$, and with $a = 1-2s$, we assume that $u\in \operatorname{Dom}(H^s)$, and consider 
the following \emph{extension problem}: 
\begin{equation}\label{ext}
\begin{cases}
y^{a} \partial_{t} U(X,t) = \operatorname{div}_{X}(y^{a} \nabla_{X}U)(X,t),
\\
U(x, 0,t)= u(x,t),
\\
\underset{y\to 0^+}{\lim} y^{a} \frac{\p U}{\p y}(x,y,t) = - V(x,t) u(x,t),
\end{cases}
\end{equation}
for $(X,t) \in \Rnp\times \R$.
We  note that according to \eqref{dtn2} in Theorem \ref{T:st}, the third equation in \eqref{ext} must be interpreted in the sense of $L^{2}(\Rnn)$. We also note that the fact that the constant in front of $V$ is $-1$ is the reason for which we introduced the normalization constant $\frac{\Gamma(1-s)}{2^{2s-1} \Gamma(s)}$ in \eqref{e00}. 

In order to work with the problem \eqref{ext} we need to specify the notion of weak solution. For a given $r>0$, and $X_0 = (x_0,y_0)\in \Rnn$ we indicate with $\B_r(X_0) = \{X = (x,y)\in \Rnn\mid |X-X_0|^2 = |x-x_0|^2 + (y-y_0)^2 < r^2\}$ the ball centered at $X_0$ with radius $r$ in the thick space, and use the notation $B_r(x_0) = \{x\in \Rn\mid |x-x_0|<r\}$. When $X_0 = (0,0)\in \Rnn$ we simply write $\B_r$ and $B_r$, instead of $\B_r(0), B_r(0)$. We also set  $\B_r^{+} = \B_r \cap \{y > 0\}$. Furthermore, for a given function $f = f(X)$, we will denote $f_i$ the partial derivative $\frac{\p f}{\p x_i}$ for $i=1,...n$, and use the standard notation $f_y$ for $\frac{\p f}{\p y}$. Henceforth, unless we specify otherwise, when we write $\nabla$ and div we intend that these operators act with respect to the variable $X = (x,y)\in \Rnn$. For instance, following this agreement, for $\la>0$ we denote by
\begin{equation}\label{pardil}
\delta_\la(X,t) = (\la X,\la^2 t)
\end{equation}
the parabolic dilations in $\Rnn\times \R$, and by
\begin{equation}\label{z1}
Zf= <X,\nabla f> + 2tf_t = <x,\nabla_x f> + y f_y + 2tf_t
\end{equation}
the generator of the group $\{\delta_\la\}_{\la>0}$. For later use we notice that for every $(X,t)$ such that $t\not= 0$, \eqref{z1} can be rewritten
\begin{equation}\label{z11}
\frac{Zf}{2t} = f_t + <\nabla f,\frac{X}{2t}>.
\end{equation}
One easily recognizes that a $C^1$ function $f:\Rnn\times \R\to \R$ is homogeneous of degree $\kappa\in \R$ with respect to \eqref{pardil}, i.e., $f\circ \delta_\la = \la^\kappa f$, if and only if one has 
\begin{equation}\label{euler}
Zf(X,t) = \kappa f(X,t).
\end{equation}

\begin{dfn}\label{d1}
For given numbers $a\in (-1,1), r>0$ and $0<T_1<T_2$ we define the space
\[
V^{a, r, T_1, T_2} = L^2((T_1, T_2);W^{1,2}(\B_r^{+},y^a dX)),
\]
endowed with the norm
\begin{equation}
||w||^2_{V^{a, r, T_1, T_2}} = \int_{T_1}^{T_2} \left(\int_{\B_r^{+}} (|w|^2+  |\nabla w|^2 )y^a dX\right) dt < \infty.
\end{equation}
\end{dfn}

\begin{rmrk}
When the context is clear, we will simply write $V^{a}$ instead of $V^{a, r, T_1, T_2}$. 
\end{rmrk}

We now introduce the relevant notion of weak solution to \eqref{ext}, but will allow a slightly more general right-hand side in the Neumann condition.

\begin{dfn}\label{d2}
Given $W, \psi \in L^{\infty}(B_r \times (T_1, T_2))$, a function $v \in V^{a, r, T_1, T_2}$ is said to be a weak solution in $\B_r^{+} \times (T_1, T_2)$ to
 \begin{equation}\label{ep} 
\begin{cases}
\operatorname{div}(y^a \nabla v) = y^a v_t,
 \\
\underset{y \to 0}{\lim}\ y^a v_y = W v +\psi,
\end{cases}
\end{equation}
 if for every $\phi \in W^{1,2} ( \B_r^{+} \times (T_1, T_2), y^a dXdt)$ with compact  support in $(\B_r^{+} \cup B_r)  \times [T_1, T_2]$, we have
 \begin{align}\label{d8}
& \int_{t_1}^{t_2} \left(\int_{\B_r^{+}}  <\nabla v, \nabla \phi> y^a dX\right) dt= \int_{t_1}^{t_2} \left(\int_{\B_r^{+}} y^a \phi_t v dX\right) dt - \int_{\B_r^{+}} \phi (\cdot,t_2) v (\cdot, t_2) y^a dX
\\
&+ \int_{\B_r^{+}} \phi (\cdot,t_1) v (\cdot,t_1) y^a  dX - \int_{t_1}^{t_2} \int_{B_r}  (Wv + \psi) \phi dx dt
\notag
 \end{align}
 for almost every $t_1 , t_2$ such that $T_1 <t_1 <t_2<T_2$. 
 \end{dfn}
 \begin{rmrk}
 We note that in \eqref{d8} the boundary integral at $y=0$, $\int_{t_1}^{t_2} \int_{B_r}  (Wv + \psi) \phi dx dt$, is to be interpreted in the sense of traces. We refer to \cite{Ne} for traces of weighted Sobolev spaces.
 \end{rmrk}

We now return to the extension problem \eqref{ext} and establish two basic regularity estimates for its solution $U$. 

\begin{lemma}\label{regU1}
Let $u \in \operatorname{Dom}(H^s)$ and $U$ be as in \eqref{UU}. Then, for any $M>0$ one has 
\begin{equation}\label{finite1}
\int_{0}^{M} \left(\int_{\Rnn} y^a   |U|^2 dx dt\right) dy \le \frac{M^{a+1}}{a+1} < \infty.
\end{equation}
We also have for some universal constant $C_s>0$
\begin{equation}\label{finite}
\int_{-\infty}^{\infty} \left(\int_{\Rnn_{+}} y^a   |\nabla U|^2 dX\right) dt \le C_s \left(||u||_{L^2(\Rnn)} + ||H^s u||_{L^2(\Rnn)}\right)  <\infty,
\end{equation}
the right-hand side being finite by the hypothesis $u\in \operatorname{Dom}(H^s)$, see \eqref{norm} above. 
\end{lemma}

\begin{proof}
We note that  from the representation \eqref{UU} of $U$, and from \eqref{hl2}, we have 
\begin{align*}
||U(\cdot,y,\cdot)||_{L^{2}(\Rnn)} & \leq \frac{y^{2s}}{2^{2s} \G(s)} \int_0^\infty \tau^{-(1+s)} e^{-\frac{y^2}{4\tau}}  ||e^{-\tau H} u||_{L^{2}(\Rnn)} d\tau
\\
& \le ||u||_{L^{2}(\Rnn)} \frac{y^{2s}}{2^{2s} \G(s)} \int_0^\infty \tau^{-(1+s)} e^{-\frac{y^2}{4\tau}}  d\tau.
\end{align*}
A simple computation now gives
\[
\int_0^\infty \tau^{-(1+s)} e^{-\frac{y^2}{4\tau}}  d\tau = \frac{2^{2s} \G(s)}{y^{2s}}.
\]
Using this information in the previous inequality, we conclude
\begin{equation}\label{U22}
||U(\cdot,y,\cdot)||_{L^{2}(\Rnn)} \le ||u||_{L^{2}(\Rnn)}.
\end{equation}
The inequality \eqref{finite1} follows from \eqref{U22} in a standard way keeping in mind that $a>-1$. 

As for \eqref{finite}, we have from \eqref{hatU} 
\[
\hat{U}(\xi, y, \sigma)= \frac{1}{2^{s-1} \G(s)} \Phi_s( L(\xi, \sigma)y) \hat u(\xi, \sigma) = C(s) \Phi_s( L(\xi, \sigma)y) \hat u(\xi, \sigma),
\]
where we have let
\[
\Phi_\nu(z)=z^\nu K_\nu(z).
\]
Recall that (5.7.9) in \cite{Le} gives
\begin{equation}\label{derP}
\Phi_\nu'(z) = - z^\nu K_{\nu -1}(z) = - z^\nu K_{1-\nu}(z).
\end{equation}
Therefore, Plancherel's theorem implies
\begin{align*}
& \int_{-\infty}^{\infty} \int_{\Rnn_{+}} y^a |\nabla U|^2 dX dt \le 
C(s)^2 \int_{0}^{\infty} y^{a} \left(\int_{\Rnn} (2\pi |\xi|)^2 |\Phi_s( L(\xi, \sigma)y)|^2 |\hat{u}(\xi,\sigma)|^2  d\xi d\sigma\right) dy
\\
& + C(s)^2 \int_{0}^{\infty} y^{a} \left(\int_{\Rnn} |L(\xi,\sigma)|^2 |\Phi_s'(L(\xi, \sigma)y)|^2 |\hat{u}(\xi,\sigma)|^2  d\xi d\sigma\right) dy
\\
& = C(s)^2 \int_{\Rnn} |L(\xi,\sigma)|^{2s} |\hat{u}(\xi,\sigma)|^2 \bigg(\int_{0}^{\infty} y^{a} |L(\xi,\sigma)|^{-2s}\bigg[(2\pi |\xi|)^2 |L(\xi, \sigma)|^{2s} y^{2s} |K_s( L(\xi, \sigma)y)|^2
\\
&  + |L(\xi,\sigma)|^2 |L(\xi, \sigma)|^{2s} y^{2s}|K_{1-s}(L(\xi, \sigma)y)|^2 \bigg] dy\bigg)  d\xi d\sigma,
\end{align*}
where in the last equality we have used \eqref{derP}. Keeping in mind that $a = 1-2s$, and that $(2\pi |\xi|)^2 \le |L(\xi,\sigma)|^2$, we conclude that
\begin{align*}
& \int_{-\infty}^{\infty} \int_{\Rnn_{+}} y^a |\nabla U|^2 dX dt \le 
C(s)^2 \int_{\Rnn} |L(\xi,\sigma)|^{2s} |\hat{u}(\xi,\sigma)|^2 \bigg(\int_{0}^{\infty} y \bigg[(2\pi |\xi|)^2 |K_s( L(\xi, \sigma)y)|^2
\\
&  + |L(\xi,\sigma)|^2 |K_{1-s}(L(\xi, \sigma)y)|^2 \bigg] dy\bigg)  d\xi d\sigma
\\
& \le 
C(s)^2 \int_{\Rnn} |L(\xi,\sigma)|^{2s} |\hat{u}(\xi,\sigma)|^2 \bigg(\int_{0}^{\infty} |L(\xi,\sigma)|^2 y \bigg[|K_s( L(\xi, \sigma)y)|^2
\\
&  +  |K_{1-s}(L(\xi, \sigma)y)|^2 \bigg] dy\bigg)  d\xi d\sigma.
\end{align*}
Next, for every $\xi\in \Rn\setminus\{0\}$, we write 
\begin{align*}
& \int_{0}^{\infty} |L(\xi,\sigma)|^2 y \bigg[|K_s( L(\xi, \sigma)y)|^2  + |K_{1-s}(L(\xi, \sigma)y)|^2 \bigg] dy
\\
& = \int_{0}^{|L(\xi,\sigma)|^{-1}} |L(\xi,\sigma)|^2 y \bigg[|K_s( L(\xi, \sigma)y)|^2
  + |K_{1-s}(L(\xi, \sigma)y)|^2 \bigg] dy
\\
& + \int_{|L(\xi,\sigma)|^{-1}}^\infty |L(\xi,\sigma)|^2 y \bigg[|K_s( L(\xi, \sigma)y)|^2
  + |K_{1-s}(L(\xi, \sigma)y)|^2 \bigg] dy
\end{align*}
Now, on the interval $0\le y \le |L(\xi,\sigma)|^{-1}$ we use the asymptotics
\begin{equation}
|K_s( L(\xi, \sigma)y)|^2 = O(|L(\xi, \sigma)y)|^{-2s}),\ \ \ \ |K_{1-s}( L(\xi, \sigma)y)|^2 = O(|L(\xi, \sigma)y)|^{2s-2}),
\end{equation}
to infer that for some universal $C_s'>0$
\[
\int_{0}^{|L(\xi,\sigma)|^{-1}} |L(\xi,\sigma)|^2 y \bigg[|K_s(L(\xi, \sigma)y)|^2
  +  |K_{1-s}(L(\xi, \sigma)y)|^2 \bigg] dy \le C_s'.
  \]
Since the argument of the complex number $y L(\xi,\sigma)$ ranges between $-\frac \pi{4}$ and $\frac \pi{4}$, on the interval $|L(\xi,\sigma)|^{-1} \le y < \infty$ we can use the asymptotic in (5.11.9) in \cite{Le} that gives
\begin{equation}
\begin{cases}
|K_s(L(\xi, \sigma)y)|^2 = O(|L(\xi, \sigma)y)|^{-1} ) e^{-y|L(\xi, \sigma)|},
\\
|K_{1-s}(L(\xi, \sigma)y)|^2 = O(|L(\xi, \sigma)y)|^{-1} ) e^{-y|L(\xi, \sigma)|},
\end{cases}
\end{equation}
This allows to infer that for some universal $C''_s>0$
\[
\int_{|L(\xi,\sigma)|^{-1}}^\infty |L(\xi,\sigma)|^2 y \bigg[|K_s( L(\xi, \sigma)y)|^2
  + |K_{1-s}(L(\xi, \sigma)y)|^2 \bigg] dy \le C''_s.
  \]
In conclusion, we have proved that there exists a universal constant $C_s>0$ such that
\begin{equation}\label{ds}
\int_{-\infty}^{\infty} \int_{\Rnn_{+}} y^a |\nabla U|^2 dX dt  \le 
C_s \int_{\Rnn} |L(\xi,\sigma)|^{2s} |\hat{u}(\xi,\sigma)|^2 d\xi d\sigma.
\end{equation}
Once \eqref{ds} is established, we have
\begin{align*}
& \int_{\Rnn} |L(\xi,\sigma)|^{2s} |\hat{u}(\xi,\sigma)|^2 d\xi d\sigma \le \int_{|L(\xi,\sigma)|\le 1} |L(\xi,\sigma)|^{2s} |\hat{u}(\xi,\sigma)|^2 d\xi d\sigma
\\
& + \int_{|L(\xi,\sigma)|\ge 1} |L(\xi,\sigma)|^{2s} |\hat{u}(\xi,\sigma)|^2 d\xi d\sigma
\\
& \le \int_{\Rnn} |\hat{u}(\xi,\sigma)|^2 d\xi d\sigma + \int_{|L(\xi,\sigma)|\ge 1} |L(\xi,\sigma)|^{4s} |\hat{u}(\xi,\sigma)|^2 d\xi d\sigma
\\
& \le ||u||_{L^2(\Rnn)}  + \int_{\Rnn} |L(\xi,\sigma)|^{4s} |\hat{u}(\xi,\sigma)|^2 d\xi d\sigma  
\\
& = ||u||_{L^2(\Rnn)} + ||H^s u||_{L^2(\Rnn)} < \infty,
\end{align*}
where in  the last inequality we have used the fact that $u \in \operatorname{Dom}(H^s)$. 
We conclude that  
\[
\int_{-\infty}^{\infty} \int_{\Rnn_{+}} y^a |\nabla U|^2 dX dt  \leq C_s \left(||u||_{L^2(\Rnn)} + ||H^s u||_{L^2(\Rnn)}\right)  <\infty,
\]
which proves \eqref{finite}. 

\end{proof}

Lemma \ref{regU1} establishes the following important fact.
 
\begin{cor}\label{C:ws}
Given arbitrary $r>0$ and $T_1<T_2$, the function $U$ in \eqref{UU} is a weak solution to \eqref{ext} in $\B_r^{+} \times (T_1, T_2)$.  
 \end{cor}
 
Before  proceeding further we make the following remark.
    
\begin{rmrk}
Corollary \ref{C:ws} will be used in Theorem \ref{reg5} and Lemma \ref{regU} below to establish the  higher regularity of $U$. The latter, in turn, will be crucially used to justify the computations in Section \ref{S:mono}. 
\end{rmrk}

We close this section by recalling a trace  inequality that will be used repeatedly in this paper. Its proof can be found on p. 65 in \cite{Ru}. We emphasize that in the present work we will apply such inequality at every time level $t$ in the relevant domain of integration. In the next statement, given a function $f = f(X)$, where $X = (x,y)\in \Rnp$, by abuse of notation we will denote the trace $f(x,0)$ of $f$ on $\Rn\times \{0\}$ by $f$ itself. 

\begin{lemma}[Trace inequality]\label{tr}
Let   $f\in C_0^\infty(\R^{n+1}_+)$. There exists a constant $C_0 = C_0(n,s)>0$ such that for every $\mu>0$ one has
\begin{equation*}
||f||_{L^2(\Rn \times \{0\})} \leq  C_0 \left(\mu^{1-s} ||y^{\frac{1-2s}{2}} f||_{L^2(\Rnp)} + \mu^{-s} ||y^{\frac{1-2s}{2}} \nabla f||_{L^2(\Rnp)}\right).
\end{equation*}
\end{lemma}

We also  need a  ``surface" version of Lemma \ref{tr} that can be found in Lemma 3.1  in  \cite{Ru2}. Before stating it we fix some intermediate notations. We will indicate with $\mathbb S^{n}$ the $n$-dimensional sphere in $\Rnn$ and by $\mathbb S^{n-1}$ that in $\Rn$. Also, we set  $\mathbb S^{n}_{+} = \mathbb S^{n} \cap \{y> 0\}$.  We  also denote an arbitrary point in $\mathbb S^{n}$ by $\omega= (\omega_1,...,\omega_n, \omega_{n+1})$ and one in $\mathbb S^{n-1}$ by $\omega'=(\omega_1,...,\omega_n)$. The surface measure on $\mathbb S^{n}$ will be denoted by $d\omega$ and that on $\mathbb S^{n-1}$ by $d\omega'$.

\begin{lemma}[Surface trace inequality] \label{tr5}
Let $g:\mathbb S^{n+1}_{+} \to \R$ be a measurable function  such that $g, \nabla_{\mathbb S^{n}} g \in L^{2}(\mathbb S^{n}_{+}, \omega_{n+1}^a d\omega)$, where $\nabla_{\mathbb S^{n}} g$ denotes the Riemannian gradient of the function $g$ with respect to the induced metric on $\mathbb S^{n}$. Then, there exists $C = C(n,s)>0$ such that  for all $\tau>1$
\begin{equation*}
 ||g||_{L^{2}(\mathbb S^{n-1})} \leq C \left (\tau^{1-s} ||\omega_{n+1}^{\frac{1-2s}{2}} g||_{L^{2}(\mathbb S^{n}_{+})} + \tau^{-s} ||\omega_{n+1}^{\frac{1-2s}{2}}\nabla_{\mathbb S^{n}} g||_{L^{2}(\mathbb S^{n}_{+})} \right).
 \end{equation*}
 \end{lemma}


\section{De Giorgi-Nash-Moser theory for the extension problem}\label{S:dgn}

This section is devoted to developing the regularity theory for the extension problem \eqref{ext}  which is essential to the proof of Theorems \ref{main} and \ref{mon1}.
The central result is Theorem \ref{reg5} below. The latter is a basic H\"older continuity theorem of De Giorgi-Nash-Moser type for the extension problem. For the elliptic counterpart of such result one should see \cite{CS}, but also \cite{JLX}, \cite{TX}. Henceforth, for a given domain $\Om \subset \Rnn$, and $0<\alpha\le 1$, we denote by $\Ha^{\alpha}(\Om)$ the non-isotropic parabolic H\"older space with exponent $\alpha$ defined on p. 46 in Chapter 4 in \cite{Li} (these spaces are denoted by $\Ha^{\alpha,\alpha/2}(\Om)$ in \cite{LSU}). 

\medskip

\noindent \textbf{Note:} Henceforth in this section, as well as in the rest of the paper, unless there is risk of confusion, we will routinely avoid writing explicitly $dX, dx, dt, dr$, etc., in the integrals involved.

\medskip

\begin{thrm}[of De Giorgi-Nash-Moser type]\label{reg5}
Let $W, \psi \in L^{\infty}(B_1 \times (-1,0))$, and $v \in V^a = V^{a, 1,-1,0}$ be a weak solution in $\B_1^{+} \times (-1, 0]$ to the problem \eqref{ep}.
Then, for any $r< 1$ one has $v \in \Ha^{\alpha}(\B_{r}^{+} \times (-r^2, 0])$ for some $\alpha\in (0,1)$ depending only on $a, n$. 
 \end{thrm}

\begin{proof}
Before starting, we make an important remark. In the course of the proof we will be using in a substantial way ideas and/or  results  from the following papers: \cite{M}, \cite{M1}, \cite{FKS} and \cite{CSe}. For obvious considerations of space, we will not be able to provide detailed accounts of these works. 

It suffices to show that  $v \in \Ha^{\alpha}(\B_{r}^{+} \times (-r^2, 0])$ for $0<r<\frac{1}{2}$. The conclusion then follows for any $r<1$ by a standard  covering argument. For a given function $f$ and $h>0$ we introduce the Steklov  averages of $f$ defined by
\begin{equation*}\label{stek}
f_{h}(X, t) = \frac{1}{h} \int_{t}^{t+h} f(X, s) ds,
\end{equation*}
and
\begin{equation*}\label{stek1}
f_{-h}(X,t) = \frac{1}{h} \int_{t-h}^t f(X,s) ds,
\end{equation*}
see (4.4) on p. 85 in \cite{LSU}, or also p. 100 in \cite{Li}.
Then, for every $\delta>0$ we have that $f_{-h} \longrightarrow f$, $\nabla f_{-h} \longrightarrow \nabla f$ in $L^{2}(\B_1^+ \times (-1+\delta,0), |y|^a dX dt)$ provided $f$ belongs to $V^{a}$ (see Lemmas 4.7 and 4.8 in \cite{LSU}). 
 
Let $\tau$ be a cut-off function which is compactly supported in $\B_{2r} \times (-4r^2, 4r^2)$, with $\tau \equiv 1$ in $\B_{r} \times [-r^2,  r^2]$, $\tau \equiv 0$ outside $\B_{2r} \times (-\frac{3}{2} r^2, \frac{3}{2} r^2)$, and such that 
 \[
|\nabla \tau(X,t)| \leq \frac{C}{r},\ \ \ \ \ \ |\tau_t(X,t)| \leq \frac{C}{r^2}
 \]
 for some absolute constant $C>0$.
Fix $0<\rho<r^2$ and define $\xi(t)$ by
 \begin{equation*}
 \begin{cases}
 \xi(t)=0,\ t>-\rho
 \\
 \xi(t)=1,\ t\le -\rho.
 \end{cases}
 \end{equation*}
Finally, we let $\eta(X,t)= \tau^2(X,t) v_{-h}(X,t) \xi(t)$ for $|h| < \rho/2$, and take $\eta_{h}$ as a test function in the weak formulation  \eqref{d8} above. The fact that $\eta_{h}$ is an admissible test function follows by approximation of $\xi$ by piecewise  linear functions and is quite standard in the parabolic theory. For instance, we  refer to p. 103 in \cite{Li} for relevant details. Using this test function in \eqref{d8} we obtain
\begin{align*}
 & \int_{ \B_{2r}^{+} \times (-4r^2, -\rho)} y^a |\nabla v_{-h}|^2 \tau^2 + \int_{ \B_{2r}^{+} \times (-4r^2, -\rho)} y^a <\nabla v_{-h}, \nabla \tau> 2\tau v_{-h}  
 \\
 & = -\int_{ \B_{2r}^{+} \times (-4r^2, -\rho)} y^a \tau^2 v_{-h}(v_{-h})_t   + \int_{ B_{2r} \times (-4r^2, -\rho)} (W v)_{-h} v_{-h} \tau^2 + (\psi)_{-h} v_{-h} \tau^2.
 \notag
  \end{align*}
Keeping in mind that  $\tau \equiv 0$ at $t= -2r^2$, an integration by parts with respect to $t$ gives
\begin{align}\label{t}
& \int_{ \B_{2r}^{+} \times (-4r^2, -\rho)} y^a \tau^2 v_{-h}(v_{-h})_t =   \frac{1}{2} \int_{\B_{2r}^{+}} y^a \tau^2 v_{-h}(X,\rho)^2 dX
\\
&  - \frac{1}{2} \int_{ \B_{2r}^{+} \times (-4r^2, -\rho)} y^a (\tau^2)_t (v_{-h})^2. 
\notag
\end{align}
Applying Lemma \ref{tr} at every fixed time level $t\in (-4r^2, -\rho)$ to the integral 
  \[
  \int_{ B_{2r} \times (-4r^2, -\rho)} (W v)_{-h} v_{-h} \tau^2 + (\psi)_{-h} v_{-h} \tau^2, 
\]
with $\mu>0$ such that
   \[
   \mu^{-s} (||W||_{L^{\infty}(B_1 \times (-1, 0)} + ||\psi||_{L^{\infty}(B_1 \times (-1, 0)}) < 1/2,
   \]
then letting $h \to 0$ and finally choosing   $\rho$  in \eqref{t} such that
    \[
   \int_{\B_{r}^{+}} y^a  v(X,\rho)^2 dX \geq \frac{1}{2}  \underset{t \in (-r^2,0]}{\operatorname{essup}} \int_{\B_r^{+}}      y^a  v(X, t)^2 dX,    
   \]
we conclude in a standard way that the following Caccioppoli type inequality holds for some universal $C>0$ which depends on the $L^{\infty}$ norm of $W, \psi$ in $\B^+_{2r}\times(-4r^2,0]$,
  \begin{align*}
&  \int_{\B_{r}^{+} \times (-r^2, 0]}  |\nabla v|^2 y^a dX dt + \underset{t \in (-r^2,0]} {\operatorname{essup}}\ \int_{\B_{r}^{+}} v^2 y^a  dX   
\\
&\leq \frac{C}{r^2} \int_{\B_{2r}^{+} \times (-4r^2, 0]}  |v|^2 y^a dX dt + C r^{n+2}.
\notag
  \end{align*}
We note that in the latter inequality we have also used that $\tau \equiv 1$ in $\B_{r} \times (-r^2,r^2]$.    

We now let $v^{+}= \max \{v,0\}$, and then set $\overline{v}= v^{+} + k$, where $k = ||\psi||_{L^{\infty}(B_1 \times (-1,0))}$.  For a  given  function $f$ and $m>0$, we denote $f_m= \min \{f,m\}$.
Let $t_1, t_2$ be such that $ -r^2 < t_1 < t_2 < 0$. We indicate with $\chi(t)$ the indicator function of the interval $(t_1, t_2)$. Let $\tau$ be a compactly   supported  function in $\B_{r} \times (-r^2,0]$. For a  given $p>1$ we let $\eta= \tau^2 ((\overline{v}_{-h})_m)^{p} (\overline{v})_{-h} \chi(t)$, and take $\eta_{h}$ as a test function in the weak formulation \eqref{d8}. Again, the fact that $\eta_{h}$  is an admissible test function follows from approximating $\chi$ by piecewise linear functions, and we refer to p. 103 in Chapter 6 in \cite{Li} for details. We note here that we take $h$ small enough such that $-r^2< t_1 - h<t_2 +h <0$. Furthermore, We now use the following truncations from \cite{AS} 
\[
H^{m}(s)=  \begin{cases}
\frac{1}{p+2} s^{p+2}\ \ \ \ \ \ \ \ \ \ \ \ \ \ \ \ \ \text{for}\ s \leq m,
\\
\\
\frac{1}{2} m^p s^2 + \left(\frac{1}{p+2} - \frac{1}{2}\right) m^{p+2}\ \ \ \ \text{for}\ s\geq m.
\end{cases}
\]
Then, arguing as in the proof of Theorem 6.15 and/or Theorem 6.17 in \cite{Li}, after letting $h\to 0$ we obtain for almost every $t_1, t_2 \in (-r^2, 0]$ 
\begin{align}\label{a100}
& \int_{\B_{r}^{+} \times (t_1, t_2)}  \overline{v}_m^{p} (p |\nabla \overline v_m|^2 + |\nabla \overline{v}|^2) \tau^2  y^a dX dt  + \int_{\B_r^{+}} y^a H^{m}(\overline{v})(\cdot, t_2) \tau^2 (\cdot,t_2)
\\
& - \int_{\B_r^{+}} y^a H^m(\overline{v})(\cdot,t_1) \tau^2(\cdot,t_1) \leq C \int_{B_{r} \times (t_1,t_2)} \overline v_m^p \overline{v}^2 \tau^2
\notag
\\
& + C \int_{\B_{r}^{+} \times (t_1, t_2)}  y^a ( |\nabla \tau|^2 + (\partial_t \tau)^2 +\tau^2) \overline v_m^{p} \overline{v}^2.
\notag
\end{align}
We note that in arriving to \eqref{a100} we have crucially used that $H^{m}(\overline{v})$ is comparable  to $\frac{1}{p+2}\overline v_m^{p} \overline{v}^2$, which is in turn comparable to $\frac{1}{p} \overline v_m^{p} \overline{v}^2$ for large enough $p$. Also note that we only care about large values of $p$ since eventually we let $p \to \infty$ in Moser's iteration procedure.

At every time level $t\in (t_1,t_2)$ we now apply Lemma \ref{tr} to the function $f= \overline v_m^{p/2} \overline{v} \tau$. In doing so, with $C$ as in \eqref{a100}, we have chosen $\mu>0$ such that  
\[
C \mu^{-2s} \overline{v}_m^{p} (p^2 |\nabla \overline v_m|^2 + |\nabla \overline{v}|^2) \leq \frac{1}{2} \overline{v}_m^{p} \left(p |\nabla \overline v_m|^2 + |\nabla \overline{v}|^2\right).
\]
For instance, we  could take $\mu= ( 2C(1+p))^{1/s}$.
We thus have 
\begin{align}\label{a102}
C \int_{\B_{r}^{+} \times (t_1, t_2)\cap \{y=0\}} \overline v_m^p \overline{v}^2 \tau^2  & \leq \frac{1}{2}  \int_{\B_{r}^{+} \times (t_1, t_2)}  y^a \overline v_m^{p} (p |\nabla v_m|^2 + |\nabla \overline{v}|^2) \tau^2 
\\
& +  C_2(1+p)^{\frac{2-2s}{s}} \int_{\B_{r}^{+} \times (t_1, t_2)} y^a  \overline v_m^p\overline{v}^2( |\nabla\tau|^2 + |\tau|^2),
\notag
\end{align}
where $C_2= C C_0$,
with $C_0$ as in Lemma \ref{tr}. 
Substituting \eqref{a102} in \eqref{a100} we find for some $C = C(n,s)>0$,
\begin{align}\label{a110}
& \int_{\B_{r}^{+} \times (t_1, t_2)}  y^a \overline{v}_m^{p} (p |\nabla \overline v_m|^2 + |\nabla \overline{v}|^2) \tau^2   dX dt  + \int_{\B_r^{+}} y^a H^{m}(\overline{v})(\cdot, t_2) \tau^2 (\cdot,t_2)
\\
& - \int_{\B_r^{+}} y^a H^m(\overline{v})(\cdot,t_1) \tau^2(\cdot,t_1)
 \leq  C (1+p)^{\frac{2-2s}{s}} \int_{\B_{r}^{+} \times (t_1, t_2)}  y^a (|\nabla \tau|^2 + |\tau_t|^2 +\tau^2) \overline v_m^{p} \overline{v}^2.
\notag
\end{align} 
We next select $t_1<t_2$ such that $\tau \equiv 0$ at $t=t_1$, and $t_2$ such that 
 \[
\int_{\B_r^{+}} y^a H^{m}(\overline{v})(\cdot, t_2) \tau^2(\cdot,t_2)   \geq \frac{1}{2} \underset{t \in (-r^2,0]} {\operatorname{essup}}\  \int_{\B_r^{+}} y^a H^{m}(\overline{v})(\cdot,t) \tau^2(\cdot,t).
\]  
With these choices we obtain from \eqref{a110}, and the fact that $H^{m}(\overline{v})$ is comparable  to $\frac{1}{p}\overline{v}_m^{p} \overline{v}^2$,  
\begin{align}\label{a03}
& \underset{t \in (-r^2,0]} {\operatorname{essup}}\int_{\B_r^{+}} y^a \overline{v}_m^{p}\overline{v}^2  \tau^2 dX
\\
& \leq C(1+p)^{\frac{2-2s}{s}}\int_{\B_{r}^{+} \times (-r^2, 0]} y^a  (|\nabla \tau|^2 + |\tau_t|^2 +\tau^2) \overline{v}_m^p \overline{v}^{2}.
\notag
\end{align}
Next, we make a different selection of $t_1$ and $t_2$ in \eqref{a110}. More precisely, we let  $t_1=-r^2$ and $t_2=0$. Since $\tau(X,t_1) \equiv 0$, we obtain from \eqref{a110}
\begin{align}\label{a04}
&  \int_{\B_{r}^{+} \times (-r^2, 0]}   \overline{v}_m^{p} (p |\nabla \overline{v}_m|^2 + |\nabla \overline{v}|^2) \tau^2  y^a dX dt 
\\
& \leq C(1+p)^{\frac{2-2s}{s}} \int_{\B_{r}^{+} \times (-r^2, 0]} y^a (|\nabla \tau|^2 + |\tau_t|^2 +\tau^2) \overline{v}_m^p \overline{v}^{2}.
\notag
\end{align}

Combining \eqref{a03} and \eqref{a04} we find
\begin{align*}
 &  \int_{\B_{r}^{+} \times (-r^2, 0]}   \overline v_m^{p} (p |\nabla \overline v_m|^2 + |\nabla \overline{v}|^2) \tau^2  y^a dX dt + \underset{t \in (-r^2,0]}{\operatorname{essup}} \int_{\B_r^{+}} y^a \overline v_m^{p}\overline{v}^2  \tau^2 dX
   \\
 & \leq C(1+p)^{\frac{2-2s}{s}}  \int_{\B_{r}^{+} \times (-r^2, 0]} y^a  (|\nabla \tau|^2 + |\tau_t|^2 +\tau^2) \overline v_m^p \overline{v}^{2} dX dt.
 \notag
   \end{align*}   
   
At this point, we use the Sobolev inequality for $A_2$ weights in \cite{FKS} with the Moser iteration in \cite{M} to conclude that $\overline v$ is bounded. This implies, in particular, that $v$ be bounded from above. More precisely, we obtain the following estimate for $\overline{v}$:
\begin{equation}\label{mos}
    ||\overline{v}||_{L^{\infty}(\B_{1/2}^{+} \times (-1/4, 0])} \leq C \left( \int_{\B_{1}^{+} \times (-1, 0])} y^a \overline{v}^2\right)^{1/2}.  
    \end{equation}
In reaching the conclusion \eqref{mos} we have crucially used that the Sobolev  inequality for $A_2$ weights also holds for functions compactly supported in $\B_{r}^{+} \cup B_{r}$, instead of $\B_r$, see for instance Lemma 2.1 in \cite{TX}. Then,  the corresponding parabolic version of Sobolev inequality in Lemma 1.2 in \cite{CSe} follows similarly  for any $r>0$.  We note that,  by a standard rescaling argument, in the left-hand side of the estimate \eqref{mos} we can take the $L^{\infty}$ norm  over $\B_{\rho}^{+} \times (-\rho^2, 0]$ for any $\rho<1$. The constant  $C$ in the right-hand side will change accordingly depending on $\rho$. 
   
   The boundedness  of $v$ from below now follows by noting that $-v$ also solves the equation \eqref{ep}, but with $\psi$ replaced by $-\psi$. Therefore,  from \eqref{mos} we obtain the following Moser type estimate for some universal $C$, 
\begin{equation*}
||v||_{L^{\infty}(\B_{1/2}^{+} \times (-1/4, 0])} \leq C \left(\left(\int_{\B_{1}^{+} \times (-1,0]} y^a v^2\right)^{1/2}+ ||\psi||_{L^{\infty}(B_1 \times (-1, 0))}\right).      
\end{equation*}
Applying this estimate to $v_r(x,t)= v(rx,r^2 t)$, we find
\begin{equation*}
||v||_{L^{\infty}(\B_{r/2}^{+} \times (-r^2/4,0])} \leq C \left(r^{-(n+3+a)/2}\left(\int_{\B_{r}^{+} \times (-r^2, 0]} y^a v^2\right)^{1/2}+ r^{2s}||\psi||_{L^{\infty}(B_r \times (-r^2, 0))}\right).
\end{equation*} 

Since we now know that $v$ is locally bounded,  by adding a suitable  constant we may assume without restriction that $v > 0$  in $\B_{3/4}^{+} \times (-9/16, 0]$.     
At this point, similarly to the proof of Lemma 1 in \cite{M1}, for any given $-1 < t_1<t_2 <0$ we  let $\chi(t)$ be the characteristic function of $(t_1, t_2)$. If $\psi$ is a  smooth function having compact support in $X$ for each $t$, and
\[
\eta= \overline{v}_{-h}^{p-1} \psi^2 \chi(t),
\]
we then take $\phi= \eta_{h}$ as test  function in the weak formulation \eqref{d8}.
Following the computations in the proof of Lemma 1 in \cite{M1}, 
we obtain
\begin{align*}
  & \pm \frac{1}{4} \int_{\B_r^{+} \times (t_1, t_2)} y^a \p_t (\psi^2 w^2) + \ve \int_{\B_r^{+} \times (t_1, t_2)}  y^a \psi^2 |\nabla w|^2 \leq 
  \\
  & \frac{1}{4}   \int_{\B_r^{+} \times (t_1, t_2)}  \ve^{-1} y^a( |\nabla \psi|^2 + 2 |\psi\psi_t| ) w^2 + C_1|p| \int_{B_r \times (t_1, t_2)} w^2 \psi^2,
  \notag
\end{align*}
with 
\[
w=\overline{v}_{-h}^{p/2},\ \ \ \ \ \ \text{and}\ \ \ \ \  \ve= \frac{1}{2} |1- p^{-1}|.
\]
Here, $C_1$ depends on $L^{\infty}$ norm of $W$ over $B_1 \times (-1, 0)$ and the $+$ sign in front of the first integral on the left-hand side of the above inequality corresponds to $1/p<1$, whereas the $-$ sign corresponds to $1/p>1$. We now apply Lemma \ref{tr} to the term $\int_{B_r \times (t_1, t_2)} w^2 \psi^2$, with $\mu$ such that
\[
C_1 |p| \mu^{-2s} \leq \frac{\ve}{4}.
\]    
More precisely, we choose 
\[
\mu= \left(\frac{\ve}{4|p|C_1}\right)^{-1/2s}.
\]    
After an application of Lemma \ref{tr} at every time level with such a choice of $\mu$,   we find 
   \begin{align*}
  & \pm \frac{1}{4} \int_{\B_r^{+} \times (t_1, t_2)} y^a \p_t (\psi^2 w^2) + \frac{\ve}{2}\int_{\B_r^{+} \times (t_1, t_2)}  y^a \psi^2 |\nabla w|^2
  \\
  & \leq \frac{1}{2}   \int_{\B_r^{+} \times (t_1, t_2)}  \ve^{-1} y^a( |\nabla \psi|^2 + 2 |\psi\psi_t| ) w^2 + \left(\frac{\ve}{4|p|C_1 }\right
  )^{-\frac{1-s}{s}} \int_{\B_r^{+} \times (t_1, t_2)} y^a w^2 \psi^2. 
  \notag  
    \end{align*}
     
  At this point we argue as in page 738-739 in \cite{M1} and by finally letting $h \to 0$ (the reader should keep in mind that $w=\overline{v}_{-h}^{p/2}$), we  conclude that the weighted analogue of the  $L^{\infty}$ estimate $(6)$ in \cite{M1} holds for $\overline{v}$. This means that for $0<p <1/2 $, $t_0 \in (-1, 0]$ such that $-1< t_0-r^2<t_0+r^2 <0$, and $\rho <r <1/2$, we obtain for some constant $c_1$ independent of $p$,   
   \begin{equation}\label{m6}
    ||\overline{v}||_{L^{\infty}( \B_{\rho}^{+} \times ( t_0 - \rho^2,  t_0 + \rho^2 ])} \leq  c_1 (r-\rho)^{-(n+3 +a)/p} \left(\int_{\B_{r}^{+} \times (t_0-r^2, t_0+r^2])} y^a \overline{v}^p\right)^{1/p}.   
    \end{equation}           
   
Similarly,  the weighted  analogue of $(6^{-})$ in \cite{M1} holds  for $\overline{v}$. I.e., for any $t_0 \in (-1, 0]$ such that $t_0- r^2 >-1$, we have  for  $0 < - p < 1/2$, 
    \begin{equation}\label{m6'}
     ||\overline{v}^{-1}||_{L^{\infty}( \B_{\rho}^{+} \times ( t_0-\rho^2,  t_0])} \leq  c_2 (r-\rho)^{-(n+3 +a)/|p|}\left(\int_{\B_{r}^{+} \times (t_0 -r^2, t_0])} y^a \overline{v}^p\right)^{1/|p|},   
    \end{equation}
    where $c_2$ is also independent of $p$.  See also   the analogous estimates (2.5) and (2.5') in \cite{CSe}. We emphasize that in the  estimates \eqref{m6} and \eqref{m6'} we crucially need  the constants  $c_1$ and $c_2$ to be  independent  of $p$ for $|p| <1/2$. This is needed to apply the analogue of Bombieri's Lemma 2.3 in \cite{CSe} to $\overline{v}$, which then leads to the Harnack estimate in \eqref{wk11} below. 
    
Now for a given $0<\rho<1$ and $ -1<t_1 <t_2<0$, we take $\eta(X,t)= (\overline{v}_{-h})^{-1} \tau(X)^2  \xi(t)$, where $\xi(t)$ is the indicator function of the interval $(t_1,t_2)$, and  $\tau(X)$ is compactly  supported in $\B_{\rho}$. Using this test function in the weak formulation \eqref{d8}, and applying Young's inequality we find
   \begin{align*}
   & \int_{t_1}^{t_2} \int_{\B_{\rho}^{+} } y^a \tau^2|\nabla q|^2 + \int_{\B_{\rho}^{+}} y^a q \tau^2 (\cdot,t_2) -  \int_{\B_{\rho}^{+}} y^a q \tau^2 (\cdot,t_1)   
   \\
   & \leq C\left( \int_{t_1}^{t_2} \int_{\B_{\rho}^{+} } y^{a} |\nabla \tau|^2 +  \left|\int_{t_1}^{t_2} \int_{B_{\rho}} \frac{(Wv + \psi)_{-h}}{\overline{v}_{-h}} \tau^2\right|\right),
   \notag
   \end{align*}
   where we have let $q= - \ln \overline{v}_{-h}$. We now note that
   \begin{equation*}
   \left|\frac{(Wv + \psi)_{-h}}{\overline{v}_{-h}}\right| \leq ||W||_{L^{\infty}(B_1 \times (-1, 0))} + 2.
   \end{equation*}
Therefore, using Lemma \ref{tr} with $\mu=1$ at each time level $t$,  
we obtain 
\begin{align}\label{c500}
& \left|\int_{t_1}^{t_2} \int_{ B_{\rho}} \frac{(Wv + \psi)_{-h}}{\overline{v}_{-h}} \tau^2\right| \leq  C \int_{t_1}^{t_2} \int_{\B_\rho^{+}} y^a (\tau^2 + |\nabla \tau|^2),
\end{align}
where $C$ depends on $||W||_{L^{\infty}(B_1 \times (-1, 0))}$. Since $\tau$ is compactly supported in $\B_{\rho}^{+} \cup B_{\rho}$, by the Poincar\'e inequality for $A_2$ weights, we have  for some constant $C_5$
\begin{equation*}
\int_{t_1}^{t_2} \int_{\B_\rho^{+}} y^a \tau^2  \leq C_5 \int_{t_1}^{t_2} \int_{\B_\rho^{+}} y^a |\nabla \tau|^2.
\end{equation*}
We infer
\begin{align}\label{c501}
& \left|\int_{t_1}^{t_2} \int_{ B_{\rho}} \frac{(Wv + \psi)_{-h}}{\overline{v}_{-h}} \tau^2\right| \leq C_6\int_{t_1}^{t_2} \int_{\B_\rho^{+}} y^a |\nabla \tau|^2
\end{align}
for some $C_6$. Finally, from \eqref{c500} and \eqref{c501} we obtain, 
\begin{align*}
   & \int_{t_1}^{t_2} \int_{\B_{\rho}^{+} } y^a \tau^2|\nabla  q|^2 + \int_{\B_{\rho}^{+}} y^a q \tau^2 (\cdot,t_2) -  \int_{\B_{\rho}^{+}} y^a q \tau^2 (\cdot,t_1) 
   \\
   & \leq C \int_{t_1}^{t_2} \int_{\B_\rho^{+}} y^a |\nabla \tau|^2.
   \notag
   \end{align*}
We next observe that the modified  Poincar\'e inequality for $A_2$ weights, similar to the one that in Lemma 1.3 in \cite{CSe} is stated on the whole ball $\B_r$, continues to be valid for $\B_{r}^{+} $. One can thus argue as on p. 188-189 in \cite{CSe} and conclude that the analogue of Lemma 2.2 in \cite{CSe} holds for $\overline{v}$ in $\B_1^{+} \times (-1,0]$. Combining such result with \eqref{m6} and \eqref{m6'} above, we now use Bombieri's Lemma 2.3 in \cite{CSe}, and argue as in p. 734 in \cite{M1},   to finally conclude by rescaling that the following Harnack estimate holds for $r<1/2$, 
\begin{equation}\label{wk11}
\underset{\B_r^{+} \times (-r^2, -\frac{r^2}2]}{\sup}\ v  \leq C \left(\underset{\B_r^{+} \times (-\frac{r^2}4, 0]}{\inf}\ v +  r^{2s} ||\psi||_{L^{\infty}(B_{2r} \times (-2r^2, 0))}\right).
\end{equation}
Finally, the local H\"older continuity of $v$  follows from \eqref{wk11} in a  standard  way and we refer the reader to Chapter 6 in \cite{Li} for details. 

\end{proof}

Before we proceed further, we note that in the course of the proof of Theorem \ref{reg5} we have established the following basic result.

\begin{thrm}[Scale invariant Harnack inequality]\label{T:harnack}
Let $V, \psi \in L^{\infty} (B_{2r} \times (-2r^2, 0])$ and let $u \in \operatorname{Dom}(H^{s})$, $u\ge 0$, be a solution to
\[
H^{s} u = Vu + \psi
\]
in $B_{2r} \times (-2r^2, 0])$. Then, the following Harnack inequality holds for any $r>0$
\[
\underset{B_r \times (-r^2, -\frac{r^2}{2})}{\sup}\  u \leq C ( \underset{B_r \times (-\frac{r^2}{4}, 0)}{\inf}\ u + r^{2s} ||\psi||_{L^{\infty}( B_{2r} \times (-2r^2, 0])}),
\]
where $C = C(n, s,r^{2s}||V||_{L^\infty(B_{2r} \times (-2r^2, 0])})>0$.
\end{thrm}
\begin{proof}
This is a direct consequence of the estimate \eqref{wk11} above.  
\end{proof}

Theorem \ref{reg5} has the following basic consequence.

\begin{cor}\label{bdd}
With $u$ as in Theorem \ref{main}, we have that  $u \in \Ha^{\alpha}(\Rnn)$ for some $\alpha >0$. In particular, $u \in L^{\infty}(\Rnn)$.
\end{cor}
\begin{proof}
Given an arbitrary point $(x_0, t_0) \in \Rnn$, the corollary follows from a direct application of Theorem \ref{reg5} above to the extension function $U$ in $B_{1} (x_0) \times \{ y < 1\}  \times (t_0 -1,t_0]$ combined with the  estimates in Lemma \ref{regU1}. 

\end{proof}

\begin{rmrk}\label{R:sup}
We note that using the representation \eqref{stbis}, the equation  \eqref{e00} satisfied by $u$, Corollary \ref{bdd} and the estimates  \eqref{vasump} and \eqref{linfty}, we obtain the following estimate:
\begin{equation}\label{sup}
||U_y(\cdot,y,\cdot)||_{L^{\infty}(\R^{n+1})}  \leq C(s)  ||V||_{L^\infty(\R^{n+1})}||u||_{L^{\infty}(\R^{n+1})} \frac{1}{y^{1-2s}},
\end{equation}
where $C(s)>0$ depends on $s$ only.
\end{rmrk}

We will also need the following regularity result.  
  
\begin{lemma}\label{reg2}
Let $v \in V^{a, r, -4r^2, 0}$ be a weak solution in $\B^{+}_{2r} \cap (-4r^2, 0]$ to 
\begin{equation}\label{reg}
\begin{cases}
\operatorname{div}(y^a \nabla v)= y^a v_t
\\
- \lim_{y \to 0} y^a v_y= \phi,
\end{cases}
\end{equation}
where $\phi$ is assumed to be  $C^{2} $ for $0<s<1/2$, and $C^{1}$ for $s\geq 1/2$. Then, there exists $\alpha' > 0$ depending on $a, n$ such that for $i=1,...n$, we have $D_i v, v_t , y^{a} v_y \in \Ha^{\alpha'}$  up to $y=0$  in $\B_{r}^{+} \times (-r^2, 0]$.
\end{lemma}
\begin{proof}
The proof follows from ideas similar to that of   Lemma 4.5  in  \cite{CS}  for the elliptic case. We first note that, with $\alpha$ as in Theorem \ref{reg5}, we have that  $v$ is in $ \Ha^{\alpha}$ up to $y=0$  in $\B_{r}^{+} \times (-r^2, 0]$. Given such $\alpha$, starting with  $k=1, 2, ..., \lfloor{1/\alpha}\rfloor + 1$, we take iterated difference quotients of $v$ of the type
\begin{align}\label{rep}
& v^{h}_{ e_i}= \frac{v(x+he_i, t) - v(x, t)}{ h^{k\alpha}},\ \ \ v^{h}_ {t} = \frac{v(x, h+t) - v(x, t)}{h^{k\alpha/2}},
\end{align}
$i=1, ..., n$. Note that $v^{h}_{ e_i}$, $v^{h}_ {t}$ solve a problem similar to \eqref{reg}, with $\phi$ replaced by its corresponding difference quotient. Therefore   by making use of the fact that $\phi$ is Lipschitz,  we can conclude by   applying Theorem \ref{reg5}  to these  difference quotients that $ D_i v, v_t \in \Ha^{\alpha}$, $i=1,...,n$. Applying Nirenberg's method of difference quotients in the directions $x_1, ... , x_{n}$ (see e.g. the proof of Theorem 6.6 in \cite{Li}), and also using the trace inequality  as in the proof of  Theorem \ref{reg5} above, we conclude that $y^a (v_{ij})^2, y^a(v_i)_y^2  \in L^{1}(\B_r \times (-r^2, 0])$ for $i=1,..,n$ and $j=1, ...n$. From the equation \eqref{reg}, and the fact that $v_t \in \Ha^{\alpha}(\B_r^{+} \times (-r^2, 0])$, and therefore it is bounded, we conclude that $ y^{-a} ((y^{a} v_y)_{y})^2 \in L^{1} (\B_r \times (-r^2, 0])$. This implies that  $y^a v_y \in V^{-a, r, -r^2, 0}$. 

At this point we note that $h=y^{a} v_y$ is a weak solution to the following conjugate equation:
\begin{equation}\label{reg5'}
\begin{cases}
\operatorname{div}(y^{-a} \nabla h)= y^{-a} h_t,
\\
h(x, 0, t)= -\phi.
\end{cases}
\end{equation}
Here, the boundary condition  is interpreted in the trace sense. Now, similarly to the proof of Lemma 4.5 in  \cite{CS}, we first extend $\phi$  in a $C^{2}$  manner to the whole of $\R^{n+1}$  for $s<1/2$, and in a $C^{1}$ manner for $s\geq 1/2$,  and then   we multiply $\phi$ by a cut off $\eta \equiv 1$ in $\B_{r} \times (-r^2, 0]$ and call this new function $\overline{\phi}$. Let $\overline{h}$ be the solution of \eqref{reg5'} with $\phi$ replaced by $\overline{\phi}$. It follows from \eqref{pos} in Theorem \ref{C:st} that $\overline{h}$ has a Poisson representation formula  given by (1.7) in \cite{ST}. More specifically, from the representation \eqref{pos}  with  $s$ replaced by $1-s$,  we have 
\begin{equation*}
\overline{h}(X,t)= C(n, s) \int_{0}^{\infty} \int_{\Rn} \frac{y^{2(1-s)}e^{-\frac{y^2+ |z|^2}{4\tau}}}{\tau^{n/2 + 2 -s}} \overline{\phi} ( t -  \tau, x-z) dz d\tau.
\end{equation*}
The change of variable $\frac{z}{\sqrt{\tau}}=\xi$ gives
\begin{equation*}
\overline{h}(X,t)= C(n, s) y^{2(1-s)}\int_{0}^{\infty} e^{-\frac{y^2}{4\tau}}\int_{\Rn} \frac{e^{- \xi^2/4} }{\tau^{  2 -s}} \overline{\phi} ( t -  \tau, x-\sqrt{\tau} \xi) d\xi d\tau.
\end{equation*}
Next, we let $\frac{\tau}{y^2}= r$ to find
\begin{equation*}
\overline{h}(X,t)= C(n, s) \int_{0}^{\infty} e^{-\frac{1}{4r}}\int_{\Rn} \frac{e^{- \xi^2/4} }{r^{  2 -s}} \overline{\phi} ( t -  y^2r , x- y \sqrt{r} \xi) d\xi d r.
\end{equation*}
This implies $\overline{h}$ is in $\Ha^{\alpha'}$ up to $y=0$ for all $\alpha'< 2(1-s)$. We also note that $D_i \overline{h} \in L^{\infty}$ for $i=1,...,n$. This follows from the representation \eqref{pos} with $u$ replaced  by $\overline{\phi}$ and $s$ replaced by $1-s$. The alternate representation of $\overline{h}$ in \eqref{stbis}, with $s$ replaced by $1-s$, and the fact that $||e^{-tH} f||_{L^2(\Rnn)} \leq ||f||_{L^{2}(\Rnn)}$, allow to infer 
\begin{equation}\label{rep1}
||\overline{h}_{y}||_{L^{2}(\Rnn)} \leq C||H^{1-s} \overline{\phi}||_{L^2(\Rnn)
} \frac{1}{y^{2s-1}}.
\end{equation}
From \eqref{rep1} we conclude in a standard way that $\overline{h} \in V^{-a}$. We note that the regularity  assumption on $\phi$ in the hypothesis of the lemma  is needed in \eqref{rep1} to ensure that $\overline{\phi}$ is in the domain of $H^{1-s}$. Therefore, 
  $h - \overline{h}$ is a weak solution to \eqref{reg5'} with $\phi =0$ in the set where $\eta \equiv 1$, i.e., in $\B_{r}^{+} \times (-r^2, 0]$. At this point, by taking the odd reflection of the  function $h- \overline{h}$ in the variable $y$  across $y=0$, we obtain that $f \overset{def}{=} h-\overline{h}$ is a weak solution to
\begin{equation}\label{full}
\operatorname{div}(|y|^{-a} \nabla f) = |y|^{-a}  f_t.
  \end{equation}
   Now we can use results from \cite{CSe} and \cite{Is} to conclude that $h - \overline{h}$ is H\"older continuous  up to $y=0$ in $\B_r^{+} \times (-r^2, 0]$. This finally implies the H\"older continuity of $h= y^a v_y$. 

\end{proof}

We now have the   following result concerning the H\"older  continuity of $y^a U_y$.

\begin{lemma}\label{regU}
With $U$ as in \eqref{ext}, we have that $y^{a} U_y$ is H\"older continuous  up to $y=0$   with   exponent $\alpha'>0$ depending only on  $a$ and $n$.
\end{lemma}

\begin{proof}
From Theorem \ref{reg5}, we have that $U \in H^{\alpha}$ upto  $y=0$ for some $\alpha$ depending on $a, n$. Now with this choice of  $\alpha$,    by taking repeated difference quotients   of the type \eqref{rep} with $v$ replaced  by $U$ similarly to that  in the proof of Lemma \ref{reg2}, we note that  such difference quotients satisfy in the weak sense
\begin{equation}
\begin{cases}
\operatorname{div} (y^a \nabla w) = y^a w_t,
\\
\underset{y \to 0}{\lim}\ y^a w_y= -f + gw.
\end{cases}
\end{equation}
Now because of  the  regularity  assumptions on $V$ in   \eqref{vasump} and the regularity of $U$ gained  by applying Theorem \ref{reg5} in the previous  step $k-1$, we have that $f, g \in L^{\infty}$. Therefore, by repeated use of Theorem \ref{reg5}, we  conclude that for $i, j=1,...,n$,  $D_i U, U_{t} \in \Ha^{\alpha}$ up to $y=0$  in the set $\B_{M}^{+} \times (-M^2, M^2)$ for any $M>0$. For $s\geq 1/2$, we have now that  $\phi=Vu$ is in $C^{1}$  and at this point we can use Lemma \ref{reg2} to conclude the H\"older continuity of $y^a U_y$ upto $y=0$.

We next consider the case $0<s<1/2$. We first note that from the above arguments, we have that for $i=1, ..., n$, $D_i U$ satisfies in the weak sense
\begin{equation}
\begin{cases}
\operatorname{div}(y^a \nabla D_i U) = y^a (D_i U)_t,
\\
\underset{y \to 0}{\lim} y^a (D_i U)_y=  - D_i V u - V D_i u.
\end{cases}
\end{equation}
Similarly, $U_t$ satisfies
\begin{equation}
\begin{cases}
\operatorname{div}(y^a \nabla U_t) = y^a (U_t)_t,
\\
\underset{y \to 0}{\lim} y^a (U_t)_y=  - V_t u - Vu_t.
\end{cases}
\end{equation}

Since $V$ is in $C^{2}$ in this case, and it satisfies the bound \eqref{vasump},  therefore  for $i=1,...,n$, by again taking repeated difference quotients in the $x_i$  and $t$ directions, we conclude that for $i,j=1, ..., n$, $D_{ij} U, U_{tt}, D_t D_i U \in \Ha^{\alpha}$ up to $y=0$ in $\B_{M}^{+} \times (-M^2, M^2)$ for any $M>0$ where $\alpha$ is as in Theorem \ref{reg5}. Therefore, we now have $\phi= Vu$ is in $C^{2}$ and at this point we can again use Lemma \ref{reg2} to conclude the H\"older continuity of $y^aU_y$ up to $y=0$. 

\end{proof}

Before  we  proceed  further, we have the following remark.  We would like to mention that  the notation $C^{k}_{(x,t)}(\Om)$ indicates the standard isotropic $C^k$ spaces with respect to the variable $(x,t)\in \Rnn$. Such spaces are endowed with the corresponding $C^{k}$ norm.
\begin{rmrk}\label{C1}
It follows from the proof of  Lemma \ref{regU} that $u \in C^{1}_{(x,t)}(\Rnn)$.
\end{rmrk}
We finally  conclude this section  with the following $C^{1}_{(x,t)}(\Rnn)$  estimate  for $U(\cdot, y,\cdot)$.

\begin{lemma}\label{reg1}
Let $U$ be the solution to \eqref{ext}. Then, for every  $y > 0$ we have 
  \begin{equation*}
  ||U(\cdot,y,\cdot)||_{C_{(x,t)}^{1} (\R^{n+1})} \leq C(n,s) ||u||_{C^{1}_{(x,t)}(\R^{n+1})}.
  \end{equation*}
  \end{lemma}
  
  \begin{proof}
  This follows from  the representation \eqref{pos} of $U$, and the proof is similar to the time-independent case of Proposition 2.8 in \cite{Yu}.
    
\end{proof}


 \section{Monotonicity of the frequency}\label{S:mono}
 
In this section we consider a solution $u\in \operatorname{Dom}(H^s)$ to the equation \eqref{e00}, and denote by $U$ the corresponding solution 
of the extension problem \eqref{ext}, where as before $a = 1-2s$. Given such $U$, we introduce the height  $H(U,r)$ of $U$, its energy $I(U,r)$ and the \emph{frequency} $N(U,r)$ of $U$, see Definition \ref{D:fre} below. We intend to study the monotonicity properties of $N(U,r)$ as a function of $r>0$.
Our main objective is proving Theorem \ref{mon1}. 

Henceforth, for $X\in \Rnp$ and $t<0$ we let 
\[
\Ga = \Ga(X,t) = \Ga(0,0,0,x,y,t)  = (-4\pi t)^{-\frac{n+1}{2}} e^{\frac{|X|^2}{4t}} = (4\pi |t|)^{-\frac{n+1}{2}} e^{-\frac{|X|^2}{4|t|}}
\]
 denote the  backward heat kernel in $\Rnn \times \R$ centered at $(0,0,0)$. We thus have for $X\in \Rnp$ and $t<0$
\begin{equation}\label{f}
\nabla \Ga= \frac{X}{2t} \Ga,
\end{equation}
and
\begin{equation}\label{g}
\Delta \Ga + \Ga_t=0.
\end{equation}
Following \cite{DGPT}, for $t<0$ we introduce the quantity
\begin{equation}\label{h1}
h(U,t)=  \int_{\Rnp} y^{a}\ U(X,t)^2\ \Ga(X,t)\ dX,
\end{equation}
where $\Rnp = \{X = (x,y)\in \Rnn\mid y>0\}$. We also define the quantity  
for $t<0$ 
\begin{equation}\label{i}
i(U, t)= - t \int_{ \Rnp}  y^a |\nabla U(X,t)|^2 \Ga(X,t)\ dX + t  \int_{ \R^{n}\times \{0\}} V(x,t) u(x,t)^2 \Ga(x,0,t)\ dx.
\end{equation}
Before proceeding, it is important to note that  \eqref{sup}  above and Lemma \ref{reg1} imply that $i(U,t)$ is finite. We have in fact $|\nabla U(X,t)|^2 = |\nabla_x U(X,t)|^2 + U_y(X,t)^2$. Now, \eqref{sup} gives for every $X\in \Rnp$ and $t<0$
\[
U_y(X,t)^2  \leq \left(\frac{2^{1-2s}\G(1-s)}{\G(s)}   ||V||_{L^\infty(\R^{n+1})}||u||_{L^{\infty}(\R^{n+1})}\right)^2 y^{-2a}. 
\]
Therefore, 
\begin{align}\label{g1}
\int_{ \Rnp}  y^a U_y(X,t)^2 \Ga(X,t)\ dX & \le C (4\pi |t|)^{-\frac{n+1}{2}} \int_{ \Rnp}  y^{-a}  e^{-\frac{|X|^2}{4|t|}}\ dX
\\
& = C (4\pi |t|)^{-\frac{n+1}{2}} \int_{\Rn} e^{-\frac{|x|^2}{4|t|}} dx \int_0^\infty y^{-a}  e^{-\frac{y^2}{4|t|}} dy < \infty,
\notag
\end{align}
since $a<1$. On the other hand, Lemma \ref{reg1} gives 
 \begin{equation*}
|\nabla_x U(X,t)| \leq C(n,s) ||u||_{C^{1}_{(x,t)}(\R^{n+1})},
  \end{equation*}
and therefore 
\begin{align}\label{g2}
\int_{ \Rnp} y^a |\nabla_x U(X,t)|^2 \Ga(X,t)\ dX & \le C (4\pi |t|)^{-\frac{n+1}{2}} \int_{ \Rnp}  y^{a}  e^{-\frac{|X|^2}{4|t|}}\ dX < \infty,
\end{align}
since $a>-1$. Finally, we have
\[
\int_{ \R^{n}} |V(x,t) u(x,t)^2 \Ga(x,0,t)|\ dx \le ||V||_{L^\infty(\Rnn)}||u||^2_{L^\infty(\Rnn)} (4\pi |t|)^{-\frac{n+1}{2}} \int_{ \R^{n}} e^{-\frac{|x|^2}{4|t|}} dx < \infty.
\]

Our first result is the following alternate expression of the energy $i(t)$.

\begin{lemma}\label{L:ibi}
For every $t>0$ one has
\begin{equation}\label{i1}
i(t) =  t \int_{\Rnp}  y^{a} \left(UU_t + U < \nabla U, \frac{X}{2t}>\right) \Ga = \frac 12  \int_{\Rnp}  y^{a} U ZU \Ga,
\end{equation}
where $Z$ is the vector field in \eqref{z1}, \eqref{z11} above.
\end{lemma}

\begin{proof}
To verify \eqref{i1} we first observe that, thanks to \eqref{g1}, \eqref{g2} and dominated  convergence, we can write 
\begin{equation}\label{eps}
\int_{\Rnp}  y^a |\nabla U|^2 \Ga = \underset{\ve\to 0^+}{\lim} \int_{\Rn\times \{y > \ve\}}  y^a |\nabla U|^2 \Ga.
\end{equation}
From the equation $y^{a} U_t = \operatorname{div}(y^{a} \nabla U)$ satisfied by $U$, see \eqref{ext}, we see that on the set $\Rn \times \{y > \ve\}$ we have
\begin{align*}
y^a |\nabla U|^2  & = \frac 12 \operatorname{div}(y^{a} \nabla U^2) - \frac 12 y^{a} (U^2)_t  
\end{align*}
Therefore,  the divergence theorem and \eqref{f} give
\begin{align}\label{ve}
& \int_{\Rn\times \{y >\ve\}} y^a |\nabla U|^2 \Ga=  -\int_{\Rn\times \{y >\ve\}}  y^a U_t U \Ga - \int_{\Rn\times \{y >\ve\}}y^a <\nabla U,\frac{X}{2t}> U \Ga
\\
& - \int_{\Rn\times \{y=\ve\}} y^a U_y U \Ga.
\notag
\end{align}
Using the H\"older continuity of $y^a U_y$ which follows from Lemma \ref{regU}, \eqref{ext} that gives 
\[
\underset{y\to 0^+}{\lim} y^{a} \frac{\p U}{\p y}(x,y,t) = - V(x,t) u(x,t),
\]
the estimate \eqref{sup} and dominated convergence, passing to the limit as $\ve \to 0$ in \eqref{ve} we reach the desired conclusion \eqref{i1}.

\end{proof}

We next introduce averaged versions of the quantities $h(U,t)$ and $i(U,t)$. Similarly to the case $s = 1/2$ in \cite{DGPT}, it is crucial to work with these new quantities since, unlike \eqref{h1} and \eqref{i}, they lead to a priori estimates which are essential for the compactness arguments in Section \ref{S:comp}. Before proceeding further, for every $r>0$ we introduce the notation:
\begin{equation}\label{strips}
\begin{cases}
 \Sa_r= \{(X,t)|X\in \Rnn,\ \  -r^2 < t< 0\},
 \\
\Sa_{r}^{+} =\{(X,t)\mid X\in \Rnp,\ \  -r^2 <t <0\},
\\  
 S_r =\{(x,t)\mid x\in \Rn,\ \  -r^2 <t < 0\}.
 \\
 \end{cases}
 \end{equation}

\begin{dfn}\label{D:fre}
We define the \emph{height function} of $U$ as
\begin{equation}\label{H}
H(U,r)= \frac{1}{r^2} \int_{-r^2}^0 h(U,t) dt = \frac{1}{r^2} \int_{\Sa_{r}^{+}} y^{a} U(X,t)^2 \Ga(X,t) dX dt,
\end{equation}
and the \emph{energy} of $U$ as
\begin{align}\label{I100}
I(U,r) & =  \frac{1}{r^2}\int_{-r^2}^ 0 i(U,t) dt = \frac{1}{r^2} \int_{\Sa_{r}^{+}} |t|  y^a |\nabla U(X,t)|^2 \Ga(X,t) dX dt
\\
& - \frac{1}{r^2} \int_{S_r} |t| V(x,t) u(x,t)^2 \Ga(x,0,t) dx dt.
\notag
\end{align}
For those values of $r>0$ for which $H(U,r)\not= 0$ the \emph{frequency} of $U$ is defined as 
\begin{equation}\label{N100}
N(U,r)= \frac{I(U,r)}{H(U,r)}.
\end{equation}
\end{dfn}

\begin{rmrk}\label{R:nondeg}
We note that, unless $U(X,t) \equiv 0$ in $\Sa_{r}^{+}$, we must have $H(U,r) \not= 0$.
Otherwise, we would have a contradiction from \eqref{H}. 
\end{rmrk}

We record for later use the following consequence of Lemma \ref{L:ibi} and of definition \eqref{I100}.

\begin{lemma}\label{L:I}
For every $r>0$ one has
\begin{equation}\label{e91}
2 I(U,r) = \frac{1}{r^2} \int_{\Sa_{r}^{+}} y^a U ZU \Ga dX dt.
\end{equation}
\end{lemma}

Henceforth, to simplify the notation, whenever convenient we will simply write $h(t)$ and $ H(r)$ instead of $h(U,t)$ and $H(U,r)$ respectively. Similarly, we will denote $i(U,t)$ by $i(t)$,  $I(U,r)$ by $I(r)$ and $N(U,r)$ by $N(r)$. Also, unless there is risk of confusion, as in Section \ref{S:nltol} we will routinely avoid writing explicitly $dX, dx, dt, dr$, etc., in the integrals involved. Thus, for instance, \eqref{h1} will be written as
\[
h(t)=  \int_{\Rnp} y^{a} U^2 \Ga.
\]

The next result plays a basic role in what follows. 

\begin{lemma}[First variation of the height]\label{L:fvh}
For every $r>0$ one has
\begin{equation}\label{mon10}
H'(r) = \frac{4}{r} I(r) + \frac{a}{r} H(r).
\end{equation}
\end{lemma}

\begin{proof}
As a first step we calculate $h'$. We write
\[
h(t) = \int_{\Rnp} F(X,t) dX,
\]
where $F(X,t) = y^{a} U(X,t)^2 \Ga(X,t)$. If we know that for a.e. $X\in \Rnp$ there exists $\frac{\p F}{\p t}(X,t) $ and that
\[
\frac{\p F}{\p t}(X,t) = 2 y^a U(X,t) U_t(X,t) \Ga(X,t) + y^{a} U(X,t)^2 \Ga_t(X,t)
\]
has a dominant in $X\in \Rnp$ which is in $L^1(\Rnp)$ and which is uniform in $t$, then differentiating under the integral sign we can conclude that
\begin{equation}\label{hh}
h'(t) = \int_{\Rnp} \frac{\p F}{\p t}(X,t) dX = \int_{\Rnp} y^a ( 2UU_t \Ga + U^2 \Ga_t).
\end{equation}
Of course, for any given $t<0$ it suffices to verify such uniform integrability in a small interval $(t-\delta,t+\delta)$. Now the desired conclusion follows immediately since from Lemma \ref{reg1} we know that $U, U_t\in L^\infty(\Rnn)$. Therefore, \eqref{hh} holds true. Using \eqref{g}, we now obtain
\begin{equation}\label{e1}
h'(t)= \int_{\Rnp} y^a \left(2 UU_t \Ga -U^2 \Delta \Ga\right).
\end{equation}
We next integrate by parts the second term in the right-hand side of \eqref{e1}. This can be justified by first performing the integration by parts over the interior regions $\Rn\times \{y > \ve \}$, using \eqref{f}, and then pass to the limit $\ve \to 0$ using the estimates in  \eqref{sup} and  Lemma \ref{reg1}. We note that for a given $\ve$ such integration by parts would produce the boundary term
\[
\int_{\Rn \times \{y=\ve\}} \frac{ \ve^{1+a} }{2t} U^2 \Ga
\]
which, thanks to the boundedness of $U$ and the fact that $a>-1$, converges to $0$  as $\ve \to 0$. 
We thus find 
\begin{equation}\label{e'}
h'(t) = \int_{\Rnp} y^ a \left( 2 UU_t\Ga + 2 U<\nabla U, \frac{X}{2t} > \Ga\right) + \int_{\Rnp} a y^{a-1} \frac{y}{2t} U^2 \Ga.
\end{equation}

From \eqref{e'}, Lemma \ref{L:ibi} and \eqref{h1} it is now obvious that we have proved
\begin{equation}\label{h''}
h'(t) = \frac{2}{t} i(t) + \frac{a}{2t} h(t).
\end{equation}
We next use \eqref{h''} to compute $H'(r)$. The first step is to observe that
\eqref{H} and \eqref{I100} give
\begin{equation}\label{HI}
H(r)= \int_{-1}^0 h(r^2 t) dt,\ \ \ \ \  I(r)= \int_{-1}^0 i(r^2 t) dt.
\end{equation}
We next want to differentiate under the integral sign 
to obtain
\begin{equation}\label{dh1}
H'(r)=  2r \int_{-1}^0  t h'(r^2 t) dt.
\end{equation}
Suppose for a moment we have proved \eqref{dh1}. Then, from \eqref{dh1} and \eqref{h''} we obtain
\begin{align*}
H'(r) & = 2r \int_{-1}^0  \left(\frac{2}{r^2} i(r^2 t) + \frac{a}{2 r^2} h(r^2 t)\right) dt
\\
& = \frac 4r I(r) + \frac{a}{r} H(r),
\end{align*}
which gives the desired conclusion \eqref{mon10}. 

We are thus left with proving \eqref{dh1}. One obstruction to directly differentiating under the integral sign is represented by the fact that the
integrals involved in \eqref{e'} may become unbounded near the endpoint $t=0$, where $\Ga$
becomes singular. To remedy this problem we introduce the following
truncated versions of $H$ and $I$. 
For a given $\ve>0$ we consider 
\[
H_{\ve}(r) = \frac{1}{r^2} \int_{-r^2}^{-\ve r^2}   h(t) dt = \int_{-1}^{-\ve} h(r^2 t)dt.
\]
We note that from the expression of $h'$ in \eqref{e'}, using \eqref{sup}, the estimate in Lemma \ref{reg1} and dominated convergence, we deduce that
\[
H_{\ve}'(r)= 2r \int_{-1}^{-\ve} t h'(r^2 t) dt.
\]
Here, we have crucially used the fact that $h'(t)$ is evaluated at points $t$  such that  $t < - \ve$, and therefore uniform bounds are available.

Now fix $\delta>0$ arbitrarily. We claim that as $\ve \to 0^+$:
\begin{itemize}
\item[(i)] $\underset{r\in [\delta,1]}{\sup} \left|H_\ve(r) - H(r)\right| \longrightarrow 0$;
\item[(ii)] $\underset{r\in [\delta,1]}{\sup} \left|H_{\ve}'(r) - 2r \int_{-1}^{0}  t h'(r^2 t)dt\right| \longrightarrow 0$.
\end{itemize}
Taking the claim for granted, from it we infer that for $r \in [\delta,1]$
\[
H'(r) =  2r \int_{-1}^{0}  t h'(r^2 t)dt,
\]
which is \eqref{dh1}. The arbitrariness of $\delta$ implies that \eqref{dh1} holds for $ r \in (0, 1]$, thus completing the proof. We are thus left with proving the claim. We first prove (i).
To see this, observe that Lemma \ref{reg1} gives
\begin{align*}
\left|H_\ve(r) - H(r)\right| & \le \int_{-\ve}^0 |h(r^2 t)| dt =   \int_{-\ve}^0 \int_{\Rnp} y^{a} U(X,r^2 t)^2 \Ga(X,r^2 t) dX dt
\\
& \le C \int_{-\ve}^0 \int_{\Rnp} y^{a} \Ga(X,r^2 t) dX dt \le C r^a \int_{-\ve}^0 |t|^{a/2} dt
\\
& \le C \delta^{-|a|} \int_{-\ve}^0 |t|^{a/2} dt \to 0
\end{align*}
as $\ve \to 0^+$, uniformly in $r\in [\delta,1]$. This proves (i). Next, we establish (ii). Using \eqref{e'} and \eqref{sup}, which gives 
\begin{equation}\label{est50}
<\nabla U(X,t),X> \ \leq C (|x| + y^{1-a}),
\end{equation}
we obtain 
\begin{align*}
& \left|H_{\ve}'(r) - 2r \int_{-1}^{0}  t h'(r^2 t)dt\right| \leq 2r \int_{-\ve}^0 |t| |h'(r^2 t)| dt
\\
& \le r \int_{-\ve}^0 |t| \int_{\Rnp} y^a \bigg|4 U(X,r^2 t) U_t(X,r^2 t) + U(X,r^2 t)<\nabla U(X,r^2 t), \frac{X}{r^2 |t|} >
\\
& +  \frac{a}{r^2 |t|} U(X,r^2 t)^2\bigg|\ \Ga(X,r^2 t) dX dt
\\
& \le C  r \int_{-\ve}^0 |t| \int_{\Rnp} y^a \Ga(X,r^2 t) dX dt + \frac Cr \int_{-\ve}^0 \int_{\Rnp} y^a (|x| + y^{1-a}) \Ga(X,r^2 t) dX dt
\\
& + \frac Cr \int_{-\ve}^0 \int_{\Rnp} y^a \Ga(X,r^2 t) dX dt
\\
& \le C  \int_{-\ve}^0 |t|^{1+\frac a2} dt + C \delta^{-|a|} \int_{-\ve}^0 |t|^{\frac{1+a}{2}} dt + C \int_{-\ve}^0 |t|^{\frac 12} dt 
\longrightarrow 0
\end{align*}
as $\ve \to 0^+$, uniformly in $r\in [\delta,1]$. This completes the proof of the lemma.

\end{proof}

The following result concerning the first variation of the energy $i(t)$ plays a central role in the proof of Theorem \ref{mon1}.

\begin{lemma}[First variation of the energy $i(t)$]\label{L:fvei}
For every $t\in (-1,0)$ one has
\begin{align}\label{e7}
& i'(t) =  \frac{a}{2t} i(t) +  2t\int_{\Rnp} 
 y^{a} \left(U_t + < \nabla U, \frac{X}{2t}>\right)^2\Ga
\\
& + \frac{1-a}{2} \int_{\Rn\times\{0\}} Vu^2 \Ga + t \int_{\Rn\times\{0\}} V_t u^2 \Ga + \frac 12 \int_{\Rn\times\{0\}} <\nabla V,x> u^2\Ga.
\notag
\end{align}
\end{lemma}

\begin{proof}
The proof of \eqref{e7} is rather technical and to simplify its presentation we divide it into three steps.

\textbf{Step 1:} We work first under the following qualitative assumption:
\begin{equation}\label{qasump}
 ||u||_ {C^{3}(\Rnn)} \leq K',\ \ \ \ \ \ \ \  ||V||_{C^3(\R^{n+1})} \leq K',
 \end{equation} 
for some constant $K'>0$. Subsequently, in the Step $3$ using an approximation argument we remove  \eqref{qasump}.  

First of all, we note that from the representation formula \eqref{pos} the hypothesis \eqref{qasump} implies that $U(\cdot,y,\cdot) \in C^{2, \alpha}_{x,t}(\Rnn)$ for any $\alpha\in(0,1)$. Therefore, we can differentiate in $t$ and $y$, to assert the following estimate
\begin{equation}\label{x2}
||(U_{t})_{y}(\cdot,y,\cdot)||_{\infty} \leq \frac{C}{y},\ \ \ \ \ \ \ \ y>1. 
\end{equation}

Furthermore, using  the equation \eqref{ext}  and taking difference quotients in time, we see that $U_t$ is a weak solution to \eqref{reg} in Lemma \ref{reg2} above with $\phi= Vu_t + V_t u$. We note here that the application of Lemma \ref{reg2} requires $\phi= Vu_t + V_tu$ to be $C^2$ for  $0<s<1/2$, a fact that is guaranteed by \eqref{qasump}. Then, from Lemma \ref{reg2} we obtain  the following estimate which follows from the H\"older continuity of $y^a (U_t)_y$,
\begin{equation}\label{x1}
||(U_t)_{y}(\cdot,y,\cdot)||_{\infty} \leq  \frac{C}{y^{1-2s}},\ \ \ \ \ \ 0<y\le 1.
\end{equation}
The estimates \eqref{x2}, \eqref{x1} will play a key role in guaranteeing that some of the forthcoming integrals  involving $\nabla U_t$ in \eqref{dh6}, \eqref{e''}, \eqref{e'''}, \eqref{intermed}, \eqref{e4}, \eqref{e5} and  \eqref{abov}   are finite.

In order to compute $i'(r)$ we argue similarly to the calculation of $h'(r)$, using  \eqref{sup}, Lemma \ref{reg1},  Lemma \ref{regU}, \eqref{x2} and \eqref{x1} to justify differentiating under the integral sign. We thus obtain from \eqref{i}
\begin{align}\label{dh6}
& i'(t)= - \int_{\Rnp} y^{a}  |\nabla U|^2 \Ga  +  \int_{\Rn\times\{0\}} Vu^2 \Ga +2 t\int_{\Rn\times\{0\}} V u u_t \Ga+ t \int_{\Rn\times\{0\}} V_t u^2 \Ga 
\\
& + t \int_{\Rn\times\{0\}} Vu^2 \Ga_t - t \int_{\Rnp} y^{a} \left(2 <\nabla U,\nabla U_t> \Ga + |\nabla U|^2 \Ga_t\right).
\notag
\end{align}
Using \eqref{g} in the second to the last and in the last integral in the right-end side of \eqref{dh6} we find
\begin{align}\label{e''}
& i'(t) = - \int_{\Rnp} y^{a}  |\nabla U|^2 \Ga  +  \int_{\Rn\times\{0\}} Vu^2 \Ga +2 t\int_{\Rn\times\{0\}} V u u_t \Ga+ t \int_{\Rn\times\{0\}} V_t u^2 \Ga 
\\
& - t \int_{\Rn\times\{0\}} Vu^2 \Delta \Ga - t \int_{\Rnp} y^{a} \left(2 <\nabla U,\nabla U_t> \Ga - |\nabla U|^2 \Delta \Ga\right).
\notag
\end{align}
We want to evaluate the second to the last integral in the right-hand side of \eqref{e''}. We have
\begin{align*}
&  - t\int_{\Rn \times \{0\}} Vu^2 \Delta \Ga  = - t\int_{\Rn \times \{0\}} Vu^2 ( \Delta_{x} \Ga + \Ga_{yy})  
\end{align*}
where $\Delta_{x}$ denotes  the Laplace operator in the $x$ variable. Since
\[
\Ga_{yy} = \left(\frac{y^2}{4t^2} + \frac{1}{2t}\right) \Ga,
\]
we conclude that 
\[
\Ga_{yy}(x,0,t) = \frac{1}{2t} \Ga(x,0,t).
\]
Furthermore, an integration by parts in the $x$ variable, and \eqref{f}, give 
\begin{align*}
&  - t\int_{\Rn \times \{0\}} Vu^2  \Delta_{x} \Ga =  \int_{\Rn \times \{0\}} V u < \nabla u, x>  \Ga + \frac 12 \int_{\Rn \times \{0\}} u^2  <\nabla V,x> \Ga.
\end{align*}
In conclusion, we find
\begin{align*}
& - t\int_{\Rn \times \{0\}} Vu^2 \Delta \Ga = -\frac{1}{2} \int_{\Rn \times \{0\}}  Vu^2 \Ga 
\\
& + \int_{\Rn \times \{0\}} V u < \nabla u, x>  \Ga + \frac 12 \int_{\Rn \times \{0\}} u^2  <\nabla V,x> \Ga.
\notag
\end{align*}
Substituting this equation in \eqref{e''} we finally obtain
\begin{align}\label{e'''}
& i'(t) =  \frac 12 \int_{\Rn\times\{0\}} Vu^2 \Ga +2 t\int_{\Rn\times\{0\}} V u u_t \Ga+ t \int_{\Rn\times\{0\}} V_t u^2 \Ga 
\\
& + \int_{\Rn \times \{0\}} V u < \nabla u, x>  \Ga + \frac 12 \int_{\Rn \times \{0\}} u^2  <\nabla V,x> \Ga
\notag
\\
& - \int_{\Rnp} y^{a}  |\nabla U|^2 \Ga  - t \int_{\Rnp} y^{a} \left(2 <\nabla U,\nabla U_t> \Ga - |\nabla U|^2 \Delta \Ga\right).
\notag
\end{align}
 
Before we proceed further we warn the reader that the computations in \eqref{intermed}-\eqref{e6} below are purely formal. Such computations will be rigorously justified in the subsequent Step $2$. We have chosen to keep such formal intermediate step to better demonstrate the main ideas. 
Keeping this proviso in mind, a formal integration by parts in the last term in the right-hand side of \eqref{e'''} gives 
\begin{align}\label{intermed}
& - t\int_{\Rnp} y^{a} \left(2 <\nabla U,\nabla U_t> \Ga - |\nabla U|^2 \Delta \Ga\right)= - \frac a2  \int_{\Rnp} y^{a} |\nabla U|^2 \Ga 
\\
&- 2 t \int_{\Rnp} y^a <\nabla U,\nabla U_t> \Ga - \int_{\Rnp} y^a <\nabla^2 U (\nabla U), X> \Ga
\notag \\
& - \frac 12 \int_{\Rn\times\{0\}} y^a |\nabla U|^2 <X,e_{n+1}> \Ga, 
\notag
\end{align}
where $\nabla^2 U$ is the Hessian of $U$, and we have denoted by $e_{n+1} = (0,1)$ the unit vector of the standard basis in $\Rn_x\times\R_y$. One should keep in mind that $-e_{n+1}$ is the outer unit normal to $\p \Rnp$. Since $y^a <X,e_{n+1}> = y^{a+1}$ and $a+1>0$, and since respectively from Lemma \ref{reg1} and \eqref{sup} in Remark \ref{R:sup} we have $\nabla_x U, y^a U_y\in L^\infty(\Rnp)$, we see that the boundary integral $\int_{\Rn\times\{0\}} y^a |\nabla U|^2 <X,e_{n+1}> \Ga$ vanishes, and we obtain

\begin{align}\label{e4}
& - t\int_{\Rnp} y^{a} \left(2 <\nabla U,\nabla U_t> \Ga - |\nabla U|^2 \Delta \Ga \right)=  \frac{a}{2t} i(t) - \frac{a}{2} \int_{\Rn} Vu^2 \Ga
\\
& - 2 t \int_{\Rnp} y^a <\nabla U,\nabla U_t> \Ga - \int_{\Rnp} y^a <\nabla^2 U (\nabla U), X> \Ga,
\notag
\end{align}
where we have used \eqref{i} that gives
\[
\frac a2  \int_{\Rnp} y^{a} |\nabla U|^2 \Ga = \frac{a}{2t} i(t) -\frac{a}{2} \int_{\Rn} Vu^2 \Ga.
\] 
Using  \eqref{e4} in \eqref{e'''} we find
\begin{align}\label{e5}
& i'(t) =  \frac{a}{2t} i(t) - \int_{\Rnp} y^{a}  |\nabla U|^2 \Ga 
\\
&- 2 t \int_{\Rnp} y^a <\nabla U,\nabla U_t> \Ga - \int_{\Rnp} y^a <\nabla^2 U (\nabla U), X> \Ga
\notag\\
&  + \frac{1-a}2 \int_{\Rn\times\{0\}} Vu^2 \Ga + 2t\int_{\Rn\times\{0\}} Vuu_t \Ga + t \int_{\Rn\times\{0\}} V_t u^2 \Ga
\notag
\\
&   + \int_{\Rn\times\{0\}} V < \nabla u,x> u \Ga + \frac 12 \int_{\Rn\times\{0\}} <\nabla V,x> u^2\Ga.
\notag
\end{align}

In order to evaluate the third integral in the right-hand side of \eqref{e5} we now note the identity
 \[
 <\nabla^2 U (\nabla U),X>= <\nabla U, \nabla\left(<\nabla U,X>\right)> - |\nabla U|^2.
 \]
Using it we obtain
\begin{align}\label{abov}
& - 2 t \int_{\Rnp} y^a <\nabla U,\nabla U_t> \Ga - \int_{\Rnp} y^a <\nabla^2 U (\nabla U), X> \Ga 
\\
&= -2t \left(\int_{\Rnp}y^a < \nabla U, \nabla U_t> \Ga + y^a <\nabla^2 U(\nabla U), \frac{X}{2t}> \Ga\right)
\notag
\\
&= - 2t \int_{\Rnp}y^a < \nabla U,\nabla\left(U_t + <\nabla U, \frac{X}{2t} >\right)> \Ga  + \int_{\Rnp} y^a |\nabla U|^2 \Ga.
\notag
\end{align}
Formally integrating by parts, and using \eqref{ext} and \eqref{f}, we find
 \begin{align}\label{ab1}
& - 2t \int_{\Rnp}y^a < \nabla U,\nabla\left(U_t + <\nabla U, \frac{X}{2t} >\right)> \Ga=
\\
&=  2t \int_{\Rnp} \operatorname{div}(y^a \nabla U) \left(U_t + <\nabla U, \frac{X}{2t}>\right) \Ga
\notag
\\
& + 2t \int_{\Rnp} y^a <\nabla U, \frac{X}{2t}> \left(U_t + <\nabla U,\frac{X}{2t} >\right)\Ga
\notag
\\
& -2t \int_{\Rn \times \{0\}} \underset{y\to 0}{\lim}\left(y^a< \nabla U, -e_{n+1}>\right) \left(u_t + <\nabla u,\frac{x}{2t}>\right)\Ga.
\notag
\\
& =  2t \int_{\Rnp} 
 y^{a} \left(U_t + <\nabla U,\frac{X}{2t}>\right)^2\Ga -2t  \int_{\Rn\times\{0\}} Vu \left(u_t + <\nabla u,\frac{x}{2t}>\right)\Ga. 
 \notag
 \end{align}
 where again we have used that $\nu = -e_{n+1}$ is the outward unit normal to $\p \Rnp = \Rn\times \{0\}$. Note that in \eqref{ab1} we have also used \eqref{ext}, which gives
 \[
\underset{y\to 0}{\lim}\left( y^a< \nabla U, -e_{n+1}>\right)= - \underset{y\to 0}{\lim} y^a U_y = Vu,
\]
and the fact that  $U_t + <\nabla U, \frac{X}{2t} >$  restricted to $\Rn\times\{0\}$ equals  $u_t + <\nabla u, \frac{x}{2t}>$. These computations will be rigorously justified in Step 2 below. 
 
From \eqref{abov} and \eqref{ab1} we have 
\begin{align}\label{e6}
& - 2 t \int_{\Rnp} y^a <\nabla U,\nabla U_t> \Ga - \int_{\Rnp} y^a <\nabla^2 U (\nabla U), X> \Ga 
\\
& = 2t \int_{\Rnp} 
 y^{a} \left(U_t + <\nabla U,\frac{X}{2t}>\right)^2\Ga -2t  \int_{\Rn\times\{0\}} Vu \left(u_t + <\nabla u,\frac{x}{2t}>\right)\Ga
\notag \\
& + \int_{\Rnp} y^a |\nabla U|^2 \Ga.
 \notag
\end{align}

Substituting \eqref{e6} into \eqref{e5} we finally obtain \eqref{e7}.

\medskip

\noindent \textbf{Step 2:}  We now  show how to justify the above formal computations \eqref{intermed}-\eqref{e6} leading to \eqref{e7}. With this objective in mind, similarly to the proof of Lemma \ref{L:ibi} we perform the said computations over the interior regions $\Rn\times \{y > \ve \}$, instead of $\Rnp$, and then pass to the limit as $\ve \to 0$.
In this process, the finiteness of the integrals on $\Rn\times \{y > \ve \}$ involving $\nabla^2 U$ can be established using the Poisson representation \eqref{pos}. From the latter we have in fact the following estimate 
\[
||U_{yy}(\cdot,y,\cdot)||_\infty \leq \frac{C}{y^{2}},
\]
which implies that $U_{yy}$ is bounded in $\Rn\times \{y > \ve \}$. 
Similarly, we have  for $i=1, ...n$,
\[
||(U_{i})_y(\cdot,y,\cdot)||_\infty \leq \frac{C}{y}.
\]
Finally, from \eqref{pos} and from our qualitative assumption \eqref{qasump}, we obtain that $U_{ij} \in L^{\infty}$ for $i, j=1,...,n$. Proceeding from \eqref{intermed} to \eqref{e6}, when passing to the limit as $\ve\to 0^+$ we need to worry about two terms.
The former is the following  boundary integral in \eqref{intermed}, 
\[
-\frac 1 2 \int_{\Rn\times \{y =\ve \}} \ve^{1+a} |\nabla U|^2 \Ga,
\]
which in view of \eqref{sup} goes to zero as $\ve\to 0^+$. In \eqref{e6} an analogous integration by parts on the region $\Rn\times \{y > \ve \}$ will produce the boundary integral  
 \[
\int_{\Rn\times \{y =\ve \}} y^a U_y\left(U_t+ <\nabla U,X>\right)\Ga.
 \] 
Using the H\"older continuity of $y^{a} U_y$ and the estimate \eqref{sup}, we see that this integral converges to  $\int_{\Rn\times \{y =0\}} Vu \left(u_t + <\nabla u,x>\right)\Ga$ as $\ve \to 0^+$. After passing to the limit as $\ve\to 0^+$ we end up with \eqref{e7}, which is now rigorously justified under the hypothesis \eqref{qasump}.

\medskip

\noindent \textbf{Step 3:} In this last part of the proof we remove the hypothesis \eqref{qasump} and show how to establish \eqref{e7} 
under the sole assumption that $u\in  \operatorname{Dom}(H^{s})$ be a solution to \eqref{e00} with $V$ satisfying \eqref{vasump}. In order to do so, we  first note that  from Lemma \ref{regU} we have that $y^aU_y$ is H\"older continuous up to $y=0$.  Moreover, from the hypoellipticity of the extension operator in \eqref{ext}, we know that $U$ is smooth for $\{y>0\}$. 
Having said this, given  $\ve >0$ we let $\tau_{\ve}$ be a $C^\infty_0$ function of the $X=(x,y)$ variable such that $\tau_\ve \equiv 1$ in $\B_{1/\ve}$, and $\tau_\ve \equiv 0$ outside $\B_{2/\ve}$. We then  define
\[
i_{\ve} (t)= -t \int_{\Rn\times \{y > \ve \}} y^a |\nabla U|^2 \Ga \tau_{\ve} + t \int_{\Rn \times \{0\}} Vu^2 \Ga,
\] 
and note that  
\[
|i_{\ve} (t) - i(t)| \leq |t| \int_{\Rn\times \{0<y < \ve \}} y^a |\nabla U|^2 \Ga + |t| \int_{\Rnp \cap \{|X| >1/\ve\}} y^a |\nabla U|^2 \Ga.
\]
Now for a given $\delta >0$ and $t \in [-1,-\delta]$, from \eqref{sup} and Lemma \ref{reg1} we have 
\[
|t| \int_{\Rn\times \{0<y < \ve \}} y^a |\nabla U|^2 \Ga \leq  C |t| \int_{\Rn\times \{0<y < \ve \}}  ( y^a + y^{-a}) \Ga, 
\]
and
\[
 |t| \int_{\Rnp \cap \{|X| >1/\ve\}} y^a |\nabla U|^2 \Ga \leq C |t| \int_{\Rnp \cap \{|X| > 1/\ve\}} (y^a+ y^{-a}) \Ga.
 \]
The change  of variable
\[
X' = \frac{X}{\sqrt{|t|}}
\] 
gives for $t \in [-1,-\delta]$, 
\begin{align}\label{Arg1}
&|t| \int_{\Rn\times \{0<y < \ve \}} (y^a + y^{-a}) \Ga  \leq  C \left(|t|^{1+ \frac a2} \int_0^{\frac{\ve^2}{4 \delta}} \tau^{-\frac{1-a}{2}} +|t|^{1- \frac a2} \int_0^{\frac{\ve^2}{4 \delta}} \tau^{-\frac{1+a}{2}} \right) \longrightarrow 0,
\end{align}
uniformly as $\ve \to 0^+$.
Similarly, we see that for $t \in [-1,-\delta]$
\begin{align}
& |t| \int_{\Rnp \cap \{|X| > 1/\ve\}} (y^a+ y^{-a}) \Ga \longrightarrow 0,
\end{align}
uniformly as $\ve \to 0^+$. It follows that $i_{\ve} \to i$, uniformly in $[-1,-\delta]$. At this point we perform on $i_\ve(t)$ computations similar to those that in Step $1$ have led to \eqref{e7}, and conclude that

\begin{align}\label{trunc}
i_{\ve}'(t) & = \frac{a}{2t} i_{\ve}(t) +  2t \int_{\Rn\times \{y > \ve \}}y^{a} \left(U_t + <\nabla U, \frac{X}{2t}>\right)^2  \Ga \tau_{\ve}
\\
&  +\frac{1-a }{2} \int_{\Rn \times \{ 0\}} Vu^2 G + t \int_{ \Rn \times \{0\}} V_t u^2 \Ga + \frac 12  \int_{\Rn \times \{0\}} <\nabla V,x> u^2 \Ga
 \notag
 \\
 & + \int_{\Rn \times \{0\}} V u <\nabla u,x>  \Ga  +2t \int_{\Rn\times \{y = \ve \}} y^a U_y\left(U_t+  <\nabla U, \frac{X}{2t}>\right)  \Ga  \tau_{\ve}
   \notag
 \\
&- \frac 12 \int_{\Rn\times \{y > \ve \}}y^a |\nabla U|^2 < \nabla  \tau_\ve,X> \Ga - \frac{1}{2} \int_{\Rn\times \{y = \ve \}} y^{a+1} |\nabla U|^2  \Ga \tau_{\ve} 
 \notag
 \\
 & +    2t \int_{\Rn\times \{y > \ve \}} y^a \left(U_t + <\nabla U, \frac{X}{2t}>\right) < \nabla \tau_\ve, \nabla U> \Ga
 \notag\\
 & + 2t\int_{\Rn \times \{0\}} V u u_t \Ga. 
\notag
 \end{align}
In deriving \eqref{trunc} we have crucially used the fact that, in view of the hypoellipticity of the extension operator in \eqref{ext}, the functions $\nabla U_t, \nabla^2 U$ are bounded in the region $\B_{2/\ve}^{+} \cap \left(\Rn\times \{y > \ve \}\right)$. As a consequence, the counterparts of some of the intermediate calculations in \eqref{intermed}-\eqref{e7} above are justified. From \eqref{est50} above and Lemma \ref{reg1}, similarly to \eqref{Arg1} we obtain as $\ve \to 0$ 
 \begin{align*}
 & 2t \int_{\Rn\times \{y > \ve \}}y^{a} \left(U_t + <\nabla U, \frac{X}{2t}>\right)^2  \Ga \tau_{\ve}\  \longrightarrow\  2t \int_{\Rnp}y^{a} \left(U_t + <\nabla U,  \frac{X}{2t}>\right)^2 \Ga,
 \end{align*}  
uniformly in $ t \in [-1,-\delta]$. Similarly, the term 
   \[
   -\frac{1}{2} \int_{\Rn\times \{y = \ve \}} y^{a+1} |\nabla U|^2 \Ga \tau_{\ve} 
   \]
    goes to zero uniformly in $\ve$ for $t \in [-1, -\delta]$. Now  from the assumptions on $V, u$ and Lemma \ref{reg2} we infer the existence of $\alpha', C >0$ such that for $\Rn\times \{y = \ve \}$ the following holds 
 \[
 |-y^a U_y - Vu| \leq C\ve^{\alpha'},\ \ \ |U_i- u_i| <C\ve^{\alpha'},\ \ \  |U_t - u_t| \leq C\ve^{\alpha'},\ i=1,...n.
\]
These estimates, coupled with \eqref{sup}, Lemma \ref{reg1} and the change of variable $X' = \frac{X}{\sqrt{|t|}}$, imply the following uniform convergence as $\ve \to 0$, for $t\in [-1, -\delta]$, 
\[
2t \int_{\Rn\times \{y = \ve \}} y^a U_y\left(U_t+  <\nabla U, \frac{X}{2t}>\right)  \Ga \tau_{\ve}\  \longrightarrow\  -2t \int_{\Rn \times \{0\}} Vu (u_t + <\nabla u,x/2t>)\Ga. 
\]
Similarly, since $|\nabla \tau_{\ve}| \leq C \ve$ for some universal $C$, and $\nabla \tau_{\ve}$ is supported in $\frac{1}{\ve} \leq |X| \leq \frac{2}{\ve}$, we have 
\[
|<\nabla  \tau_{\ve},X>| \le C.
\]
This implies that the integrals
\[
- t \int_{\Rn\times \{y > \ve \}} y^a |\nabla U|^2 < \nabla \tau_\ve,\frac{X}{2t}> \Ga,
\]
and
\begin{equation}\label{trunc1}
2t \int_{\Rn\times \{y > \ve \}} y^a \left(U_t+  <\nabla U, \frac{X}{2t}>\right) <\nabla \tau_\ve,\nabla U> \Ga
\end{equation}
converge to $0$ uniformly  in $t\in [-1, -\delta]$ as $\ve \to 0$. Finally, given the uniform convergence of $i_{\ve}, i_{\ve}'$ in $[-1,-\delta]$ and the arbitrariness of $\delta$, we conclude that under  the more general assumptions on $u, V$ as in Theorem \ref{main},  $i'(t)$ is given by the right-hand side in \eqref{e7}. 

\end{proof}

\begin{lemma}\label{L:trace}
There exists $0<t_0<1$, depending only on $n, s$ and $K$ in \eqref{vasump}, such that for $-t_0 \leq  t < 0$ one has
\begin{equation}\label{e9'}
\int_{\Rn \times \{0\}} u^2 \Ga = \int_{\Rn} u(x,t)^2 \Ga(x,0,t) dx \leq C |t|^{s-1} \left(i(t) + h(t)\right).
\end{equation}
When $\frac{1}{2}\le s<1$ one also has 
\begin{equation}\label{e91}
\left |\int_{\Rn \times \{0\}} <\nabla V(x, t), x> u(x, t)^2 \Ga(x, 0, t) \right| \leq C |t|^{s-1} \left(i(t) + h(t) \right).
\end{equation}
\end{lemma}

\begin{proof}
We first establish \eqref{e9'}. At each time level $t$ we apply the trace inequality in Lemma \ref{tr} to the function $f= U \Ga^{1/2}$ with $\mu= |t|^{-1/2}$. Of course, such $f$ is not compactly supported. Nevertheless, the application of the trace inequality can be justified by an approximation argument using cut-off functions which we omit. Since by \eqref{f} one has
\[
\nabla(U\Ga^{1/2}) = \Ga^{1/2} \nabla U + \Ga^{1/2} U \frac{X}{4t},
\]
we obtain
\begin{align}\label{ei}  
\int_{\Rn \times \{0\}} u^2 \Ga & \leq C \bigg(|t|^{s-1} \int_{\Rnp} y^a U^2 \Ga + |t|^{s} \int_{\Rnp}  y^{a} U^2 \frac{|X|^2}{16t^2} \Ga
\\
& + |t|^{s} \int_{\Rnp} y^a  |\nabla U|^2 \Ga\bigg).
\notag
\end{align}

We now claim that
\begin{equation}\label{iner5}
\int_{\Rnp} y^a U^2 \frac{|X|^2} {16t^2 } \Ga \leq C \left(\frac{1}{|t|} \int_{\Rnp} y^a U^2 \Ga + \int_{\Rnp} y^a |\nabla U|^2 \Ga\right).
\end{equation}
To prove \eqref{iner5} we first observe that
 \begin{equation*}\label{id}
\int_{\Rnp} y^a U^2 \frac{|X|^2} {16 t^2 } \Ga = \frac{1}{8 t}  \int_{\Rnp} y^a U^2  < X,  \nabla \Ga>.
\end{equation*}
Integrating by parts we find
\begin{align}\label{iner}
& \frac{1}{8 t}  \int_{\Rnp} y^a U^2  <X,\nabla \Ga>   = - \frac{1}{8t} \int_{\Rnp} \left(y^ a (\operatorname{div} X) U^2 \Ga   + ay^a U^2 \Ga  + 2U <\nabla U, X> \Ga\right)
\\
\notag
& = -\frac{n+1 + a}{8t} \int_{\Rnp} y^a U^2 \Ga  - \frac{1}{8t} \int_{\Rnp} y^a 2 U < \nabla U, X> \Ga.
\notag
\end{align}
We note that \eqref{iner} can be justified by first integrating by parts on the region $\Rn \times \{y> \ve\}$, and then passing to the limit as $\ve \to 0$. In this process, one obtains the following boundary integral 
\[
-\frac{1}{8t} \int_{\Rn \times \{y=\ve\}} y^{a+1} U^2 \Ga,
\]
which is easily seen to tend to $0$ as $\ve \to 0$. Next, we apply the numerical inequality 
$ab \leq \delta a^2 + \frac{1}{\delta} b^2$,
with  $a=  2 |\nabla U|$, $b= |U| |X|$ and $\delta= 4|t|$, obtaining 
\begin{align}\label{iner1}
& \left| \frac{1}{8t} \int_{\Rnp} y^a 2 U < \nabla U, X> \Ga\right| \leq  \int_{\Rnp} y^a U^2\frac{|X|^2}{32t^2}  \Ga +  C \int_{\Rnp} y^a |\nabla U|^2 \Ga.
\end{align}
Combining \eqref{iner} with \eqref{iner1} we easily conclude that \eqref{iner5} holds. Using now the inequality \eqref{iner5} in \eqref{ei}, we conclude  for some $C>0$, depending on $n, s$ and $K$ in \eqref{vasump},  
\begin{align}\label{ei1}
&\int_{\Rn \times \{0\}} u^2 \Ga \leq C \left(|t|^{s-1} \int_{\Rnp} y^a U^2 \Ga + |t|^{s} \int_{\Rnp} y^a  |\nabla U|^2 \Ga\right).
\end{align}

Recalling \eqref{h1}, \eqref{i}, 
we finally have from \eqref{ei1}
\begin{align}\label{eii}
& \int_{\Rn \times \{0\}} u^2 \Ga \leq C |t|^{s-1} \left( i(t) + h(t) + |t| \int_{\Rn \times \{0\}} V u^2 \Ga\right)
\\
& \le C |t|^{s-1} \left( i(t) + h(t) + K |t| \int_{\Rn \times \{0\}} u^2 \Ga\right),
\notag
\end{align}
where in the last inequality we have used \eqref{vasump}. We now choose $t_0>0$ such that
\[
C K t_0^{s} \leq 1/2.
\]
For $-t_0\le t < 0$ we thus obtain the desired conclusion \eqref{e9'} from \eqref{eii}. 

We next establish \eqref{e91} for $\frac{1}{2}\le s < 1$. We first split  the integral in the left-hand side of \eqref{e91} as follows:
 \begin{align*}
&  \int_{\Rn} <\nabla V(x, t), x> u(x, t)^2 \Ga(x, 0, t) =  \int_{\Rn \cap \{|x|  \leq 1\}} <\nabla V(x, t), x> u(x, t)^2 \Ga(x, 0, t) 
\\
& +  \int_{\Rn \cap \{|x| >1\}}  <\nabla V(x, t), x> u(x,t)^2 \Ga(x,0,t). 
\notag
\end{align*} 
Since $|\nabla_x V| \leq K$  by \eqref{vasump}, for $|x|<1$ obtain $\left|<\nabla_x V, x>\right| \leq K$.
It follows  that the first integral in the right-hand side of the latter inequality can be bounded from above  by $K \int_{\Rn} u(x,t)^2 \Ga(x,0,t)$, which in turn can be  estimated by \eqref{e9'}. Consequently, we have 
\begin{equation*}
 \int_{\Rn \cap \{|x|  \leq 1\}} <\nabla_{x} V(x,t),x> u(x,t)^2 \Ga(x,0,t) \leq C|t|^{s-1} \left( i(t) + h(t) \right).
 \end{equation*}
To complete the proof of the lemma we are thus left with estimating the second integral in the right-hand side. Because of \eqref{vasump} again, we have
\begin{align*}
&\int_{(\Rn\times \{0\}) \cap \{|x| >1\}}  <\nabla V, x> u^2 \Ga \leq K \int_{\Rn \cap \{|x| >1\}}  |x| u(x,t)^2 \Ga(x,0,t)
\\
&\ \ \ \ \ \ \ \ \text{(passing to spherical coordinates, and letting $f(X)= U(X,t) \Ga^{1/2}(X,t)$)}
\notag\\
& = K \int_{1}^{\infty} r^{n} \int_{\mathbb{S}^{n-1}} u(r\omega',t)^2 \Ga(r\omega',0,t) d\omega'  dr = K \int_{1}^{\infty} \int_{\mathbb{S}^{n-1}} r^n f(r\omega', 0)^2 d\omega' dr.
\notag
\end{align*}
For any fixed $r>1$, we next apply the trace inequality in Lemma \ref{tr5} to $ g(\omega)= f(r\omega)$ with $\tau= |t|^{-\frac{1}{2}} r^{\frac{2-a}{2s}}$. Keeping in mind that $a = 1 - 2s$, we  find
\begin{align*}
&  \int_{\mathbb{S}^{n-1}} r^n f(r\omega',0)^2 d\omega'  \leq C \bigg(|t|^{s-1} r^{n+(2-a)\frac{1-s}{s}}\int_{\mathbb{S}^{n}_{+}} \omega_{n+1}^{a} f(r\omega)^2  d\omega
  \\
  & + |t|^s r^{n+a-2} \int_{\mathbb{S}^{n}_{+}} \omega_{n+1}^a |\nabla_{\mathbb{S}^{n}} f(r\omega)|^2 d\omega\bigg).
\notag
\end{align*}
Keeping in mind that $|\nabla f|^2 = f_r^2 + \frac{1}{r^2} |\nabla_{\mathbb S^{n}} f|^2$, which gives in particular $\frac{1}{r^2} |\nabla_{\mathbb S^{n}} f|^2\le |\nabla f|^2$,
we see that the latter inequality trivially implies
\begin{align*}
&  \int_{\mathbb{S}^{n-1}} r^n f(r\omega',0)^2 d\omega'  \leq C \bigg(|t|^{s-1} r^{n+(2-a)\frac{1-s}{s}}\int_{\mathbb{S}^{n}_{+}} \omega_{n+1}^{a} f(r\omega)^2  d\omega
 + |t|^s r^{n+a} \int_{\mathbb{S}^{n}_{+}} \omega_{n+1}^a |\nabla f(r\omega)|^2 d\omega\bigg).
\end{align*}
We now make the simple, yet crucial observation that $1>s \geq \frac{1}{2}$ implies
\[
 n + (2-a)\frac{1-s}{s} = n+a + \frac{1}{s} \le n + a + 2.
\]
Since $r>1$, this gives $r^{n + (2-a)\frac{1-s}{s}} \leq r^{n+a + 2}$. We thus find from the above inequality
\begin{align*}
&  \int_1^\infty \int_{\mathbb{S}^{n-1}} r^n f(r\omega',0)^2 d\omega' dr \leq C \bigg(|t|^{s-1} \int_1^\infty r^{n+a + 2} \int_{\mathbb{S}^{n}_{+}} \omega_{n+1}^{a} f(r\omega)^2  d\omega dr
\\
& + |t|^s \int_1^\infty r^{n+a} \int_{\mathbb{S}^{n}_{+}} \omega_{n+1}^a |\nabla f(r\omega)|^2 d\omega dr\bigg).
\end{align*}
Returning to Euclidean coordinates, and keeping in mind that $X = r \omega = (r\omega',r\omega_{n+1})$, we finally have recalling the definition of $f$
\begin{align}\label{ei1000}  
\int_{(\Rn \times \{0\}) \cap \{|x|>1\}} <\nabla V,x>  u^2 \Ga & \leq C \bigg( |t|^{s} \int_{\Rnp}  y^{a} U^2 \frac{|X|^2}{|t|^2} \Ga
+ |t|^{s} \int_{\Rnp} y^a  |\nabla U|^2 \Ga\bigg).
\end{align}
Using \eqref{iner5} in \eqref{ei1000}, and recalling \eqref{h1}, \eqref{i}, we conclude 
\begin{equation*}
\int_{\Rn \times \{0\} \cap \{|x|>1 \}} <\nabla V, x>  u^2 \Ga \leq C |t|^{s-1} \left(i(t) + h(t) + |t| \int_{\Rn \times \{0\}} V u^2 \Ga\right).
\end{equation*}
In view of \eqref{vasump} and \eqref{e9'}, the last term in the right-hand side of the latter ienquality is estimated as follows
\begin{equation*}
|t| \int_{\Rn \times \{0\}} V u^2 \Ga \leq K |t| \int_{\Rn \times \{0\} }Vu^2 \Ga \leq C |t|^{s} (h(t) + i(t)). 
\end{equation*}
To finish, with $C>0$ as in the right-hand side of the latter inequality, we choose $t_0$ such that
\[
C |t_0|^s \leq 1.
\]
Then, for $-t_0<t<0$ we reach the desired conclusion 
\[
\int_{(\Rn \times \{0\}) \cap \{|x|>1 \}} <\nabla V, x>  u^2 \Ga \leq C|t|^{s-1} \left(i(t) + h(t) \right).
\]

\end{proof}

\begin{rmrk}\label{R:pos}
For later use we note explicitly that Lemma \ref{L:trace} implies in particular that $i(t) + h(t) \ge 0$ for $-t_0\le t\le 0$. Integrating such inequality we obtain $I(r) + H(r) \ge 0$ for $0\le r \le r_0 = \sqrt{t_0}$. This guarantees, in turn, that $N(r) + 1 \ge 0$, and therefore that the frequency is bounded from below for $0\le r \le r_0$. In Remark \ref{R:imp} below we will show that, as a consequence of Theorem \ref{mon1}, the frequency is also bounded from above. 
\end{rmrk}

\begin{rmrk}\label{R:trace}
If $h(\overline t) = 0$ for some $-t_0\le \overline t < 0$, then we must have $u(x,\overline t) = 0$ for every $x\in \Rn$. In fact, from  $h(\overline t) = 0$ and \eqref{h1} we infer $U(X,\overline t) \equiv 0$ for $X\in \Rnp$. It then follows from the continuity of $U$ upto $\{y=0\}$  that $u(x, \overline t) =0$ for every $x \in \Rn$. 
\end{rmrk}

Using Lemmas \ref{L:fvei} and \ref{L:trace} we can now establish the following fundamental result.

\begin{lemma}\label{L:fve}
There exists $t_0 = t_0(n,s,K)>0$, where $K$ is as in \eqref{vasump}, such that for $r \leq r_0= \sqrt{t_0}$ and $Z$ as in \eqref{z1}, one has   
\begin{align}\label{e8}
& I'(r) \geq  \frac{a}{r} I(r) + \frac{1}{r^3} \int_{\Sa_{r}^{+}} y^{a} (ZU)^2 \Ga dX dt - C r^{-a} I(r)  - C r^{-a} H(r).
\end{align}
\end{lemma}

\begin{proof}

We start from the expression $I(r) = \int_{-1}^{0} i(r^2 t) dt$ in \eqref{HI}. Our first objective is proving the following identity 
\begin{equation}\label{dh5}
I'(r)=  2r \int_{-1}^0 t i'(r^2 t) dt.
\end{equation}
To this end we note that, similarly to what was done in the computation of $H'$ in Lemma \ref{L:fvh}, the formal differentiation in \eqref{dh5} can be justified in the following way from the expression of $i'$ in \eqref{e7}. For a given $\ve>0$, we let
\[
I_{\ve}(r)= \frac{1}{r^2} \int_{-r^2}^{-\ve r^2} i(t) dt.
\]
Now fix $\delta>0$ arbitrarily. We claim that as $\ve \to 0^+$:
\begin{itemize}
\item[(i)] $\underset{r\in [\delta,1]}{\sup} \left|I_\ve(r) - I(r)\right| \longrightarrow 0$;
\item[(ii)] $\underset{r\in [\delta,1]}{\sup} \left|I_{\ve}'(r) - 2r \int_{-1}^{0}  t i'(r^2 t)dt\right| \longrightarrow 0$.
\end{itemize}
Taking the claim for granted, from it we infer that \eqref{dh5} does hold for $r \in [\delta,1]$. The arbitrariness of $\delta$ implies that \eqref{dh5} holds for $ r \in (0, 1]$. We are thus left with proving the claim.

With this objective in mind, a change of variable gives
\[
I_{\ve}(r)= \int_{-1}^{-\ve} i(r^2 t) dt.
\]
Similarly to the computation of $H_{\ve}'$, for a given $\delta>0$ and for $r \in [\delta, 1]$, using the expression of $i'$ in \eqref{e7}, Lemma \ref{reg1}, \eqref{sup}, we can differentiate under the integral sign and deduce that
\[
I_{\ve}'(r)=  2r \int_{-1}^{-\ve} t i'(r^2t)dt.
\]
Now, from \eqref{sup} and the estimate in Lemma \ref{reg1}, we have
\[
|\nabla_{x} U| \leq C,\ \ \  |U_y| \leq \frac{C}{y^a}.
\]
Using such bounds, the boundedness of $u$, \eqref{vasump} and the change of variable $X' = \frac{X}{\sqrt{|t|}}$, we obtain as $\ve \to 0$ 
\begin{align*}
&|I_{\ve}(r) -I(r)| = \int_{-\ve}^{0} i(r^2 t) dt = \int_{-\ve}^{0} \int_{\Rnp} -r^2 t y^a |\nabla U(X, r^2t)|^2 \Ga (X, r^2t)  
\\
& + \int_{-\ve}^{0} \int_{\Rn} -r^2 t Vu(x, r^2 t) \Ga(x, 0, r^2t)
\notag
\\
& \leq C \int_{-\ve}^{0} \left[(r^2|t|)^{1+a/2}  +  (r^2|t|)^{1-a/2} + (r^2|t|)^{1/2}\right] dt\ \longrightarrow\ 0,
\end{align*}
uniformly for $r \in [\delta,1]$. This establishes (i).

Similarly, using the expression of $i'$ in \eqref{e7},  \eqref{sup}, Lemma \ref{reg1}, \eqref{est50} and the change  of variable $X' = \frac{X}{\sqrt{|t|}}$, we find as $\ve \to 0$
\[
|I_{\ve}' (r) - 2r \int_{-1}^{0} t i'(r^2t)dt| \leq \frac{C}{r}  \int_{-\ve}^{0} \left[(r^2|t|)^{1+a/2}  + (r^2|t|)^{1-a/2}  + (r^2|t|)^{1/2}  + (r^2|t|)^{a/2}\right]\ \longrightarrow\ 0,
\]
uniformly  for $r \in [\delta,1]$. This proves (ii), and therefore the claim. We have thus established \eqref{dh5}.

With \eqref{dh5} in hand we note that, in order to establish \eqref{e8}, we will need to obtain an estimate from above for $i'(t)$ in \eqref{e7}. We consider two cases:
\begin{itemize}
\item[(i)] When $0<s< \frac{1}{2}$ we know from \eqref{vasump} that $V, V_t, <\nabla V,x> \in L^\infty$. Therefore, for $-1<t<0$ we have from \eqref{e7} 
\begin{align}\label{e89}
& i'(t) \leq \frac{a}{2t} i(t) + 2t \int_{\Rnp} 
 y^{a} \left(U_t + < \nabla U,\frac{X}{2t}>\right)^2\Ga  + \tilde{K} \int_{\Rn\times\{0\}} u^2 \Ga,
\end{align}
where $\tilde{K}$ depends on the constant $K$ in \eqref{vasump}, as well as on $n$ and $s$.  
\item[(ii)] If instead $\frac{1}{2}\le s<1$, then \eqref{e7} and \eqref{vasump} imply
\begin{align}\label{e892}
& i'(t) \leq \frac{a}{2t} i(t) + 2t \int_{\Rnp} 
 y^{a} \left(U_t + < \nabla U,\frac{X}{2t}>\right)^2\Ga  + \tilde{K} \int_{\Rn\times\{0\}} u^2 \Ga
 \\
 &+\frac{1}{2} \left |\int_{\Rn \times \{0\}} <\nabla V,x> u^2 \Ga \right|
 \notag
 \end{align}
 \end{itemize}
Finally, we show that \eqref{e89} in case (i), and \eqref{e892} in case (ii) imply the sought for conclusion \eqref{e8}. To see this, in case (i) we use \eqref{e9'} in Lemma \ref{L:trace} to estimate the boundary integral $\int_{\Rn\times\{0\}} u^2 \Ga$ in \eqref{e89}, obtaining for $-t_0\le t < 0$    
\begin{align}\label{e890}
& i'(t) \leq \frac{a}{2t} i(t) + C |t|^{s-1} i(t)  + 2t \int_{\Rnp} 
 y^{a} \left(U_t + < \nabla U,\frac{X}{2t}>\right)^2\Ga  + C |t|^{s-1} h(t).
\end{align}
In case (ii) instead, we use \eqref{e9'}, \eqref{e91} from Lemma \ref{L:trace} in \eqref{e892}. Again, we conclude that \eqref{e890} holds. 
 The inequality \eqref{e8} now follows in a standard way from \eqref{e890} if we use \eqref{dh5}, and keep in mind that $a = 1 - 2s$ and that $\int_{-1}^0 |t|^s h(r^2t) dt \le \int_{-1}^0 h(r^2 t) dt = H(r)$. 

\end{proof}
  
With  Lemma \ref{L:fve} in hand,  we can  now establish our main  result of this section, Theorem \ref{mon1}. 

\begin{proof}[Proof of Theorem \ref{mon1}]
From the definition \eqref{N100} of $N(r)$ we obtain for $H(r) \neq 0$
\begin{equation*}
N'(r) = \frac{I'(r)}{ H(r)}- \frac{H'(r)}{H(r)} N(r).
\end{equation*}
Using \eqref{mon10} in Lemma \ref{L:fvh}, \eqref{e8} in Lemma \ref{L:fve} and \eqref{e91}, we obtain for some universal $C = C(n,s,K)>0$, and for $0<r\le r_0$
\begin{align*}
N'(r) & \geq  \frac{1}{r} \frac{\int_{\Sa_{r}^{+}} y^a (ZU)^2 \Ga}{\int_{\Sa_{r}^{+}} y^a U^2 \Ga}  - \frac{4}{r} N(r)^2 - Cr^{-a} N(r) - Cr^{-a}.
\end{align*}
By \eqref{e91} in Lemma \ref{L:I} we have
\[
\frac{4}{r} N(r)^2 = \frac{4}{r} \frac{I(r)^2}{H(r)^2} = \frac{1}{r} \frac{\left(\int_{\Sa_{r}^{+}} y^a U ZU \Ga\right)^2}{\left(\int_{\Sa_{r}^{+}} y^a U^2 \Ga\right)^2}.
\]
Substituting this identity in the previous inequality, we find
\begin{align*}
N'(r) & \geq  \frac{1}{r} \frac{\int_{\Sa_{r}^{+}} y^a (ZU)^2 \Ga}{\int_{\Sa_{r}^{+}} y^a U^2 \Ga}  - \frac{1}{r} \frac{\left(\int_{\Sa_{r}^{+}} y^a U ZU \Ga\right)^2}{\left(\int_{\Sa_{r}^{+}} y^a U^2 \Ga\right)^2}- Cr^{-a} N(r) - Cr^{-a}.
\end{align*}
An application of Cauchy-Schwarz inequality gives for $0<r < r_0$
\[
\frac{\int_{\Sa_{r}^{+}} y^a (ZU)^2 \Ga}{\int_{\Sa_{r}^{+}} y^a U^2 \Ga}  - \frac{\left(\int_{\Sa_{r}^{+}} y^a U ZU \Ga\right)^2}{\left(\int_{\Sa_{r}^{+}} y^a U^2 \Ga\right)^2} \ge 0.
\]
We thus finally obtain
\begin{align*}\label{e13}
& N'(r)  \geq  - Cr^{-a} - Cr^{-a} N(r).
\end{align*}
If we set $\psi(r) = \int_0^r t^{-a} dt$, the latter inequality can be written
\[
N'(r)  \geq  - C \psi'(r) - C \psi'(r) N(r) ,\ \ \ \ \ \  0<r<r_0.
\]
This implies for $0<r<r_0$
\begin{align*}
\frac{d}{dr} e^{C\psi(r)}\left(N(r) + C \psi(r)\right) & = e^{C\psi(r)}\left(N'(r) + C \psi'(r) + C \psi'(r) N(r) + C^2 \psi(r) \psi'(r)\right)
\\
& \ge C^2 e^{C\psi(r)} \psi(r) \psi'(r) \ge 0,
\end{align*}
from which the desired conclusion readily follows. 

To prove the second part of the theorem, assume now that $V\equiv 0$. In such case, from \eqref{I100} we have
\begin{equation}\label{II}
I(r)  =  \frac{1}{r^2} \int_{\Sa_{r}^{+}} |t|  y^a |\nabla U|^2 \Ga = \frac{1}{2r^2} \int_{\Sa_{r}^{+}} y^a U ZU \Ga,
\end{equation}
where in the second equality we have used \eqref{e91} in Lemma \ref{L:I}.
Furthermore, \eqref{e8} in Lemma \ref{L:fvh} gives
\[ 
I'(r) \geq  \frac{a}{r} I(r) + \frac{1}{r^3} \int_{\Sa_{r}^{+}} y^{a} (ZU)^2 \Ga.
\]
We conclude that
\begin{align}\label{monomono}
N'(r) & \geq  \frac{1}{r} \frac{\int_{\Sa_{r}^{+}} y^a (ZU)^2 \Ga}{\int_{\Sa_{r}^{+}} y^a U^2 \Ga}  - \frac{1}{r} \frac{\left(\int_{\Sa_{r}^{+}} y^a U ZU \Ga\right)^2}{\left(\int_{\Sa_{r}^{+}} y^a U^2 \Ga\right)^2} \ge 0,
\end{align}
by Cauchy-Schwarz inequality. This proves the monotonicity of $r\to N(r)$. Suppose now that $U$ be homogeneous of degree $2\kappa$ in $\Sa_{R}^{+}$ with respect to the parabolic dilations \eqref{pardil}. Then, \eqref{euler} gives $ZU = 2\kappa U$. We thus find from \eqref{II}
\[
I(r)  = \frac{2\kappa}{2r^2} \int_{\Sa_{r}^{+}} y^a U^2 \Ga = \kappa H(r).
\]
This gives for $0<r<R$
\[
N(r) = \frac{I(r)}{H(r)} \equiv \kappa.
\]
Vice-versa, if this equation holds for $0<r<R$, then $N' \equiv 0$ in $(0,R)$. Combining this with \eqref{monomono}, we obtain
\begin{align*}
0 \equiv N'(r) & \geq  \frac{1}{r} \frac{\int_{\Sa_{r}^{+}} y^a (ZU)^2 \Ga}{\int_{\Sa_{r}^{+}} y^a U^2 \Ga}  - \frac{1}{r} \frac{\left(\int_{\Sa_{r}^{+}} y^a U ZU \Ga\right)^2}{\left(\int_{\Sa_{r}^{+}} y^a U^2 \Ga\right)^2} \ge 0.
\end{align*}
This means that there must be equality in the Cauchy-Schwarz's inequality, and therefore for every $r\in (0,R)$ there exists $\alpha(r)$ such that $ZU = \alpha(r) U$ in $\Sa_{r}^{+}$. Using this information in \eqref{II}, we find
\[
I(r) =  \frac{\alpha(r)}{2r^2} \int_{\Sa_{r}^{+}} y^a U^2 \Ga = \frac{\alpha(r)}{2} H(r),
\]
or, equivalently, $N(r) = \frac{\alpha(r)}{2}$. This implies $\alpha(r) \equiv  2\kappa$ in $(0,R)$, and thus $ZU = 2\kappa U$ in $\Sa_{R}^{+}$.

\end{proof}

\begin{rmrk}\label{R:imp}
It is important to observe that Theorem \ref{mon1} implies, in particular, the boundedness from above of the frequency $N(r)$ for $r\leq r_0$. This follows from the fact that, since $|a|<1$, we have $\psi(r) = \int_{0}^{r} t^{-a} < \infty$ for every $r>0$, and moreover
\[
N(r) \le e^{C\psi(r)} \left(N(r) + C \psi(r)\right) \le e^{C\psi(r_0)} \left(N(r_0) + C \psi(r_0)\right),
\]
by the monotonicity of $r\to e^{C\psi(r)} \left(N(r) + C \psi(r)\right)$. On the other hand, we have observed in Remark \ref{R:pos} that $N(r)$ is bounded from below on $(0,r_0)$. Thus, $N\in L^\infty(0,r_0)$.
\end{rmrk}

Keeping in mind \eqref{H} and \eqref{I100}, which give
\[
\int_{\Sa_{r}^{+}} y^{a} U^2 \Ga dX dt = r^2 H(r),
\]
\[
\int_{\Sa_{r}^{+}} y^a |t|  |\nabla U|^2 \Ga dX dt = r^2 I(r) + \int_{S_r} |t| V u^2 \Ga(x,0,t) dx dt,
\]
we next introduce the quantity 
\begin{equation}\label{n1}
N_1(r) \overset{def}{=} \frac{\int_{\Sa_r^{+}} y^a |t| |\nabla U|^2\Ga}{ \int_{\Sa_r^{+}} y^a U^2\Ga} = N(r) + \frac{\int_{S_r} |t| V u^2 \Ga(x,0,t) dx dt}{r^2 H(r)},
\end{equation}
and notice that we trivially have $N_1(r) \ge 0$.
Since \eqref{vasump} and \eqref{e9'} give
\begin{align*}
\left|\int_{S_r} |t| V u^2 \Ga(x,0,t) dx dt\right| & \le K \int_{-r^2}^0 |t| \int_{\Rn \times \{0\}} u^2 \Ga \leq C K \int_{-r^2}^0  |t|^{s} \left(i(t) + h(t)\right)
\\
& \le C K r^{2s} \int_{-r^2}^0 \left(i(t) + h(t)\right) = C K r^{2s+2} \left(I(r) + H(r)\right),
\end{align*}
where in the second to the last inequality we have used that $i(t)+h(t)\ge 0$ for $-t_0\le t\le 0$, see  \eqref{e9'} in Lemma \ref{L:trace}. We infer that
\[
\left|\frac{\int_{S_r} |t| V u^2 \Ga(x,0,t) dx dt}{r^2 H(r)}\right| \le K_1 r^{2s} \left(N(r) + 1\right),
\]
where $K_1 = K_1(n,s,K)>0$.
From this estimate and \eqref{n1} we obtain that, under the assumptions of Theorem \ref{mon1}, the following inequality holds for $0\le r \le r_0$
\begin{equation}\label{n11}
N_1(r) - K_1 r^{2s} N_1 (r) -K_1 r^{2s} \le N(r) \le N_1(r) + K_1 r^{2s} N_1 (r) + K_1 r^{2s}.
\end{equation}
From \eqref{n11} we conclude, in particular, that the boundedness of $N$ on $(0,r_0)$ (see Remark \ref{R:imp}) implies that of $N_1$. 
We also have the following consequence of Theorem \ref{mon1}. 

\begin{cor}\label{min}
Under the assumptions of Theorem \ref{mon1} the limits 
\[
\underset{r \to 0}{\lim} N(r)= \underset{r \to 0}{\lim} N_1(r) 
\]
exist finite and they coincide.
\end{cor}

\begin{proof}
Consider the function in \eqref{mon}
\[
\overline{N}(r) \overset{def}{=} e^{C\psi(r)} \left(N(r)  +  C \psi(r)\right),
\]
where $\psi(r) = \int_{0}^r t^{-a} dt$.
By Theorem \ref{mon} and Remark \ref{R:imp} we know that there exists finite 
\[
\overline N(0+) = \underset{r \to 0}{\lim} \overline N(r).
\]
Since 
\[
N(r)= e^{-C \psi(r)} \overline{N}(r) - C \psi(r),
\]
we infer that also 
\[
N(0+) = \underset{r \to 0}{\lim} N(r)
\]
exists finite. By the boundedness of $N_1$ on $(0,r_0)$ and \eqref{n11} we conclude that also 
\[
N_1(0+) = \underset{r \to 0}{\lim} N_1(r)
\]
exists finite and furthermore $N_1(0+)  = N(0+)$.

\end{proof}

The following result about the growth of the function $H(r)$ 
 will be important in the sequel.  

\begin{cor}[Non-degeneracy]\label{growth}
Under the assumptions in Theorem \ref{mon1} one has for every $0<r<r_0$ 
\begin{equation}\label{iner10}
H(r) \geq  H(r_0) \left(\frac{r}{r_0}\right)^{4||\overline N||_\infty+a}
\end{equation}
where we have let 
\[
||\overline N||_\infty = ||\overline N||_{L^\infty(0,r_0)}. 
\]
\end{cor}

\begin{proof}
From \eqref{mon10} in Lemma \ref{L:fvh} we have for $0<\sigma<r_0$
\[
\frac{d}{d\sigma} \log H(\sigma) = \frac{4}{\sigma} N(\sigma) + \frac{a}{\sigma}.
\]
Integrating this identity on the interval $(r,r_0)$, we find
\[
\log\left(\frac{H(r_0)}{H(r)} \left(\frac{r}{r_0}\right)^a\right) = 4 \int_r^{r_0} \frac{N(\sigma)}{\sigma} d\sigma \le 4 ||\overline N||_\infty \log \frac{r_0}{r},
\] 
where in the last inequality we have used Remark \ref{R:imp}. Exponentiating we obtain the desired conclusion.
 
\end{proof}

\begin{cor}\label{int}
If with $r_0$ as in Theorem \ref{mon1} we have $H(r_0) \neq 0$, then we must have $H(r) \neq 0$ for all $0<r <r_0$.  
\end{cor}

\begin{proof}

We argue by contradiction and assume that there exist $0 <\overline r < r_0$  such that $H(\overline r) =0$. Define
\[
\rho = \sup  \{r\leq r_0\mid H(r) =0\}.
\]
Since $H(r_0) \neq 0$, we must have $0 < \rho< r_0$. Since by the hypothesis in Theorem \ref{mon1} we have $H(r)\not= 0$ for $r \in (\rho, r_0]$, applying \eqref{iner10} above we obtain 
\[
H(r) \geq  H(r_0) \left(\frac{r}{r_0}\right)^{4||\overline N||_\infty+a}
\]
for every $r \in (\rho, r_0]$. Letting $r \to  \rho^+$ this leads to a contradiction since $H(\rho)=0$. 
 
\end{proof}


\section{Blow-up analysis and the Proof of  Theorem \ref{main}}\label{S:comp}

In this section we develop a blow-up analysis for solutions of the extension problem \eqref{ext} above with the objective of proving Theorem \ref{main}. Our arguments are similar in spirit to those in \cite{FF} (strong unique continuation for $(-\Delta)^s$) and \cite{DGPT} (parabolic Signorini problem).  As mentioned in the introduction,  we will  show that  certain rescaled versions  $\{U_r\}$  of $U$  converge to a globally defined function $U_0$ as the rescaling parameter $r$ goes to $0$ and the proof of which  crucially uses  the monotonicity result Theorem \ref{mon1} (see Lemma \ref{convergence} below.) We then show that this $U_0$  is homogeneous with respect to the non-isotropic parabolic scalings (see Proposition \ref{homg1}). Finally, from the homogeneity of $U_0$  and the equation satisfied by it,  we   prove our main result Theorem \ref{main}. Roughly speaking, this blow-up analysis  allows us to derive "quantitative" information  about vanishing order in the thick space (i.e., gives us  information about the  vanishing order for $U$  at $(0, 0, 0) \in \Rnp \times \R$)  given the   knowledge  about  vanishing  order  at the boundary (i.e., once we know about  vanishing order for $u$ at $(0, 0) \in \Rn \times \R$).

Now as in Section \ref{S:dgn}, with  $u$ as in the statement of Theorem \ref{main}, we denote by $U$ the solution to the extension problem \eqref{ext} above. Without loss of generality we assume hereafter that the point $(x_0, t_0)$ in that statement be the origin $(0,0)$. 
Throughout this section the number $r_0>0$ will always indicate that specified in Theorem \ref{mon1}. 
To prove Theorem \ref{main} we need to be able to conclude that, under the hypothesis that $u$ vanishes to infinite order at $(0,0)$, we must have $u\equiv 0$ in $\Rn\times [-r_0^2,0)$. This will of course be the case if we can show that $H(U, r_0) = 0$. 

We thus suppose that 
\begin{equation}\label{assump}
H(r_0)= H(U, r_0) > 0,
\end{equation}
and our objective is to show that \eqref{assump} leads to a contradiction.
Note that \eqref{assump} combined with Corollary \ref{int} implies that $H(r) >0$ for all $0<r <r_0$.  
We recall the definition \eqref{pardil} of the parabolic dilations in $\Rnn\times \R$.
 
 \begin{dfn}\label{D:abups}
We define the \emph{Almgren rescalings} of $U$ as 
\begin{equation}\label{bl}
U_r (X, t)= \frac{r^{a/2}U(\delta_r(X,t))}{\sqrt{H(U, r)}} = \frac{r^{a/2}U(rX, r^2 t)}{\sqrt{H(U, r)}}.
\end{equation}
\end{dfn}

An elementary, yet crucial property, of the Almgren rescalings which follows from \eqref{H}, the fact that $\Ga(r^{-1}X,r^{-2}t) = r^{n+1} \Ga(X,t)$, and a change of variable is that
\begin{equation}\label{15}
H(U_r, 1)= \int_{\Sa_1^{+}}y^{a} U_r^2 \Ga dX dt = 1.
\end{equation}
From \eqref{15} and analogous considerations we see that, with $N_1$ defined as in \eqref{n1}, we have
\begin{equation}\label{16}
 N_1(U,r) = N_1(U_r,1) = \int_{\Sa_1^{+}}y^{a} |t| |\nabla U_r|^2 \Ga \le ||N_1||_\infty = ||N_1||_{L^\infty(0,r_0)}<\infty, 
\end{equation}
since, as we have noted after \eqref{n11}, $N_1(r)$ is bounded for $r\in (0,r_0)$ as a consequence of Theorem \ref{mon1}. In fact, \eqref{16} is a special case of the following more general property
\begin{equation}\label{si}
N_1(U_r,\rho) = N_1(U, r\rho),\ \ \ \ \ \ \ \ \ \ \ \ r, \rho >0.
\end{equation}
We remark that $U_r$ solves the following problem
\begin{equation}\label{ext1}
\begin{cases}
\operatorname{div}(y^a \nabla U_r)= y^a \partial_t U_r,
\\
\underset{y \to 0}{\lim} - y^a \partial_{y} U_r = r^{2s} V(rx, r^2 t) U_r(x,0,t).
\end{cases}
\end{equation}

The next result shows that, on a subsequence  $r_j \to 0$, the Almgren rescalings $U_{r_j}$ converge in some appropriate topology. Note that when in the next statement we say $\nabla U_j \to \nabla U_0$ weakly  on compact subsets of $\overline{\Rnp} \times (-1, 0)$, we mean there is weak convergence in $L^2(y^a dXdt)$ on sets of the type $\overline{\B_{A}^{+}} \times (-T, -\delta)$, for all $A>0$ and all $0< \delta <T<1$.

\begin{lemma}\label{convergence}
With $ U_r$ as in \eqref{bl}, there exists a subsequence  $r_j \to 0$, and a function  $ U_0: \Rn \times \R \to \R$, such that  $ U_{r_j}= U_j$ converges uniformly to $ U_0$  and $\nabla  U_j \to \nabla  U_0$ weakly in $L^{2}(|y|^a dX)$ on compact subsets of $\overline{\Rnp} \times (-1, 0)$. Moreover, $ U_0$ is a weak solution to 
\begin{equation}\label{hom}
\begin{cases}
\operatorname{div} (|y|^{a} \nabla  U_0)= 0
\\
\underset{y \to 0^+}{\lim} y^{a} \partial_y  U_0 =0,
\end{cases}
\end{equation}
on every compact  subset of $\overline{\Rnp} \times (-1, 0)$. 
\end{lemma}

\begin{proof}
We first note that for any given $A> 0$ and $0< \delta <T<1$, the function $\Ga$ is bounded from below in the sets $\B_A^{+} \times (-T,-\delta]$. As a consequence, we see from \eqref{15}, \eqref{16} that 
\begin{equation}\label{unif}
\int_{-T}^{-\delta} \int_{\B_A^{+} } y^a U_r^2 \leq C_0 (A,\delta,T),\ \ \ \ \int_{-T}^{-\delta} \int_{\B_A^{+} } y^a |\nabla U_r|^2 \leq C_0(A,\delta,T).
\end{equation}
The bounds \eqref{unif} imply the existence of $U_0 :\B_A^{+} \times (-T, -\delta] \to \R$ such that for a subsequence $r_j \to 0$, $U_{r_j} \to U_0$ and $\nabla U_{r_j} \to \nabla U_0$ weakly in $L^{2}(\B_A^{+} \times (-T,-\delta],y^a dXdt)$. Moreover, from the regularity estimates in Theorem \ref{reg5} and the Theorem of Ascoli-Arzel\`a, after possibly passing to another subsequence,  we can guarantee that $U_{r_j} \to U_0$ uniformly in $\overline{\B_{A'}^{+}} \times (-T', -\delta]$ for any $A'<A$ and $T' < T$. By a countable exhaustion of $\Sa_{1}^{+}$ with domains of the type $\B_A^{+} \times (-T,-\delta]$, with $A>0$ and $0< \delta< T < 1$, and a Cantor  diagonalization argument, we conclude that for a  subsequence $r_j \to 0$, $U_j = U_{r_j}$ converges uniformly to $U_0$, and $\nabla U_j \to \nabla U_0$ weakly on compact subsets of $\overline{ \Rnp} \times (-1,0)$.

Next, we show that $U_0$ solves \eqref{hom}. To prove this fact, given $A>0$ and $0<\delta < T<1$, we pick a test function $\phi \in W^{1,2}(\B_A^{+} \times (-T,-\delta], y^a dXdt)$ with compact support in $\B_A^{+} \times (-T,-\delta]$. Since $U_r = U_{r_j} = U_j$ solves \eqref{ext1}, for every $t_1, t_2 \in (-T, -\delta)$ we see that the following holds
\begin{align}\label{ext2}
\int_{t_1}^{t_2} \int_{\B_A^{+}} y^a <\nabla U_r, \nabla \phi> & = \int_{t_1}^{t_2} \int_{\B_A^{+}} y^a U_r \phi_t + \int_{\Rn\times \{0\}} r^{2s} V(rx, r^2t) \phi U_r  
\\
& - \int_{\B_A^{+}} y^a U_r(\cdot,t_2)  \phi(\cdot,t_2) + \int_{\B_A^{+}} y^a U_r(\cdot,t_1)\phi(\cdot,t_1).
\notag
\end{align}
By the weak convergence of $\nabla U_j \to \nabla U_0$, and the uniform convergence of $U_j \to U_0$, we infer that
\[
\begin{cases}
\int_{t_1}^{t_2} \int_{\B_A^{+}} y^a <\nabla U_j,\nabla \phi>\ \longrightarrow\ \int_{t_1}^{t_2} \int_{\B_A^{+}} y^a <\nabla U_0,\nabla \phi>,
\\
\int_{t_1}^{t_2} \int_{\B_A^{+}} y^a U_j \phi_t\ \longrightarrow\ \int_{t_1}^{t_2} \int_{\B_A^{+}}   y^a U_0 \phi_t,
\\
\int_{\B_A^{+}} y^a U_j(\cdot,t_i)\phi(\cdot,t_i)\ \longrightarrow\ \int_{\B_A^{+}} y^a U_0(\cdot,t_i)\phi(\cdot,t_i),\ \ \ \ \ \ i=1,2.
\end{cases}
\]
Moreover, since by \eqref{vasump} we have $r^{2s}|V(rx, r^2t)| \leq K r^{2s}$, using the trace  inequality in Lemma \ref{tr} with $\mu=1$, and the bounds in \eqref{unif}, we find 
\begin{equation}
\int_{B_A \times (t_1, t_2)} r^{2s} V(rx, r^2t) \phi U_r\ \longrightarrow\ 0\ \ \text{as}\ r \to 0.
\end{equation}
We conclude that $U_0 \in V^{a, A,-T,-\delta}$ for all $A, \delta>0$ and $T< 1$, and that it is a weak solution to \eqref{hom}.  

\end{proof}

\begin{dfn}\label{D:albu}
A function $U_0$ as in Lemma \ref{convergence} will be called an \emph{Almgren blowup} of $U$.
\end{dfn}

Before proceeding further, we observe that if $U_0$ is any Almgren blowup of $U$, then by taking  difference quotients in the $x_1,..., x_n, t$ directions, and by repeatedly applying Lemma \ref{reg2}, we see that $U_0$ is infinitely differentiable in $(x,t)\in \Rn\times (-1,0)$. 

A fundamental property of the Almgren blowups is that they are homogeneous with respect to the parabolic dilations, with a degree of homogeneity which is decided by the value at zero $N(U,+0)$ of the frequency of $U$. This is the content of Proposition \ref{homg1} below. Before we can prove such result, however, we need to establish the following basic convergence result in Gaussian spaces for which we closely follow some ideas in the proof of Theorem 7.3 in \cite{DGPT}. 
   
\begin{lemma}\label{F1}
For any subsequence $r_j \to 0$ as  in Lemma \ref{convergence}, we have for $U_j = U_{r_j}$ and any $R <1$, 
\begin{equation}\label{f1}
\int_{\Sa_{R}^{+}} y^{a} \left(|U_j -U_0|^2 + |t| |\nabla U_j - \nabla U_0|^2 \right) \Ga\ \longrightarrow\ 0.
\end{equation}
\end{lemma}

\begin{proof}
In what follows we write $H(r), I(r), N(r), N_1(r)$, instead of $H(U,r)$, etc. We first show that as $j\to \infty$ 
\begin{equation}\label{i'1}
\int_{\Sa_{R}^{+}} y^{a}  |U_j -U_0|^2 \Ga\  \longrightarrow\ 0.
\end{equation}
 
In order to establish \eqref{i'1} we note that since $N_1(r)\ge 0$, possibly restricting $r_0$, \eqref{n11} implies in particular that for $0< r\le r_0$  the following holds, 
\begin{equation}\label{e9''}
N(r) \geq - K_1 r^{2s}. 
\end{equation}
Restricting $r_0$ even further, in such a way that  
\[
4 K_1 r_0^{2s} < 1,
\]
we have for $0< r\le r_0$
\[
N(r) \ge - \frac{1}{4}.
\]
Since \eqref{mon10} in Lemma \ref{L:fvh} gives
\[
\frac{H'(r)}{H(r)} = \frac{4 N(r)}{r} + \frac{a}{r},
\]
we conclude for $0<r \leq r_0$ 
\[
\frac{H'(r)}{H(r)} \geq \frac{a - 1}{r}.
\]

Let $0<\delta<1$ be arbitrarily chosen. Given any $0<r\le r_0$, integrating the latter inequality from $\delta r$ to $r$ we find  
\begin{equation*}\label{56}
H(r\delta) \leq H(r) \delta^{a-1}.
\end{equation*}
Using this inequality and a simple calculation we conclude for $0<r\le r_0$  
\begin{equation}\label{cond4}
\int _{\Sa_{\delta}^{+}}  y^{a} U_r^2 \Ga = \frac{\delta^2 H(\delta r)}{H(r)} \leq \delta^{a+1}.
\end{equation}
At this point, we need the following inequalities from \cite{DGPT} which are  corollaries of L. Gross' log-Sobolev inequality (see Lemma 7.7 in \cite{DGPT}). We first  write
\begin{equation}\label{prod}
\Ga(X,t) = G(x,t) K(y,t),
\end{equation}
where $G(x,t) = (4\pi |t|)^{-n/2} e^{\frac{|x|^2}{4t}}$ and $K(y,t) = (4\pi |t|)^{-n/2} e^{\frac{y^2}{4t}}$ respectively indicate the backward heat kernels in $\Rn$ and in $\R$.  The following inequalities hold:
\begin{equation}\label{ineq}
\log(\frac{1}{\int_{|f| > 0} \Ga(\cdot,s)} ) \int_{\Rnn}  f^2 \Ga(\cdot,s) \leq 2 |s| \int_{\Rnn} |\nabla f|^2 \Ga(\cdot,s),\ \ \ \ \ f \in W^{1,2}(\Rnn, \Ga(\cdot,s)),
\end{equation}
and
\begin{equation}\label{ineq1}
\log(\frac{1}{\int_{|f| > 0} G(\cdot,s)} ) \int_{\Rn}  f^2 G(\cdot,s) \leq 2 |s| \int_{\Rn} |\nabla f|^2 G(\cdot,s),\ \ \ \ \ \  f \in W^{1,2}(\Rn,G(\cdot,s)).
\end{equation}
We now choose  $A> 2$  large  enough such that for all $-1 < t < 0$, 
\begin{equation}\label{cond}
 \int_{\R^{n+1} \setminus \B_{A/2}} \Ga (X, t) dX \leq e^{-1/\delta},\ \ \ \ \ \ \  \int_{\R^{n} \setminus B_{A/2}} G(x,t) dx \leq e^{-1/\delta}.
\end{equation}
Using \eqref{ineq} and \eqref{ineq1}, we will appropriately estimate  the integral
\begin{equation}\label{end10}
\int_{\R^{n+1}_{+} \setminus \B_A}  y^{a} U_r^2 (X, t) \Ga(X,t).
\end{equation}
We note however that, unlike the situation in Sec. 7 in \cite{DGPT}, where the log-Sobolev inequality\eqref{ineq1} is employed to obtain smallness estimates for a certain integral of the type \eqref{end10}, in the present case, in view of the fact that $a<1$, a direct application of \eqref{ineq} to the natural choice $f= y^{a/2} U_r$  is not possible because $\nabla f$ behaves badly at points where $y$ is small. Therefore, we split \eqref{end10} into two regions as follows
\begin{align}\label{split}
\int_{\R^{n+1}_{+} \setminus \B_A}  y^{a} U_r^2 (X,t) \Ga(X,t) & = \int_{(\R^{n+1}_{+} \setminus \B_A) \cap (\Rn \times \{0<y \leq A/2 \})} y^a U_r^2 (X,t) \Ga(X,t)  
\\
 & + \int_{(\R^{n+1}_{+} \setminus \B_A) \cap (\Rn \times \{y > A/2 \})} y^a U_r^2 (X,t)\Ga(X,t).
 \notag
\end{align}
Now, in the region where $|X| \geq A$ and $0<y \leq A/2$, we must have $|x| \geq A/2$. Therefore, 
\begin{align}\label{iner19}
& \int_{(\Rnn_{+} \setminus \B_A) \cap (\Rn \times \{0<y \leq A/2 \}) } y^a U_r^2(X,t) \Ga(X,t)
\\
& \leq \int_{0}^{A/2} y^a K(y,t) \left(\int_{|x| \geq A/2} U_r^2 (X, t) G(x,t) dx\right) dy.
\notag
\end{align}
We now pick cut-off functions:
\begin{itemize}
\item $\bar \eta\in C^\infty_0(\Rnn)$ such that $\bar \eta \equiv 1$ in $\B_{A/2}$, and $\bar \eta \equiv 0$ outside $\B_{A}$;
\item $\eta\in C^\infty_0(\Rn)$ such that $\eta \equiv 1$ in  $B_{A/2}$, and $\eta \equiv 0$ outside $B_A$;
\item $\eta_2 \in C^\infty(\R)$ such that $\eta_2(y) \equiv 1$ for $y \leq \frac{A}{4}$, and $\eta_2(y) \equiv 0$ for $y \geq \frac{A}{2}$.  
\end{itemize}
Applying \eqref{ineq1} to $f= U_r (1-\eta)$, and noting that by the support property of $\eta$ and the second inequality in \eqref{cond} we have $\int_{|f|>0} G \le \int_{\R^{n} \setminus B_{A/2}} G \leq e^{-1/\delta}$,
we find
\begin{align*}
\int_{|x| \geq A/2} U_r^2 (x,t) G(x,t) & \leq \int_{\Rn} U_r^2(x,t) (1 -\eta(x))^2 G(x,t) 
\\
& \leq  C  \delta |t| \int_{\Rn} (U_r^2(x,t) + |\nabla_{x} U_r(x,t)|^2) G(x,t).
\notag
\end{align*}
Using this inequality in \eqref{iner19}, and keeping \eqref{prod} in mind, we obtain 
\begin{align}\label{cond1}
& \int_{(\Rnn_{+} \setminus \B_A) \cap (\Rn \times \{y \leq A/2 \})} y^a U_r^2(X,t) \Ga(X,t) 
\\
& \leq C \delta |t| \int_{\Rnn_{+}} y^a (U_r^2(X,t) + |\nabla U_r(X,t)|^2)\Ga(X,t).
\notag
\end{align}
 
Next, we estimate the second integral in the right-hand side of \eqref{split}. From the support properties of $\bar \eta$ and $\eta_2$, we find
\begin{align}\label{iner22}
& \int_{(\Rnn_{+} \setminus \B_A)  \cap (\Rn \times \{y > A/2\})} y^{a} U_r(X,t)^2 \Ga(X,t) 
\\
& \leq \int_{\Rnn} y^a  U_r^2(X,t) (1-\bar \eta(X))^2 (1- \eta_2(y))^2 \Ga(X,t).
\notag
\end{align}
We now apply \eqref{ineq} with $f= y^{a/2} U_r (1-\bar \eta(X)) (1- \eta_2(y))$. Note that, although $U_r$ is not defined for $y < 0$, given the fact that $1-\eta_{2}(y)$ vanishes for $y \leq \frac{A}{4}$, the function $f$ is smoothly extendable to the whole of $\Rnn$. For such $f$ we have the following estimate  
 \begin{align}\label{iner23}
|\nabla f|^2 & \leq C ( y^a  |\nabla  U_r|^2 + y^a U_r^2 + y^{a-2} U_r^2 (1-\bar \eta(X))^2 (1- \eta_2(y))^2)
 \\
 & \leq C y^a (U_r^2 + |\nabla U_r|^2),
 \notag
 \end{align}
 where in the last  inequality we have used the fact that, since $y^{a-2} U_r^2 (1-\eta(X))^2 (1- \eta_2(y))^2 $  is non-zero only for $y > \frac{A}{4}$ and $A>2$, we have
 \[
 y^{a-2} U_r^2 (1-\bar \eta(X))^2 (1- \eta_2(y))^2  \leq  y^{a-2} U_r^2 \leq 4 y^a U_r^2. 
\]
Therefore, applying \eqref{ineq} to $f= y^{a/2} U_r (1-\eta(X)) (1- \eta_2(y))$ we obtain from \eqref{iner22} and \eqref{iner23}
\begin{align}\label{cond2}
& \int_{(\Rnn_{+} \setminus \B_A)  \cap (\Rn \times \{y > A/2 \})} y^{a} U_r^2(X,t) \Ga(X,t) dx
\\
& \leq  C \delta |t| \int_{\Rnn_{+}} y^a (U_r^2(X,t) + |\nabla U_r(X,t)|^2)\Ga(X,t).
\notag
\end{align}
Integrating \eqref{cond2} in $t$ on $(-R^2, 0)$ we find 
\begin{align} \label{cond3}
& \int_{(\Rnp \setminus \B_A) \times (-R^2, 0)} y^{a} U_r^2(X,t) \Ga(X,t)
\\
& \leq C \delta  \int_{\Rnp \times (-R^2, 0) } |t| y^a (U_r^2(X, t) + |\nabla U_r(X, t)|^2)\Ga(X, t) \leq C \delta, 
\notag
\end{align}
where in the last inequality we have used the uniform bounds \eqref{15} and \eqref{16}. Combining\eqref{cond3} with \eqref{cond4} we conclude for $0<r\le r_0$  
\begin{equation}\label{conv10}
\int_{[(\Rnp \setminus \B_A) \times (-R^2,0)]  \cup \Sa_{\delta}^{+}}  y^{a} U_r^2(X,t) \Ga(X,t) \leq C \delta^{a_0},
\end{equation}
where $a_0= \min\{1,1+a\}$. 
In the set $\B_A^{+} \times [-R^2, -\delta^2]$, which is the complement of $[(\Rnp \setminus \B_A) \times (-R^2, 0)]  \cup \Sa_{\delta}^{+}$, we apply Lemma \ref{convergence} which guarantees that, for a suitable sequence $r_j \searrow 0$, we have $U_{r_j} \to U_0$ uniformly in $\B_A^{+} \times [-R^2, -\delta^2]$. Therefore, from \eqref{conv10} and the arbitrariness of $\delta$, we conclude that \eqref{i'1} holds.

Now  we establish the second part of \eqref{f1}, i.e., 
\[
\int_{\Sa_{R}^{+}} y^{a} |t| |\nabla U_j - \nabla U_0|^2  \Ga\ \longrightarrow\ 0.
\]
This conclusion will be directly derived from the following estimate 
\begin{align}\label{wk101}
& \int_{\Sa_{R}^{+}} y^a |\nabla U_r - \nabla U_0|^2 |t| \Ga \leq C_2  \int_{\Sa_{\rho_1}^{+} } y^a (U_r- U_0) ^2 \Ga + C_2r^{2s},
\end{align} 
where $\rho_1<1$ is arbitrarily chosen.
It is clear that letting $r = r_j\searrow 0$ in \eqref{wk101}, and using \eqref{i'1}, we reach the desired conclusion.

To proof of \eqref{wk101} is somewhat involved and will be accomplished in several steps. First, we note that  $U_r-U_0$  satisfies the following equation in $\Sa_{1}^{+}$
\begin{equation}\label{diff}
\begin{cases}
\operatorname{div}(y^a \nabla (U_r-U_0)) = y^a (U_r - U_0)_t
\\
 -\lim_{y \to 0} y^a (U_r - U_0)_y = r^{2s} V(rx, r^2t) U_r.
\end{cases}
\end{equation}
For notational convenience we will denote  $U_r-U_0$  by $v$ in the computations from \eqref{wk} to \eqref{wk4} below. For the same reason, we will write $V_r(x,t)$ instead of $V(rx, r^2t)$. 

Let $\rho \in [R, \rho_1]$ be arbitrary, where $\rho_1<1$ and let $0 < \delta < \rho$ be a number that will eventually tend to $0$.  We now use the weak formulation \eqref{d8} in Definition \ref{d2} applied to $U_r - U_0$, with $t_1 = - \rho^2$ and $t_2 =  - \delta^2$. Since from Lemmas \ref{reg2} and \ref{reg1} we know differentiability in time of $U_r - U_0$, in the weak formulation \eqref{d8} we can put the time-derivative on $v$, and furthermore the boundary terms disappear. In \eqref{d8} we now choose a test function $\eta$ such that $\eta(\cdot, t)$ is compactly supported in the $X$ variable for every $t \in [-\rho^2,-\delta^2]$. Then, corresponding to such a $\eta$, from \eqref{diff} we have the following weak formulation for $v$ 
\begin{equation}\label{wk}
\int_{\Sa_{\rho}^{+} \setminus \Sa_{\delta}^{+}} y^{a}\left(<\nabla v, \nabla \eta > + v_t \eta\right) = \int_{S_{\rho} \setminus S_{\delta}} r^{2s}  V_r  U_r  \eta.
\end{equation}
We are going to eventually derive \eqref{wk101} from \eqref{wk} by an appropriate choice of the cutoff function $\eta$, and by several estimates. With this objective in mind, we let 
\begin{equation}\label{tau1}
\tau_1(X,t) = |t|^{1/2} \tau_0\left(\frac{X}{\sqrt{|t|}}\right),
\end{equation}
where $\tau_0\in C^\infty_0(\Rnn)$, $\tau_0^2 \leq 1$, and whose support will at the end of the process exhaust the whole $\Rnn$. Corresponding to such a $\tau_0$ we choose 
\begin{equation}\label{rt1}
\eta =  \tau_1^{2} v \Ga.
\end{equation}
By substituting this choice of $\eta$ in \eqref{wk}, we find
\begin{align*}
& \int_{\Sa_\rho^{+} \setminus \Sa_{\delta}^{+}} y^a \left(|\nabla v|^2 \tau_1^{2} \Ga+ v <\nabla v,\frac{X}{2t}>\Ga \tau_1^2 + vv_t \Ga \tau_1^2 + 2 v \tau_1 <\nabla v,\nabla \tau_1>\Ga\right)
\\
& = \int_{S_\rho \setminus S_{\delta}} r^{2s} V_r U_r v  \tau_1 ^2 \Ga.
\notag
\end{align*}
Using \eqref{z11} in the latter equation we can rewrite it as follows  
\begin{align}\label{wk100}
& \int_{\Sa_\rho^{+} \setminus \Sa_{\delta}^{+}} y^{a} |\nabla v|^2 \tau_1^2 \Ga = - \int_{\Sa_\rho^{+} \setminus \Sa_{\delta}^{+}} y^{a} \frac{1}{4t} Z(v^2) \tau_1^2\Ga 
\\
& - \int_{\Sa_\rho^{+} \setminus \Sa_{\delta}^{+}} y^{a}  2v \tau_1 <\nabla v,\nabla \tau_1>\Ga
 +  r^{2s} \int_{S_\rho \setminus S_{\delta}} V_r U_r v\tau_1^2 \Ga.
\notag
\end{align}
We intend to show that \eqref{wk100} leads to the following inequality
\begin{align}\label{wk4}
& \int_{\Sa_{R}^{+} \setminus \Sa_{\delta}^{+}} y^a |\nabla U_r - \nabla U_0|^2 \tau_1^2 \Ga \leq C_2  \int_{\Sa_{\rho_1}^{+} \setminus \Sa_{\delta}^{+}} y^a (U_r- U_0) ^2 \Ga + C_2r^{2s},
\end{align}
where $C_2 = C_2(n,s,K)>0$.
At this point, if we let the support of $\tau_0$ exhaust $\R^{n+1}$ in \eqref{wk4}, and note that from \eqref{tau1} we obtain that $\tau_1^2(X,t) \longrightarrow |t|$, applying the monotone convergence theorem in \eqref{wk4}, and then letting $\delta \to 0$, we conclude that \eqref{wk101} holds.  

In order to complete the proof of the lemma we are thus left with proving that \eqref{wk100} $\Longrightarrow$ 
\eqref{wk4}.
With this objective in mind we now estimate each of the terms in the right-hand side of \eqref{wk100}, beginning with the last integral which is the most delicate one.
With $K$ as in \eqref{vasump}, and $C_0$ as in Lemma \ref{tr}, we choose $\mu$ as follows  
\[
C_0 K \mu^{-2s} = \frac{1}{2},
\]
and apply the trace inequality in Lemma \ref{tr} at every time level $t\in (-\rho^2, \delta^2)$ to both functions $f_1=  K^{1/2} U_r \tau_1 \Ga^{1/2}$ and $f_2= K^{1/2} v \tau_1 \Ga_{1/2}$. Summing the resulting inequalities, we find 
\begin{align}\label{wk2}
& r^{2s}\int_{S_\rho \setminus S_{\delta}}V_r U_r v \tau_1^2 \Ga \leq r^{2s} K \int_{S_\rho \setminus S_{\delta}} (U_r^2 + v^2) \tau_1^2 \Ga
\\
& \leq  r^{2s}  \bigg[C \int_{\Sa_\rho^{+} \setminus \Sa_{\delta}^{+}} y^a (U_r^2 + v^2) \tau_1^2\Ga 
 +  \frac{1}{2} \bigg(\int_{\Sa_\rho^{+} \setminus \Sa_{\delta}^{+}} y^a (|\nabla v|^2 + |\nabla U_r|^2) \tau_1^2 \Ga
\notag\\
&  + \int_{\Sa_\rho^{+} \setminus \Sa_{\delta}^{+}} y^a (v^2 + U_r^2) |\nabla \tau_1|^2 G  
 + \int_{\Sa_\rho^{+} \setminus \Sa_{\delta}^{+}}y^a (v^2 + U_r^2) \tau_1^2 \frac{|X|^2}{16 t^2} \Ga\bigg)\bigg],
\notag
\end{align}
where in the first inequality in \eqref{wk2} we have used \eqref{vasump} that gives $|V_r| \leq K$. We now bound from above the integral
\[
\int_{\Sa_\rho^{+} \setminus \Sa_{\delta}^{+}} y^a (v^2 + U_r^2) \tau_1^2 \frac{|X|^2}{16 t^2}\Ga= \int_{\Sa_\rho^{+} \setminus \Sa_{\delta}^{+}} y^a v^2 \tau_1^2 \frac{|X|^2}{16 t^2}\Ga + \int_{\Sa_\rho^{+} \setminus \Sa_{\delta}^{+}} y^a U_r^2 \tau_1^2 \frac{|X|^2}{16 t^2}\Ga
\]
in \eqref{wk2} by separately applying the inequality \eqref{ne1} in the Appendix to the two integrals in the right-hand side of the latter equation. Adding the resulting inequalities we obtain
\begin{align}\label{ne10}
& \int_{\Sa_\rho^{+} \setminus \Sa_{\delta}^{+}} y^a (v^2 + U_r^2) \tau_1^2 \frac{|X|^2}{16 t^2}\Ga \leq  C \bigg(\int_{\Sa_\rho^{+} \setminus \Sa_{\delta}^{+}}y^a  (v^2 + U_r^2) \frac{\tau_1^2}{|t|}\Ga 
\\
& + \int_{\Sa_\rho^{+} \setminus \Sa_{\delta}^{+}} y^a (|\nabla v|^2 + |\nabla U_r|^2)\tau_1^2 \Ga + \int_{\Sa_\rho^{+} \setminus \Sa_{\delta}^{+}} y^a (v^2 + U_r^2) |\nabla \tau_1|^2\Ga\bigg).
\notag
\end{align}
We now substitute \eqref{ne10} into \eqref{wk2}. If we keep in mind that $\tau_1^2(X, t) \leq |t|$ and  $|\nabla \tau_1|^2 \le C$, where $C$ is a universal constant, we find
\begin{align}\label{wk22}
& r^{2s}\int_{S_\rho \setminus S_{\delta}}V_r U_r v \tau_1^2 \Ga 
\\
& \leq  r^{2s}  \bigg\{C \int_{\Sa_\rho^{+} \setminus \Sa_{\delta}^{+}} y^a |t| (v^2 + U_r^2) \Ga 
 +  \frac{1}{2} \bigg[\int_{\Sa_\rho^{+} \setminus \Sa_{\delta}^{+}} y^a (|\nabla v|^2 + |\nabla U_r|^2) \tau_1^2 \Ga
\notag\\
&  + C \int_{\Sa_\rho^{+} \setminus \Sa_{\delta}^{+}} y^a (v^2 + U_r^2) \Ga  
 + C \bigg(\int_{\Sa_\rho^{+} \setminus \Sa_{\delta}^{+}}y^a  (U_r^2 + v^2) \Ga 
\notag
\\
& + \int_{\Sa_\rho^{+} \setminus \Sa_{\delta}^{+}} y^a (|\nabla v|^2 + |\nabla U_r|^2)\tau_1^2 \Ga + C \int_{\Sa_\rho^{+} \setminus \Sa_{\delta}^{+}} y^a (v^2 + U_r^2) \Ga\bigg)\bigg]\bigg\}.
\notag
\end{align}
If we further restrict $r_0$ so that
\[
\left(1+ C\right)r_0^{2s} \le 1,
\]
then for $0<r\le r_0$ we obtain from \eqref{wk100} and \eqref{wk22}
\begin{align}\label{wk1000}
& \frac 12 \int_{\Sa_\rho^{+} \setminus \Sa_{\delta}^{+}} y^{a} |\nabla v|^2 \tau_1^2 \Ga \le - \int_{\Sa_\rho^{+} \setminus \Sa_{\delta}^{+}} y^{a} \frac{1}{4t} Z(v^2) \tau_1^2\Ga - \int_{\Sa_\rho^{+} \setminus \Sa_{\delta}^{+}} y^{a}  2v \tau_1 <\nabla v,\nabla \tau_1>\Ga
\\
& +  r^{2s}  \bigg\{C \int_{\Sa_\rho^{+} \setminus \Sa_{\delta}^{+}} y^a |t| (v^2 + U_r^2) \Ga 
 +  \frac{1}{2} \bigg[(C+1) \int_{\Sa_\rho^{+} \setminus \Sa_{\delta}^{+}} y^a |t| |\nabla U_r|^2  \Ga
\notag\\
& + C \int_{\Sa_\rho^{+} \setminus \Sa_{\delta}^{+}} y^a (v^2 + U_r^2) \Ga\bigg)\bigg]\bigg\}.
\notag
\end{align}

Next, we bound from above the first two integrals in the right-hand side of \eqref{wk1000}.
First, recalling that $Z$ is the infinitesimal generator of the parabolic dilations \eqref{pardil}, by a standard argument which can be found on p. 94 of the Appendix A in \cite{DGPT}, we find  
\begin{equation}\label{wk1}
- \int_{\Sa_\rho^{+} \setminus \Sa_{\delta}^{+}}  y^a \frac{1}{2t} Z(v^2) \tau_1^2 \Ga \le  \rho^{2+a} \int_{\R^{n+1}_{+}} y^av^2(\cdot, -\rho^2) \tau_0^2 \Ga(\cdot,-\rho^2).
\end{equation}
Secondly, applying Young's inequality to $2v \tau_1 <\nabla v, \nabla \tau_1>\Ga$, we have
\begin{align}\label{iner52}
\left|2 \int_{\Sa_\rho^{+} \setminus \Sa_{\delta}^{+}}  y^a v \tau_1 <\nabla v, \nabla \tau_1> \Ga\right| \leq \frac{1}{4} \int_{\Sa_\rho^{+} \setminus \Sa_{\delta}^{+}} y^{a} |\nabla v|^2 \tau_1^2 \Ga
+ 4 \int_{\Sa_\rho^{+} \setminus \Sa_{\delta}^{+}} y^a v^2 |\nabla \tau_1|^2 \Ga.
\end{align}
Using the estimates \eqref{wk1}, \eqref{iner52} in \eqref{wk1000}, and keeping in mind that $0<\rho<1$ and $-1<t<0$, we obtain
\begin{align}\label{wk3}
& \int_{\Sa_{\rho}^{+} \setminus \Sa_{\delta}^{+}}  y^a|\nabla v|^2 \tau_1^2 \Ga \leq C \int_{\R^{n+1}_{+}} y^av^2(\cdot, -\rho^2) \tau_0^2 \Ga(\cdot,-\rho^2) +C  \int_{\Sa_{\rho}^{+} \setminus \Sa_{\delta}^{+}} y^a v^2  \Ga
\\
&  + C r^{2s}  \left(\int_{\Sa_{\rho}^{+} \setminus \Sa_{\delta}^{+}} y^a |t| |\nabla U_r|^2  \Ga +  \int_{\Sa_{\rho}^{+} \setminus \Sa_{\delta}^{+}} y^a ( v^2 + U_r^2)\Ga\right).
\notag
\end{align}
Recalling that $v= U_r- U_0$,  using \eqref{15}, \eqref{16} and \eqref{i'1} in \eqref{wk3}, we conclude for some $C_1 = C_1(n,s,K)>0$, and $ r = r_j\searrow 0$,
\begin{equation*}
\int_{\Sa_{\rho}^{+} \setminus \Sa_{\delta}^{+}} y^a |t| |\nabla U_r|^2 \Ga + \int_{\Sa_{\rho}^{+} \setminus \Sa_{\delta}^{+}} y^a (U_r^2 + (U_r-U_0)^2)\Ga \leq C_1.
\end{equation*}
Using the latter inequality in \eqref{wk3}, and the fact that $\tau_0^2 \leq 1$, we have for some $C_2 = C_2(n,s,K)>0$
\begin{align}\label{wk10}
& \int_{\Sa_{\rho}^{+} \setminus \Sa_{\delta}^{+}}  y^a|\nabla (U_r -U_0)|^2 \tau_1^2 \Ga \leq C_2 \int_{\R^{n+1}_{+}} y^av^2(\cdot, -\rho^2) \Ga(\cdot,-\rho^2)
\\
&+C_2  \int_{\Sa_{\rho}^{+} \setminus \Sa_{\delta}^{+}} y^a (U_r-U_0)^2 \Ga + C_2 r^{2s}.
\notag
\end{align}
Integrating  \eqref{wk10} with respect to $\rho \in [R,\rho_1]$, we finally obtain \eqref{wk4}. 

\end{proof}

Before proceeding further it is important to remark that, since $U_0$ solves \eqref{hom}, then thanks to the vanishing Neumann condition $\underset{y \to 0}{\lim} y^{a} \partial_y U_0 =0$, if we even reflect $U_0$ in the variable $y$, we can conclude that $U_0$ solves the following equation in $\Sa_1-\Sa_{\delta}$, for any  $\delta > 0$,   
\begin{equation}\label{conj}
\operatorname{div}(|y|^a \nabla U_0) = |y|^a (U_0)_t. 
\end{equation}
This implies that the function $g\overset{def}{=} |y|^{a} (U_0)_y$ solves the following \emph{adjoint equation} classically for all $(X,t) \in \Sa_{1}$ such that  $y \neq 0$
\begin{equation}\label{conj1}
\operatorname{div}( |y|^{-a} \nabla g)= |y|^{-a} g_t.
\end{equation}
We now want to show that, in fact, $g$ is a weak solution to \eqref{conj1} in $\Sa_1$, i.e., also across $y = 0$. For this we observe that, because of  the weak  convergence of $U_{r_j}$ to $U_0$ in Lemma \ref{convergence}, it follows  that the estimate \eqref{unif} holds for $U_0$. Therefore, from \eqref{unif} and the arguments in the  proof of Lemma  \ref{reg2}, we conclude that  
  \begin{equation}\label{e505}
     \nabla (U_0)_i,  (U_0)_t  \in L^{2}_{loc} ( |y|^a dXdt)
  \end{equation}
    in $\Sa_1$. Then, as in the proof of Lemma \ref{reg2},  from \eqref{e505} and  the equation \eqref{conj} we can now infer that 
    \begin{equation}\label{e501}
g_y = (|y|^{a} (U_0)_y )_y \in L^{2}_{loc} (|y|^{-a} dXdt)
    \end{equation}
    in  $\Sa_1$. From \eqref{conj1}, \eqref{e505} and \eqref{e501} we conclude that $g$ is a weak solution to  \eqref{conj1}. This implies the H\"older continuity of  $y^a (U_0)_y(\cdot,t)$ on compact subsets of $\overline{\Rnp} \times (-1,0)$.

We are now in a position to prove the following fundamental property of the Almgren blowups. 

\begin{prop}[Homogeneity of the Almgren blowups]\label{homg1}
Under the assumptions of Theorem \ref{mon1}, let 
\[
\kappa = N(U,0+) = N_1(U,0+)
\]
be as in Corollary \ref{min}. If $U_0$ is any Almgren blowup of $U$, then $U_0$ is homogeneous of degree $2\kappa$. 
 \end{prop}
 
\begin{proof}

We begin by noting that thanks to \eqref{f1} in Lemma \ref{F1}, for any $\rho< 1$ we have with $U_j = U_{r_j}$   
\begin{equation}\label{conv}
\begin{cases}
H(U_j,\rho)\ \longrightarrow\  H(U_0,\rho),
\\
\frac{1}{\rho^2} \int_{\Sa_{\rho}^{+}} y^a |t| |\nabla U_j|^2 \Ga\ \longrightarrow\ \frac{1}{\rho^2} \int_{\Sa_{\rho}^{+}} y^a |t| |\nabla U_0|^2 \Ga.
\end{cases}
\end{equation}
Corollary \ref{min} and \eqref{si} above give 
\begin{equation}\label{bd}
N_1(U_r, \rho) = N_1(U, r\rho)\ \longrightarrow\  N_1(U,0+) = N(U,0+) = \kappa\ \ \text{as}\ r \to 0.
\end{equation}
By arguing as for \eqref{iner10} in Corollary \ref{growth}, we obtain
\begin{equation*}
H(U, r\rho) \geq \rho^{\beta} H(U,r),
\end{equation*}
where $\beta = 4 ||\overline{N}||_{\infty} + a$.
This estimate implies 
\begin{equation}\label{res10}
H(U_r,\rho)= \frac{H(U,r\rho)}{H(U, r)} \geq \rho^{\beta}.
\end{equation}
From \eqref{conv} and \eqref{res10} we obtain
\begin{equation}\label{res1}
H(U_0,\rho) > \rho^{\beta}>0.
\end{equation}
From the definition \eqref{n1} of $N_1$, the non-degeneracy of $H(U_0, \rho)$ established in \eqref{res1} and from \eqref{conv} we infer that for any $0< \rho<1$ 
\begin{equation}\label{ndeg}
N_1(U_j,\rho)\ \longrightarrow\ N_1(U_0,\rho).
\end{equation}
From \eqref{ndeg} and \eqref{bd} we obtain that $N_1(U_0,\rho) \equiv \kappa$ for every $0<\rho<1$. Moreover since $U_0$ solves \eqref{hom}, we conclude   
\[
N(U_0,\rho) = N_1(U_0,\rho) \equiv \kappa,\ \ \ \ \ \ \ \ 0<\rho<1.  
\]
Once we know that the frequency $N(U_0,\cdot) \equiv \kappa$, we invoke Theorem \ref{just4}   in the Appendix  (which extends to $U_0$ the second half of Theorem \ref{mon1}), to conclude that $U_0$ has homogeneity $2 \kappa$ with respect to parabolic dilations \eqref{pardil}.  
 
\end{proof}

We note that the estimate \eqref{res1} has the following direct consequence.

\begin{cor}\label{des2}
The Almgren blowup $U_0$ in Lemma \ref{homg1} cannot be identically zero in $\Sa_{1}^{+}$ (the reader should keep in mind that we are assuming \eqref{assump}!).
\end{cor}

 Our next key result is the following.

\begin{lemma}\label{nont}
We have $U_0(\cdot,0,\cdot)\not\equiv 0$.
 \end{lemma}

\begin{proof}

We argue by contradiction, and suppose that $U_0(\cdot,0,\cdot)\equiv 0$. We now adapt to the parabolic setting the arguments in Step $1$ and $2$ in the proof of Proposition 2.2 in \cite{Ru2}. Such arguments involve making use of repeated differentiation in the $y$ variable of the equation satisfied by $U_0$ in \eqref{hom} and a bootstrap type argument. We finally reach the conclusion that $U_0$ must vanish to infinite order in the $y$ direction at $(x,0,t)$ for all $x \in \Rn$  and all $-1< t< 0$ (we note explicitly that the final part of Step 3 in \cite{Ru2} is based on a Carleman estimate, but we do not use that part!). Moreover, by adapting the arguments in \cite{Ru2} we also obtain that for any $m \in \mathbb{N}$ the function defined by
\[
g_{m}(X, t)=  \begin{cases}
y^{-m} U_0(X, t)\ \ \ \text{for $y\neq 0$},
\\
g_{m}(x, 0, t)= 0,
\end{cases}
\]
is uniformly continuous  and hence bounded on compact  subsets of $\overline{\Rnp} \times (-1, 0)$. It is easy to recognize that such boundedness of $g_m$ for any $m \in \mathbb{N}$ implies the following claim:

\medskip
  
\noindent  \textbf{Claim}: $U_0$ vanishes to infinite order  at $(0,0,t)$ for any $-1<t < 0$. 
  
\medskip

Using the claim and the equation \eqref{hom} satisfied by $U_0$, we will next show that $U_0 \equiv 0$ in $\Sa_{1}^{+}$. This will contradict Corollary \ref{des2}, thus proving the lemma. Since to show that $U_0 \equiv 0$ in $\Sa_{1}^{+}$ it suffices to establish that for any $t\in (-1,0)$ we have $U_0(\cdot,t) \equiv 0$, we turn our attention to proving this.

Again, we argue by contradiction and suppose that there exists $t_0\in (-1,0)$ such that $U_0(\cdot,t_0) \not\equiv 0$. Then, by the continuity of $U_0$ there exists $r_0>0$ such that $U_0(\cdot,t)\not\equiv 0$ for $ t_0 \geq t \geq t_0 - r_0^{2}$. Thus, by possibly restricting $r_0$, we may assume without loss of generality that  
\begin{equation*}
\min\left\{\frac{|t_0|+1}{2}, 2|t_0|\right\} > |t_0 - r_0^2|.
\end{equation*}
We observe that the reason for choosing $r_0$ in this range is that in later considerations we crucially need the ratio $\frac{t_0-r_0^2}{t_0}$ to be bounded from above by an absolute quantity. For this aspect we refer the reader to \eqref{ineq111} and \eqref{2piece} below. From the latter inequality we obtain for $0<r\le r_0$  
\begin{equation}\label{bd11}
\min\left\{\frac{|t_0|+1}{2}, 2|t_0|\right\} > |t_0 - r^2| > |t_0|.
\end{equation}

In order to reach a contradiction, we begin by observing that
 \eqref{f1} gives for all $\rho < 1$
\begin{equation*}
\int_{\Sa_\rho^{+}} y^a U_0^2 \Ga < \infty.
\end{equation*}
This implies for a.e. $t \in (-1, 0)$
\begin{equation}\label{fin}
\int_{\R_{+}^{n+1}} y^{a} U_0^2 (\cdot,t) \Ga (\cdot,t)< \infty.
\end{equation}
Since by Proposition \ref{homg1} we know that $U_0$ is homogeneous of parabolic degree $2\kappa$, by the change of variable $X = \sqrt{\frac{t_1}{t_2}} X'$, we obtain 
\begin{equation}
\int_{\R_{+}^{n+1}} y^{a} U_0^2 (\cdot,t_1) \Ga (\cdot,t_1) = \left(\frac{t_1}{t_2}\right)^{(2\kappa+ a)/2} \int_{\R_{+}^{n+1}} y^{a} U_0^2(\cdot,t_2) \Ga(\cdot,t_2).
\end{equation}
This implies that \eqref{fin} holds for every $-1<t<0$. In a similar way, since the space derivatives of $U_0$ are homogeneous of degree $2\kappa - 1$, we conclude that for every $-1<t<0$
\begin{equation}\label{fin1}
\int_{\R_{+}^{n+1}} y^a |\nabla U_0|^2(\cdot,t) \Ga(\cdot,t) < \infty.
\end{equation}

Let now $-1<t_0 < 0$ be the above number such that $U_0(\cdot,t_0) \not\equiv 0$. We first note that for $ -1 < t < t_0$ we have 
\begin{equation}\label{compare}
e^{-\frac{|X|^2}{4|t-t_0|} } \leq e^{-\frac{|X|^2}{4|t|}}.
\end{equation}
Therefore,
\begin{equation}\label{fin3}
\int_{\Rnp} y^a U_0^2 (X, t)e^{-\frac{|X|^2}{4|t-t_0|} }dX \leq  \int_{\Rnp} y^a U_0^2(X, t) e^{-\frac{|X|^2}{4|t|}}  dX< \infty,
\end{equation}
where in the last inequality we have used \eqref{fin}. Similarly from \eqref{fin1} we have,
\begin{align}\label{fin4}
& \int_{\Rnn_{+}} y^{a} |\nabla U_0(X,t)|^2  e^{\frac{-|X|^2}{4|t-t_0|}} < \infty.
\end{align}
To proceed, let $\Ga(X,t-t_0)$ be the backward heat kernel in $\Rnn \times \R$ centered at $(0,t_0)$, i.e.,
\[
\Ga(X,t-t_0) = (4\pi(t_0 -t))^{-\frac{n+1}{2}} e^{-\frac{|X|^2}{4(t_0-t)}}.
\]
When $t = t_0 - r^2$ we have
\[
\Ga(X,t-t_0)= \Ga(X,-r^2) = (4\pi r^2)^{-\frac{n+1}{2}} e^{-\frac{|X|^2}{4r^2}}.
\]

We next introduce the following quantities centered at $(0,t_0)\in \Rnn \times \R$:
\begin{equation}\label{h6}
\tilde{h}(r) = \tilde{h}(U_0,r)  = \int_{\Rnp }y^a U_0^2 (X,t_0-r^2)  \Ga(X,-r^2) dX,
\end{equation}
and
\begin{equation}\label{I6}
\tilde{i}(r) = \tilde{i}(U_0,r)  =  r^2 \int_{\Rnp} y^a |\nabla U_0|^2(X,t_0-r^2)  \Ga(X,-r^2) dX.
\end{equation}
We note that \eqref{fin3} and \eqref{fin4} ensure that $\tilde{h}(r), \tilde{i}(r)$ are finite for $0<r<r_0$. With these definitions in place we consider the following \emph{frequency} centered at $(0,t_0)$, 
\begin{equation}\label{n6}
\tilde{n}(r) = \frac{\tilde{i}(r)}{\tilde{h}(r)}.
\end{equation}

Now, in Theorem \ref{just5} in the Appendix we have established the following two properties of $U_0$:
\begin{equation}\label{ei10}
\tilde{h}'(r) = \frac{4}{r} \tilde{i}(r)  + \frac{a}{r} \tilde{h}(r),
\end{equation}
and
\begin{equation}\label{n101}
\tilde{n}'(r) \geq 0,\ \ \ \ \ \ \ 0<r< r_0.
\end{equation}

We intend to use \eqref{ei10}, \eqref{n101} to reach the final contradiction that we must have $U_0(\cdot,t_0) \equiv 0$. For this, if we integrate \eqref{ei10} from $r$ to ${r}_0$, and use the monotonicity in \eqref{n101}, we obtain in particular
\begin{equation}\label{fin5}
\tilde{h}(r) \geq  \tilde{h}({r}_0) \left(\frac{r}{{r}_0}\right)^{4 \tilde{n}({r}_0)+ |a|},\ \ \ \ \ \ \ 0<r< r_0.
\end{equation}
On the other hand, from the Claim above we know that $U_0$ vanishes to infinite order at $(0,0,t_0)$. Therefore, for any $\ell>1$ there exists  $C_\ell>0$ depending on $U_0$, $n, a, t_0$ and $\ell$, such that whenever $|X| + |t-t_0|^{1/2} \leq 4r$ we have 
\begin{equation}\label{van1}
|U_0(X, t)| \leq C_\ell\ r^{4\ell}.
\end{equation}
Now 
\begin{align}\label{sp}
\tilde{h}(r) &= \int_{\Rnp \cap \{|X| \leq r^{1/2}\}} y^a U_0^2 (X, t_0- r^2) \Ga(X,-r^2) 
\\
& + \int_{\Rnp \cap \{|X| \geq r^{1/2}\}} y^a U_0^2 (X, t_0-r^2) \Ga(X,-r^2).
\notag
\end{align}
Because of \eqref{van1}, we have 
\begin{equation}\label{1piece}
\int_{\Rnp \cap \{|X| \leq r^{1/2}\}} y^a U_0^2(X,t_0-r^2) \Ga(X,-r^2) \leq C_\ell^2  r^{4\ell+a} \leq C_\ell^2 r^{3\ell},
\end{equation}
 where in the last inequality we have used the fact that $|a| < 1$. For the second integral in \eqref{sp} we obtain with $C_n = (4\pi)^{-\frac{n+1}{2}}$
 \begin{align}\label{ineq111}
 & \int_{\Rnp \cap \{|X| \geq r^{1/2}\}}y^a U_0^2(X,t_0-r^2) \Ga(X,-r^2) \le  \frac{C_n }{r^{n+1} } \int_{|X| \geq r^{1/2}} y^a U_0^2(X,t_0-r^2) e^{-\frac{|X|^2}{4 r^2}}
\\
& = \frac{C_n}{r^{n+1} }\int_{|X| \geq r^{1/2}} y^a U_0^2(X,t_0-r^2) e^{\frac{|X|^2}{4( t_0 - r^2)}} e^{-\frac{|X|^2 t_0}{4r^2(t_0-r^2)}}
\notag\\
& \text{(by \eqref{bd11}, which gives $\frac{t_0}{t_0-r^2} >\frac 12$, and $|X| \geq r^{1/2}$ we have $e^{-\frac{|X|^2 t_0}{4r^2(t_0-r^2)}} \le e^{-\frac{1}{8r}} $)}
\notag\\
& \leq  \frac{C_n}{r^{n+1}} e^{-\frac{1}{8r}} \int_{|X| \geq r^{1/2}} y^a U_0^2(X,t_0-r^2) e^{\frac{|X|^2}{4(t_0-r^2)}} 
\notag\\
& \text{(using the change of variable $X= \sqrt{\frac{t_0-r^2}{t_0}} X'$ and the $2\kappa$-homogeneity}
\notag\\
& \text{  of $U_0$ in Proposition \ref{homg1})}
\notag
\\
& \leq  \frac{C_n}{r^{n+1}} e^{-\frac{1}{8r}} \left(\frac{t_0-r^2}{t_0}\right)^{\frac{4\kappa+ a+n+1}2} \int_{\Rnn_{+}} y^a U_0^2 (X,t_0)e^{-\frac{|X|^2}{4|t_0|}} 
\notag\\
& = \frac{1}{r^{n+1}} e^{-\frac{1}{8r}} \left(\frac{t_0-r^2}{t_0}\right)^{\frac{4\kappa+ a+n+1}2} |t_0|^{\frac{n+1}{2}} \int_{\Rnn_{+}} y^a U_0^2 (X,t_0) \Ga(X,t_0). 
\notag
\end{align}
From \eqref{fin} above (which remember, we have shown it holds for every $-1<t<0$) we know that  the number
\[
M(U_0,n,a,t_0) \overset{def}{=} |t_0|^{\frac{n+1}{2}} \int_{\Rnn_{+}} y^a U_0^2 (X,t_0) \Ga(X,t_0) < \infty.
\]
We also know by \eqref{bd11} that, in particular,
\begin{equation}\label{later}
1 < \frac{t_0-r^2}{t_0} < 2.
\end{equation}
On the other hand, given any $\ell >1$, there exists $\tilde{C}= \tilde{C}(n,\ell)>0$ such that
\[
\frac{1}{r^{n+1}} e^{-\frac{1}{8r}} \leq \tilde{C} r^{\ell}.
\]
We thus conclude from \eqref{ineq111} that
\begin{equation}\label{2piece}
\int_{\Rnp \cap \{|X| \geq r^{1/2}\}}y^a U_0^2(X,t_0-r^2) \Ga(X,-r^2) \le M(U_0,n,a,t_0)  2^{\frac{4\kappa+ a+n+1}2} \tilde{C}  r^{\ell}.
\end{equation}

Combining \eqref{sp}, \eqref{1piece} and \eqref{2piece}, we finally infer
that the Claim leads to the conclusion that for every $\ell>1$
\begin{equation*}
\tilde{h}(r) = O(r^\ell).
\end{equation*}
This clearly contradicts \eqref{fin5}, and thus the Claim cannot possibly be true and we must have $U_0(\cdot,t) \equiv 0$ for every $-1<t<0$. This completes the proof.

\end{proof}

The reader should note that, unlike what has been done in Section \ref{S:mono}, where we have introduced the averaged quantities $H(U,r), I(U,r)$ and $N(U,r)$, in the proof of Lemma \ref{nont} we do not take averages of $\tilde h(r), \tilde i(r)$, and it suffices to study the frequency $\tilde n(r)$ in \eqref{n6}. This difference is justified by the fact that presently we are only interested in the unique continuation property of $U_0$, not in any  further blow-up analysis of $U_0$ itself. Therefore, we do not need for $U_0$ uniform estimates such as \eqref{15}, \eqref{16}.

With Lemma \ref{nont}  in hand, we can now prove our main result, Theorem \ref{main}.

\begin{proof}[Proof of Theorem \ref{main}]
From Lemma \ref{nont} we know that $U_0(\cdot,0,\cdot)\not\equiv 0$. Therefore, there exist constants $L, c> 0$, and $0< \rho < \rho_1 < 1$, such that
\begin{equation*}
||U_0 ||_{L^{\infty}(B_L \times \{0\}\times [-\rho_1, -\rho]) } = c > 0.
\end{equation*}
We note in passing that from the homogeneity of $U_0$, we can take $L=1$. Since Lemma \ref{convergence} asserts that for a certain sequence $r_j\searrow$ the Almgren rescalings $U_j = U_{r_j}$ converge to $U_0$ uniformly on compact subsets of $\overline{\Rnp} \times (-1,0)$, we infer that for all sufficiently large $j\in \mathbb N$ one has
\begin{equation*}
||U_j||_{L^{\infty}(B_L \times \{0\} \times  [-\rho_1,  0])} \geq ||U_j||_{L^{\infty}(B_L \times \{0\} \times [-\rho_1,-\rho])} \geq \frac{c}{2} > 0.
\end{equation*}
From the latter inequality, the definition of $U_j$ in \eqref{bl}, and the fact that $U(x,0,t) = u(x,t)$, we obtain
\begin{equation*}
||u||_{L^{\infty}(B_{Lr_j} \times [-\rho_1 r_j^2, 0])} \geq \frac{c}{2}\ r_j^{-\frac a2}\ \sqrt{H(U,r_j)}.  
\end{equation*}
We now use \eqref{iner10} in Corollary \ref{growth} to infer from the latter inequality 
\begin{equation*}
||u||_{L^{\infty}(B_{Lr_j} \times [-\rho_1 r_j^2, 0])} \geq \frac{c}{2} \sqrt{H(U, r_0)} \left(\frac{r_{j}}{r_0}\right)^{2 ||\overline{N}||_{\infty} +\frac{a}{2}} r_j^{-\frac{a}{2}}.
\end{equation*}
This contradicts that $u$ vanishes to infinite order at $(0,0)$ according to Definition \ref{d11}. Therefore, keeping Corollary \ref{int} in mind, we reach the conclusion that $u(\cdot,t)  \equiv 0$  for all $ - r_0^2< t \leq 0$. At this point, we argue in a standard way to conclude that $u \equiv 0$ in $\Rn\times (-\infty,0]$. 

\end{proof}

We close this section with an important comment concerning the proof of Lemma \ref{nont}.

\begin{rmrk}\label{bl1}
If we  knew that $U_0$ were well defined at $t=0$ and vanished to infinite order at $(0,0,0)$, then we would  immediately reach a contradiction from the $2 \kappa$-homogeneity of $U_0$ in Proposition \ref{homg1}. This is precisely what happens in the elliptic case, see for instance \cite{FF}.
The difficulty in our proof is caused by the fact that from the hypothesis $U_0(\cdot,0,\cdot) \equiv 0$ we can only infer that $U_0$  vanishes to infinite order at points $(x,0,t)$ for all $x \in \Rn$ and $-1< t<0$. However, we cannot assert that $U_0$ vanishes to infinite order at  $(x,0,0)$. Now, given that $U_0$ vanishes to infinite order at $(x,0,t)$, we cannot use the  homogeneity of $U_0$ to conclude that $U_0 \equiv 0$ in $\Sa_{1}^{+}$.  Consider  for instance the following  function defined by
\begin{equation}
\begin{cases}
 f(x,y,t)= y e^{-\frac{ |x|^2 + |t|}{ y^2}},\ \ \ \ x\in \Rn, y \neq 0,\ -\infty< t<0,
\\
 f(x,0,t) =0,\ \ \ \ \ x\in \Rn, \ -\infty< t<0.
 \end{cases}
\end{equation}
It is clear that such $f$ is homogeneous of degree $1$ and vanishes to infinite order at points $(x,0,t)$ for all $x \in \Rn$ and all $t <0$, and yet $f\not\equiv 0$! Of course, the function $f$ does not vanish to infinite order at $(0,0,0)$. Therefore, in order to establish the Claim in the proof of Lemma \ref{nont} it is essential to further utilize the equation satisfied by $U_0$, which is what we have done. 
\end{rmrk}

\section{Appendix}\label{S:app}


In this section we complete the proofs of Lemma \ref{F1}, Proposition \ref{homg1} and Lemma \ref{nont} by establishing for $U_0$ results analogous to those that in Section \ref{S:mono} were obtained for the solution $U$ to the extension problem \ref{ext}. We emphasize that such work is not redundant since an Almgren blowup $U_0$ does not a priori possess the same regularity properties that in Section \ref{S:nltol} have been proved to hold for $U$. 
 
We begin by obtaining second derivative estimates for $U_0$ in the Gaussian space. In this, we again borrow some ideas from the case $s= 1/2$ in \cite{DGPT}. 

\begin{lemma}\label{firstlemma}
With $U_0$ as in Lemma \ref{convergence}, for $i=1,...,n$, we have for any $\rho <1$ 
\begin{equation}\label{sd}
\int_{\Sa_\rho^{+}} y^a t^2 \left((U_0)_t^2 \Ga+  |\nabla (U_0)_i|^2\right) \Ga< \infty.
\end{equation}
\end{lemma}

\begin{proof}
We first show that one has 
\begin{align}\label{end}
& \int_{\Sa_{\rho}^{+}} y^a t^2  |\nabla (U_0)_i|^2 \Ga \leq C, \ \ \ \ \ \ i=1, .... n.
\end{align}

The argument for this is similar to that used in the proof of Lemma \ref{F1}, based on differentiating the equation \eqref{hom} with respect to $x_i$, $i=1,...,n$. We let 
\begin{equation*}
\tau_2= |t| \tau_0\left(\frac{X}{\sqrt{|t|}}\right)
\end{equation*}
where, as in Lemma \ref{F1}, $\tau_0\in C^\infty_0(\Rnn)$ is a cut-off function, such that $\tau_0^2 \le 1$, whose support at the end of the process will exhaust the whole $\Rnn$. For a fixed $i= 1,...,n$, we define
\begin{equation}\label{tst}
\eta= (U_0)_i \tau_2^2 \Ga.
\end{equation}
Let now $r \in [\rho, \rho_1]$ for some $\rho_1$ such that  $\rho< \rho_1 < 1$. In the weak formulation of the equation  \eqref{hom} we choose as test function $ \phi= \eta_i$. Using the infinite differentiability of $U_0$ with respect to $x, t$, similarly to \eqref{wk} we find  
\begin{align*}
& \int_{\Sa_{r}^{+} \setminus \Sa_{\delta}^{+}}  y^a  \left(<\nabla (U_0), \nabla \eta_i> + (U_0)_t \eta_i \right)  = 0.
\end{align*}
Integrating by parts with respect to $x_i$, we have
\begin{align*}
& \int_{\Sa_{r}^{+} \setminus \Sa_{\delta}^{+}}  y^a \left(<\nabla (U_0)_i,\nabla \eta >  + \partial_t (U_0)_i \eta \right)  = 0.
\end{align*}
Using \eqref{tst} and \eqref{z11}, we obtain
\begin{align}\label{in61}
& \int_{\Sa_{r}^{+} \setminus \Sa_{\delta}^{+}} y^a \left[ |\nabla (U_0)_i|^2 \tau_2^2 \Ga + \frac{1}{4t} Z((U_0)_i^2) \tau_2^2 \Ga + 2(U_0)_i \tau_2< \nabla (U_0)_i ,\nabla \tau_2 >\Ga \right] =0.
\end{align}
Now   we estimate the second and third term in \eqref{in61} separately.  
Similarly to \eqref{wk1}, by a standard argument 
which can be found on p. 95 of the Appendix A in \cite{DGPT} we  obtain
\begin{equation}\label{wk5}
\int_{\Sa_r^{+} \setminus \Sa_{\delta}^{+}}  y^a \frac{1}{4t} Z((U_0)_i^2) \tau_2^2 \Ga \geq - \frac{r^{2+a}}{2} \int_{\R^{n+1}_{+}} y^a(U_0)_{i}^2(\cdot,-r^2) \tau_1^2 \Ga(\cdot,-r^2).
\end{equation}
where $\tau_1$ is as in \eqref{tau1}. 

Concerning the third term in \eqref{in61},  applying Young's inequality we find 
\begin{align}\label{ine62}
& \left|2 \int_{\Sa_{r}^{+} \setminus \Sa_{\delta}^{+}} y^a (U_0)_i \tau_2< \nabla (U_0)_i ,\nabla \tau_2 >\Ga\right| \leq \frac{1}{2} \int_{\Sa_{r}^{+} \setminus \Sa_{\delta}^{+}} y^a |\nabla (U_0)_i|^2 \tau_2^2 \Ga
\\
& + 2 \int_{\Sa_{r}^{+} \setminus \Sa_{\delta}^{+}}   y^a   (U_0)_i^2  |\nabla  \tau_2|^2\Ga.
\notag
\end{align}
Using \eqref{wk5}, \eqref{ine62}  in \eqref{in61}, and keeping in mind that $0<r<1$ and $a+1>1$, we obtain  
\begin{align*}
& \int_{\Sa_{r}^{+} \setminus \Sa_{\delta}^{+}} y^a |\nabla (U_0)_i|^2 \tau_2^2 \Ga \leq   \int_{\R^{n+1}_{+}} y^a(U_0)_{i}^2(\cdot,-r^2) \tau_1^2 \Ga(\cdot,-r^2) + 4 \int_{\Sa_{r}^{+} \setminus \Sa_{\delta}^{+}} y^a   (U_0)_i^2  |\nabla  \tau_2|^2\Ga.
\end{align*}
Integrating the latter inequality in $r \in [\rho, \rho_1]$, we find for some universal $C>0$
\begin{align}\label{it64}
& \int_{\Sa_{\rho}^{+} \setminus \Sa_{\delta}^{+}} y^a |\nabla (U_0)_i|^2 \tau_2^2 \Ga \leq  C \int_{\Sa_{\rho_1}^{+} \setminus \Sa_{\delta}^{+}} y^a\left((U_0)_i^2   |\nabla  \tau_2|^2 \Ga + (U_0)_i^2 \tau_1^2 \Ga\right).
\end{align}
Keeping in mind that $\tau_1^2 \leq |t|, |\nabla \tau_2|^2 \leq C|t|$, we obtain from \eqref{it64} 
\begin{align}\label{e116}
& \int_{\Sa_{\rho}^{+} \setminus \Sa_{\delta}^{+}} y^a |\nabla (U_0)_i|^2 \tau_2^2 \Ga \leq  C \int_{\Sa_{\rho_1}^{+} \setminus \Sa_{\delta}^{+}} y^a |t|  (U_0)_i^2  \Ga.  
\end{align}
We now recall that \eqref{f1} implies that, for any $\rho_1 < 1$, we have 
 \begin{equation}\label{f100}
 \int_{\Sa_{\rho_1}^{+}} y^a  |t| |\nabla U_0|^2   \Ga  < C_1,
 \end{equation}
 where $C_1$ depends on $\rho_1$. We use \eqref{f100} to bound the integral in the right-hand side of \eqref{e116}, and consequently obtain
 \begin{align*}
& \int_{\Sa_{\rho}^{+} \setminus \Sa_{\delta}^{+}} y^a |\nabla (U_0)_i|^2 \tau_2^2 \Ga \leq  C_2.
\end{align*}
Finally, if we let supp($\tau_0) \nearrow \R^{n+1}$ in the latter inequality, and noting that  $\tau_2^2(X,t) \nearrow t^2$, by the monotone convergence theorem, and subsequently letting $\delta \to 0$, we thus conclude that for $i=1,...,n$ and any $\rho<1$,  
\begin{align}\label{end}
& \int_{\Sa_{\rho}^{+}} y^a t^2  |\nabla (U_0)_i|^2 \Ga \leq C,
\end{align}
with $C = C(\rho)>0$.

In order to complete the proof of  the lemma we next show that for any $\rho<1$
\begin{equation}\label{end1}
  \int_{\Sa_{\rho}^{+}} y^a t^2  (U_0)_t^2  \Ga < \infty.
 \end{equation} 
With this objective in mind, we first establish Gaussian estimates for second derivatives of $U_0$ in the $y$ direction in order to prove that $(y^a(U_0)_y)_y$ in $L^{2}( \Sa_{\rho}^{+}, y^{-a} |t| \Ga dX dt)$ for all $\rho<1$. Subsequently, we use the equation \eqref{hom} satisfied by $U_0$ in combination with \eqref{end} to deduce \eqref{end1}. To implement these steps, we consider the conjugate  equation satisfied by $v= |y|^{a} (U_0)_y$.  More precisely, from our discussion in \eqref{conj1}-\eqref{e501}, after an even reflection of $U_0$ across $\{y=0\}$,  we know that, for any $\delta>0$, $v$ satisfies the following equation in the weak sense in $\Sa_1 \setminus \Sa_{\delta}$ 
\begin{align}\label{cj}
\operatorname{div}(|y|^{-a} \nabla v) = |y|^{-a} v_t.
\end{align}
Taking $\eta= v\tau_2^2 G$ as a test function in the weak formulation of \eqref{cj}, we can now argue as in \eqref{tst}-\eqref{e116}  above, with $a$ replaced by $-a$, and finally obtain by a limiting type argument  that  for any $\rho < 1$,  
\begin{equation}\label{ne2}
\int_{\Sa_\rho} |y|^{-a} t^2  |\nabla v|^2 \Ga \leq C \int_{\Sa_{\tilde{\rho}}} |y|^{-a} |t| v^2\Ga = C \int_{\Sa_{\tilde{\rho}}} |y|^{a} |t| (U_0)_y^2 \Ga < \infty,
\end{equation}
where $\tilde{\rho}$ is any arbitrary number such that  $\rho < \tilde{\rho}  < 1$.  Note that in the last inequality in \eqref{ne2}, we have used \eqref{f100} above. Since $v= |y|^a (U_0)_y$, we trivially have for $y>0$
\begin{equation*}
|(y^a (U_0)_y)_y| \leq |\nabla v |.  
\end{equation*}
Using this observation in the left-hand side of \eqref{ne2}, we find
\begin{equation}\label{ne13}
\int_{\Sa_{\rho}^{+}} y^{-a} |t|^2 (y^a (U_0)_y)_y^2 \Ga \leq C,
\end{equation}
for some $C$ universal depending  on $U_0$ and  $\rho$. Using the equation \eqref{hom} satisfied by $U_0$, in combination with
 \eqref{end} and \eqref{ne13}, we finally obtain the estimate \eqref{end1}. 
 
\end{proof}

We combine the conclusions of Lemmas \ref{F1} and \ref{sd} in the following estimate, which is valid for any $\rho < 1$ and $ i=1,...,n$:
\begin{equation}\label{est10}
\int_{\Sa_\rho^{+}} y^a \left(U_0^2 + |t| |\nabla U_0|^2 + t^2 |\nabla  (U_0)_i|^2 + t^2 (U_0)_t^2\right) \Ga < \infty.
\end{equation}
We also claim that the following inequality holds for $0 < \rho < 1$, 
\begin{equation}\label{est11}
\int_{\Sa_\rho^{+}} y^a  (ZU_0)^2 \Ga  < \infty.
\end{equation}
To establish \eqref{est11}, we multiply by $|t|$ both sides of \eqref{iner5}, obtaining the  following inequality for any $v \in W^{1,2}(\R^{n+1}, y^a \Ga dX)$ 
\begin{equation}\label{ne1}
\int_{\Rnp}  y^a v^2 \frac{|X|^2 }{|t|} \Ga \leq  C \int_{\Rnp}  y^a (v^2 + |t| |\nabla v|^2)\Ga
\end{equation}
Here, it is important to observe that in establishing \eqref{iner5} we have never used the equation satisfied  by $U$, thus the inequality actually holds for every $v \in W^{1,2}(\Rnp, y^a \Ga dX)$. Since we trivially have $x_i^2 \leq |X|^2$, by applying \eqref{ne1} to $v = U_0$ at every time level we obtain for $i=1,...,n$  and all $\rho<1$, 
\begin{equation}\label{ne3}
\int_{\Sa_{\rho}^{+}} y^a (U_0)_i^2 x_i^2  \Ga \leq  \int_{\Sa_{\rho}^{+}} y^a (U_0)_i^2 |X|^2 \Ga \leq C \int_{\Sa_{\rho}^{+}} y^a  (|t| (U_0)_i^2 + |t|^2 |\nabla  (U_0)_i|^2)\Ga < \infty,
\end{equation}
where in the last  inequality we used \eqref{end}. If we now apply \eqref{ne1} with $v= y^a (U_0)_y$, but with $a$ replaced  by $-a$, we obtain for $\rho<\rho_1<1$, 
\begin{align}\label{ne4}
& \int_{\Sa_\rho^{+}}  y^{a} (U_0)_y^2 y^2 \Ga = \int_{\Sa_\rho^{+}} y^{-a} v^2 y^2 \Ga \leq  \int_{\Sa_{\rho}^{+}} y^a (U_0)_y^2 |X|^2 \Ga
\\
&  \leq C \int_{\Sa_{\rho_1}^{+}} y^{-a} ( |t| v^2 + |t|^2 |\nabla v|^2 )\Ga < \infty,
\notag
\end{align}
where in the last inequality in \eqref{ne4} above,  we used \eqref{ne2}. We note that the use of the inequality \eqref{ne1} in \eqref{ne3} and \eqref{ne4} can be justified by an approximation  argument using cut-offs. The Cauchy-Schwarz inequality and the definition of $Z$ give 
\[
(ZU_0)^2 \leq C\left ( \Sigma_{i=1}^{n} x_i^2 (U_0)_i^2  + y^2 (U_0)_y^2 + t^2 (U_0)_t^2\right)
\]
From this observation, \eqref{est10}, \eqref{ne3} and \eqref{ne4} it is easy to conclude that \eqref{est11} holds.

We are now ready to establish the counterpart of Theorem \ref{mon1} for any Almgren blowup $U_0$. We mention explicitly that the next result rigorously completes the proof of Proposition \ref{homg1}.

\begin{thrm}\label{just4}
With $U_0$ as in Lemma \ref{convergence}, the following formula hold for any $0< r<1$:
\begin{equation}\label{h10}
H'(U_0, r)= \frac{4}{r} I(U_0, r) + \frac{a}{r} H(U_0, r),
\end{equation}
\begin{equation}\label{i10}
I'(U_0, r)=  \frac{a}{r} I(U_0, r) + \frac{1}{r^3} \int_{\Sa_r^{+}} y^{a} (ZU_0)^2 \Ga,
\end{equation}
and
\begin{align}\label{U_0}
& rN'(U_0, r)=   \frac{\int_{\Sa_r^{+}} y^a (ZU_0)^2 \Ga}{ \int_{\Sa_{r}^{+}} y^a U_0^2 \Ga}
 -  \frac{ (\int_{\Sa_r^{+}} y^a U_0 ZU_0 \Ga)^2}{ (\int_{\Sa_r^{+}} y^a U_0^2 \Ga)^2}.
\end{align}
In particular, by an application of Cauchy-Schwarz inequality we conclude from \eqref{U_0} that $N'(U_0,r) \ge 0$, and therefore $r \to N(U_0, r)$
is non-decreasing in $(0,1)$. 
\end{thrm}

\begin{proof}
With \eqref{est10} and \eqref{est11} in hands, the arguments are similar to those in Section \ref{S:mono}. Let $\tau_{N}\in C^\infty_0(\Rnn)$ be such that $\tau_{N} \equiv 1$ for $|X| \leq N$, and is identically zero outside $\B_{2N}$. For a given $\delta > 0$ we let 
\begin{equation}
H_{N}^{\delta}(r) = H_{N}^{\delta} (U_0,r) \overset{def}{=} \frac{1}{r^2} \int_{-r^2}^{-\delta r^2}  \int_{\Rnn_{+}} y^a U_0^2 \tau_N\Ga.
\end{equation}
By a change of variable we have
\begin{equation}\label{ft1}
H_{N}^{\delta}(r)= \int_{-1}^{-\delta} h_N (r^2 t) dt,
\end{equation}
where
\begin{equation}\label{der1}
h_N(t) = h_{N}(U_0,t) \overset{def}{=} \int_{\Rnn_{+}} y^a \tau_N(\cdot,t) U_0^2(\cdot,t) \Ga(\cdot,t)dX.
\end{equation}
Since $\tau_N$ is compactly supported and $\delta>0$, using the regularity properties of $U_0$ and Lebesgue dominated convergence, the formal computations for $h_{N}'$ can be justified similarly to those for $h' (U,\cdot)$ in Section \ref{S:mono}. More precisely, we know that $U_0$ is infinitely differentiable in the variables $(x,t)$ up to $\{y=0\}$, and also that given any multi-index $\alpha$ of length $n+1$, $y^a (D_{x, t}^{\alpha}U_0)_y$ is H\"older continuous on compact subsets of $\overline{\Rnp} \times (-1,0)$. As previously mentioned after Definition \ref{D:albu} above, such regularity properties of $U_0$ follow from Theorem \ref{reg5} and Lemma \ref{reg2} by taking repeated difference quotients of $U_0$ in the  $x, t$ variable as in \eqref{rep}. Consequently, with
\begin{equation*}
i_N(t) = i_N(U_0, t) \overset{ded}{=} - t \int_{\Rnn_{+}} y^a \tau_N(\cdot,t)  |\nabla U_0(\cdot,t)|^2  \Ga(\cdot,t) dX,
\end{equation*}
we have
\begin{align}\label{der2}
h_N'(t) & = \frac{2}{t} i_N(t) + \frac{a}{2t} h_N(t) + \int_{\Rnn_{+}} y^a U_0^2 < \nabla \tau_N, \frac{X}{2t} > \Ga
\\
& -2 \int_{\Rnn_{+}} y^a  U_0 < \nabla U_0, \nabla \tau_N> \Ga.
\notag
\end{align}
Corresponding to $i_N$ above, we  now let
 \begin{equation*}
I_{N} ^{\delta}(r) = I_{N} ^{\delta}(U_0,r) \overset{def}{=} \frac{1}{r^2} \int_{-r^2}^{-\delta r^2}  i_N(t) dt.
 \end{equation*}
 Again by a change of variable, we have
 \begin{equation}\label{ft2}
 I_{N} ^{\delta}(r) =  \int_{-1}^{-\delta }  i_N(r^2 t) dt.
 \end{equation}
At this point differentiating under the integral sign $H_N^{\delta}$ in \eqref{ft1} can be justified  from the expression of $h_N'$ in \eqref{der2}, by using the local regularity estimates for $U_0$  and dominated convergence theorem. We thus deduce from \eqref{der2} that 
 \begin{align}\label{exp1}
& \frac{dH_{N}^{\delta}}{dr}(r)= \frac{4}{r} I_{N} ^{\delta}(r) + \frac{a}{r} H_{N}^{\delta}(r) + \frac{1}{r^3}  \int_{\Sa_{r}^{+} \setminus \Sa_{\delta r}^{+}} y^a  U_0^2 < \nabla \tau_N, X> \Ga 
\\
&-\frac{4}{r^3} \int_{\Sa_{r}^{+} \setminus \Sa_{\delta r}^{+}} y^ a t  U_0 < \nabla U_0, \nabla \tau_N> \Ga.
\notag
  \end{align}
  
Now, because of \eqref{est10}, \eqref{est11}, given $\ve, \rho$ such that   $0< \ve < \rho <1$, we note that for $r \in [\ve,\rho]$ the following terms in \eqref{exp1} 
 \[
 \frac{1}{r^3}  \int_{\Sa_{r}^{+} \setminus \Sa_{\delta r}^{+}} y^a  U_0^2 < \nabla \tau_N, X> \Ga,\ \ \ \ \text{and}\ \ \ \  
  \frac{4}{r^3} \int_{\Sa_{r}^{+} \setminus \Sa_{\delta r}^{+}} y^ a t  U_0 < \nabla U_0, \nabla \tau_N> \Ga,
  \]
can be made uniformly small in $N$, as $N \to \infty$. Here, we also use that
$\left|<\nabla \tau^N, X>\right|$
is bounded from above by a universal constant. This follows from the fact that  
$|\nabla \tau_N| \leq \frac{C}{N}$
 for some universal $C>0$, and that $\nabla \tau_N$ is supported in $N \leq |X| \leq 2N$.
 Moreover, again because of \eqref{est10} and \eqref{est11}, we can assert that the following term in \eqref{exp1}
 \[
 \frac{4}{r} I_{N} ^{\delta}(r) + \frac{a}{r} H_{N}^{\delta}(r) 
 \]
 converges to $\frac{4}{r} I^{\delta} + \frac{a}{r} H^{\delta}$ uniformly as $N \to \infty$, 
 where
 \[
 H^{\delta}(r)= \frac{1}{r^2} \int_{-r^2}^{-\delta r^2}  \int_{\Rnn_{+}} y^a U_0^2 \Ga dX dt\ \ \ \ \text{and}\ \ \ \
 I ^{\delta}(r) = \frac{1}{r^2} \int_{-r^2}^{-\delta r^2}  |t| \int_{\Rnp} y^a |\nabla U_0|^2 \Ga dX dt. 
 \] 
We thus see that the sequences $\{H_{N}^{\delta}(U_0,r)\}$ and $ \{\frac{dH_{N}^{\delta}(U_0, r)}{dr}\}$ are  uniformly  convergent  in $[\ve,\rho]$ as $N \to \infty$. Therefore, letting $N \to \infty$ we finally obtain
 \[
 \frac{dH^{\delta}(U_0,r)}{dr} = \frac{4}{r} I ^{\delta}(U_0,r) + \frac{a}{r} H^{\delta}(U_0,r).
 \]
Let now $\delta_k\searrow 0$. Because of \eqref{est10} and \eqref{est11} again, the functions  $\{H^{\delta_k}(U_0,r)\}$ and $ \{\frac{dH^{\delta_k}(U_0,r)}{dr}\}$ are uniformly convergent in $[\ve, \rho]$  as $k \to \infty$. Letting $k \to \infty$ we thus conclude that for all $r \in [\ve, \rho]$
 \begin{equation*}
H'(U_0,r) = \frac{4}{r} I(U_0,r) + \frac{a} {r} H(U_0,r).
\end{equation*}
The arbitrariness of $\ve, \rho$ implies that the latter equality does in fact hold for all $0<r<1$. This establishes \eqref{h10}.

We next turn to the proof of \eqref{i10}. Our first step is establishing the counterpart for $U_0$ of Lemma \ref{L:I}, see \eqref{equiv2} below. We first note that an integration by parts based on the regularity properties of $U_0$, and similar to the one used in deriving \eqref{der2} above, gives 
 \begin{align}\label{equiv1}
I_{N}^{\delta}(r) & = \frac{1}{2r^2} \int_{\Sa_{r}^{+} \setminus \Sa_{\delta r}^{+}} y^a U_0 (ZU_0) \tau_N \Ga + \frac{1}{r^2} \int_{\Sa_{r}^{+} \setminus \Sa_{\delta r}^{+}} y^a  t U_0< \nabla U_0, \nabla \tau_N>\Ga.
\end{align}
Using \eqref{est10}, \eqref{est11}, one recognizes that as $N\to \infty$
 \begin{equation}\label{l11}
 \frac{1}{r^2} \int_{\Sa_{r}^{+} \setminus \Sa_{\delta r}^{+}} y^a  t U_0< \nabla U_0, \tau_N>\Ga\ \longrightarrow\ 0.
 \end{equation}
Moreover, again by \eqref{est11}, by first letting $N \to \infty$ and then $\delta \to 0$, one finds
 \begin{equation}\label{l13}
 \frac{1}{2r^2} \int_{\Sa_{r}^{+} \setminus \Sa_{\delta r}^{+}} y^a  U_0 (ZU_0) \tau_N \Ga\ \longrightarrow\   \frac{1}{2r^2} \int_{\Sa_{r}^{+}}  y^aU_0 (ZU_0) \Ga.
\end{equation}
From \eqref{equiv1}, \eqref{l11} and \eqref{l13} we conclude that the following alternate  expression for $I(U_0,r)$ holds
\begin{equation}\label{equiv2}
I(U_0,r)= \frac{1}{2r^2} \int_{\Sa_{r}^{+}} y^a U_0 (ZU_0)  \Ga.
\end{equation}

We next turn to the computation of $i_N'(r)$ since such quantity is needed in the proof of \eqref{i10}. We claim that
\begin{align}\label{iprimoN}
i_N'(t) &  = \frac{a}{2t} i_N(t) + \frac{1}{2t} \int_{\Rnn_{+}} y^a  (ZU_0)^2 \tau_N \Ga
\\
& +  \int_{\Rnn_{+}} y^a Z U_0 < \nabla \tau_N,\nabla U_0> \Ga - t \int_{\Rnn_{+}}y^a |\nabla U_0|^2 < \nabla  \tau_N, \frac{X}{2t}> \Ga. 
\notag
 \end{align}
The proof of \eqref{iprimoN} is based on arguments similar to those used in the computation of $i_{\ve} '(U,r)$  in \eqref{trunc}, by first considering integrals over the region $\Rn \times \{y > \ve \}$ as in Step 2 in the proof of Lemma \ref{L:fvei}. By the regularity estimates for $U_0$ and arguing as in Step 2, after letting $\ve \to 0$ one obtains \eqref{iprimoN}. From such formula one proceeds to differentiate under the integral  sign in \eqref{ft2} by using local regularity estimates for $U_0$ and dominated convergence theorem, finally arriving to the following expression
 \begin{align}\label{ne5}
\frac{d I_{N}^{\delta}}{dr}(r) & = \frac{a}{r} I_{N}^{\delta}(r) + \frac{1}{r^3} \int_{\Sa_r^{+} - \Sa_{\delta r}^{+}}  y^a (ZU_0)^2 \tau_{N}\Ga
 \\
 & + \frac{2}{r^3} \int_{\Sa_r^{+} - \Sa_{\delta r}^{+}} y^a t ZU_0 < \nabla \tau_N, \nabla U_0> \Ga - \frac{1}{r^3} \int _{\Sa_r^{+} - \Sa_{\delta r}^{+}}y^a  t |\nabla U_0|^2 < \nabla  \tau_N, X> \Ga.
 \notag
 \end{align}
Using again \eqref{est10}, \eqref{est11}, by a limiting type argument similar to that for $H'(U_0,r)$ above, first letting $N \to \infty$ and then $\delta \to 0$, we  conclude that for $0<r<1$
 \begin{equation}
 I'(U_0, r) = \frac{a}{r} I(U_0, r) + \frac{1}{r^3} \int_{\Sa_r^{+}} y^a (ZU_0)^2 \Ga, 
 \end{equation}
which establishes \eqref{i10}.

Finally, \eqref{U_0} follows from \eqref{h10}, \eqref{i10} and the alternate expression of $I(U_0, r)$ in \eqref{equiv2} above. 
  
 \end{proof}

We close this Appendix by justifying formula \eqref{ei10} and the monotonicity \eqref{n101} of the frequency $\tilde{n}(r)$ defined in \eqref{n6}. This is needed to complete the proof of Lemma \ref{nont} above.
The reader should keep in mind that, since we have already proved Theorem \ref{just4},  Proposition \ref{homg1} is now available to us. Therefore, in the proof of the next Theorem \ref{just5} we can use that $U_0$ is parabolically homogeneous of degree $2\kappa$.

\begin{thrm}\label{just5}
 Let $\tilde{n}(r)$  be  as in \eqref{n6},  and let $r_0>0$ be such that for all $0 < r \leq r_0$ the condition \eqref{bd11} hold. Then, for all $r \in (0,r_0]$ the following is true: 
  \begin{equation}\label{h100}
\tilde{h}'(r) = \frac{4}{r} \tilde{i}(r)  + \frac{a}{r} \tilde{h}(r),
\end{equation}
\begin{equation}\label{i100}
\tilde{i}'(r) = \frac{a}{r} \tilde{i}(r)
 + \frac{1}{r} \int_{\Rnp} y^a (Z_{t_0} U_0(X, t_0-r^2))^2 \Ga(X,-r^2),
\end{equation}
where 
\begin{equation}\label{Z_0}
Z_{t_{0}} U_0(X, t_0-r^2) =\ < \nabla U_0(X, t_0-r^2) , X>  - 2r^2  \partial_t U_0 (X, t_0-r^2).
\end{equation}
Finally, for $0<r<r_0$ we have 
\begin{equation}\label{n100}
\tilde{n}'(r) \geq 0.
\end{equation} 
\end{thrm}
 
\begin{proof}
 
 With $\tau_N$ as in the proof of Theorem \ref{just4}, we let
 \begin{equation}
 \tilde{h}_N (r) = \int_{\Rnp}  y^a U_0^2(X, t_0-r^2)  \tau_N(X) \Ga(X, -r^2)  dX
 \end{equation}
 and
 \begin{equation}
 \tilde{i}_{N} (r) = r^2 \int_{\Rnp} y^a |\nabla U_0|^2 (X, t_0-r^2)\tau_N(X) \Ga(X, -r^2) dX.
 \end{equation}
 Then, by computations similar to those in Section \ref{S:mono}, that are justified using the local regularity properties of $U_0$ and Lebesgue dominated convergence, we find 
 \begin{align}\label{frt13}
\tilde{h}'_{N}(r) & = \frac{4}{r} \tilde{i}_N(r) + \frac{a}{r} \tilde{h}_N +\frac{1}{r} \int_{\Rnp}  y^a  U_0^2(X, t_0-r^2) < \nabla  \tau_N, X > \Ga(X, t_0-r^2)
\\
&+ 4r\int_{\Rnp} y^ a  U_0 (X, t_0-r^2) < \nabla U_0(X, t_0-r^2), \nabla \tau_N> \Ga(X, -r^2).
\notag
 \end{align}
For any $\delta>0$ such that $\delta < r_0$, we let $r \in  [\delta, r_0]$. The first integral in the right-hand side of \eqref{frt13} can be estimated as follows. Note that \eqref{compare} gives for $r \in [\delta, r_0]$,
\[
\Ga (X, -r^2) \leq \frac{|t|^{(n+1)/2}}{\delta^{n+1}} \Ga(X, t_0-r^2) \leq \frac{1}{\delta^{n+1}} \Ga(X, t_0-r^2),
\]
where in the last inequality we have used $|t| <1$. Since $r \geq \delta$, the above inequality gives
\begin{align}\label{fr10}
&  \frac{1}{r} \int_{\Rnp} y^a  U_0^2 (X, t_0-r^2)< \nabla  \tau_N, X > \Ga(X, -r^2) 
\\
& \leq \frac{C_n}{ \delta^{n+2}} \int_{\Rnp \cap \{|X|> N\}}  y^a U_0^2(X, t_0-r^2)  \Ga(X, t_0-r^2).
\notag
\end{align}
In obtaining \eqref{fr10} we have also used that $|<\nabla \tau_N, X>|$ is supported in $\{|X| > N\}$ and is bounded from above by a universal $C_n>0$. The change of variable 
 \begin{equation}\label{change}
 X= X' \sqrt{\frac{t_0-r^2}{t_0}}
 \end{equation}
 and the $2\kappa$-homogeneity of $U_0$ allow to infer that the right-hand side of \eqref{fr10} be estimated as follows  
  \begin{align}\label{fr11}
& \frac{C_n}{ \delta^{n+2}} \int_{\Rnp \cap \{|X|> N\}}  y^a U_0^2(X, t_0-r^2)  \Ga(X, t_0-r^2)
\\
& \leq   \frac{C_n}{ \delta^{n+2}} \left(\frac{t_0-r^2}{t_0}\right)^{(4\kappa+a)/2} \int_{\Rnp \cap \{|X'|> \frac{N}{\sqrt{2}}\} }y^a U_0^2(X', t_0) \Ga(X',t_0).
\notag
 \end{align}
  Note that in \eqref{fr11} we have also used the fact that, thanks to \eqref{later} above, $|X| > N$ implies
$|X'| > \frac{N}{\sqrt{2}}$.
Using again \eqref{later}, we infer from \eqref{fr10}, \eqref{fr11}  
  \begin{align}\label{fr13}
 &  \frac{1}{r} \int_{\Rnp} y^a  U_0^2 (X, t_0-r^2)< \nabla  \tau_N, X > \Ga(X, -r^2) dX
  \\
  & \leq  \frac{C_n}{ \delta^{n+2}} 2^{(4\kappa+a)/2} \int_{\Rnp \cap \{|X'|> \frac{N}{\sqrt{2}}\} }y^a U_0^2(X', t_0) \Ga(X', t_0) dX'.
  \notag
  \end{align}
Because \eqref{fin} holds for $t=t_0$, we see that by choosing $N$ large enough the integral in the right hand side of \eqref{fr13} can be made arbitrarily small independently of $r \in [\delta, r_0]$. We can consequently assert from \eqref{fr13} that, by choosing $N$ large enough, the  term
 \[
 \frac{1}{r} \int_{\Rnp} y^a  U_0^2 (X, t_0-r^2)< \nabla  \tau_N, X > \Ga(X, -r^2) dX
 \]
in the right-hand side of \eqref{frt13}  can be made arbitrarily small, uniformly in $r  \in [\delta, r_0]$. 
  
Arguing similarly to \eqref{fr10}-\eqref{fr13}, using the homogeneity  of $U_0$ and $\nabla U_0$, for large enough $N$ the second integral in the right-hand side of \eqref{frt13} 
\[
 4r\int_{\Rnp} y^ a  U_0 (X, t_0-r^2) < \nabla U_0(X, t_0-r^2), \nabla \tau_N> \Ga(X, -r^2)
  \]
can be made uniformly small  for $r \in [\delta, r_0]$.
  
The arguments \eqref{fr10}-\eqref{fr13} also show that  
 \begin{equation}\label{fr14}
\underset{N\to \infty}{\lim} \underset{r \in [\delta, r_0]}{\sup} |\tilde{h}(r) - \tilde{h}_N(r)| \leq \int_{\Rnp \cap \{|X| > N\}} y^a U_0^2(X, t_0-r^2) \Ga(X,-r^2)dX\ \longrightarrow\ 0.
 \end{equation}
Therefore, letting $N \to \infty$ in \eqref{frt13} we finally obtain for all $r \geq \delta$ 
 \begin{equation*}
 \tilde{h}^{'}(r) = \frac{4}{r} \tilde{i}(r) + \frac{a}{r} \tilde{h}(r).
\end{equation*}
From the arbitrariness of  $\delta$ we have thus proved \eqref{h100}.

Before turning to the proof of \eqref{i100} we next establish the following alternate expression for $\tilde{i}(r)$:
\begin{equation}\label{equiv5}
\tilde{i} (r) = \frac{1}{2} \int_{\Rnp} y^a (Z_{t_0} U_0(X, t_0-r^2)) U_0(X, t_0-r^2)   \Ga(x,-r^2), 
\end{equation}  
where $Z_{t_0} U_0(X, t_0-r^2)$ is given by \eqref{Z_0} above. To prove \eqref{equiv5} we  integrate by parts using the equation satisfied by $U_0$, as in the derivation of Lemma \ref{L:ibi}. We find 
\begin{align}\label{equiv4}
\tilde{i}_{N}(r) & = \frac{1}{2} \int_{\Rnp} y^a (Z_{t_0} U_0(X, t_0-r^2)) U_0(X, t_0-r^2)  \tau_N \Ga(x, -r^2) 
\\
& - r^2 \int_{\Rnp} y^a U_0(X, t_0-r^2) < \nabla U_0(X, t_0-r^2), \nabla \tau_N>\Ga(X, -r^2).
\notag
\end{align}
Using the homogeneity of $U_0, \nabla U_0$, and arguing as in \eqref{fr10}-\eqref{fr13} above, we prove that as $N\to \infty$
\begin{equation}\label{con10}
- r^2 \int_{\Rnp} y^a U_0(X, t_0-r^2) < \nabla U_0(X, t_0-r^2), \nabla \tau_N>\Ga(X, -r^2)\ \longrightarrow\ 0.
\end{equation}
We also  note that from  the estimates \eqref{est10},  \eqref{ne3}, \eqref{ne4} and the homogeneity of $ \nabla U_0$ and $  < X, \nabla U_0> $,  the following holds for all $t \in (-1, 0)$ 
\begin{equation}\label{e110}
\int_{\Rnp} y^a  ( |\nabla U_0|^2(X, t)+ < X, \nabla U_0(X, t)>^2 )  \Ga(X,t) dX < \infty.
\end{equation}
From Lemma \ref{sd} and the homogeneity of $(U_0)_t$ we find for all $t \in (-1,0)$
\begin{equation}\label{e111}
\int_{\Rnp} y^a (U_0)_{t}^2 (X,t)\Ga(X, t) < \infty.
\end{equation}
The equations \eqref{e110}, \eqref{e111}  coupled with \eqref{compare} imply for $r \in (0, r_0]$
\begin{equation}\label{z}
\int_{\Rnp} y^a (Z_{t_0} U_0)^2(X, t_0-r^2) \Ga(X,  -r^2) dX< \infty.
\end{equation}
The finiteness of $\tilde{h}(r)$ (which follows from \eqref{fin} and \eqref{compare}), that of the integral in \eqref{z}, and the Cauchy-Schwarz inequality imply 
\[
\int_{\Rnp} y^a \left|Z_{t_0} U_0(X, t_0-r^2) \right| \left|U_0(X, t_0-r^2) \right|  \Ga(x, -r^2) < \infty.
\]
Since $\tau_N(X)  \to 1$ as $N \to \infty$ for all $X \in \Rnp$, by Lebesgue dominated convergence we obtain as $N \to \infty$
\begin{align}\label{con11}
&\frac{1}{2} \int_{\Rnp} y^a (Z_{t_0} U_0(X, t_0-r^2)) U_0(X, t_0-r^2)  \tau_N \Ga(x, -r^2)  
\\
& \longrightarrow\ \frac{1}{2} \int_{\Rnp} y^a (Z_{t_0} U_0(X, t_0-r^2)) U_0(X, t_0-r^2)   \Ga(x, -r^2). 
\notag
\end{align}
Since $\tilde{i}_{N} \to \tilde{i}$ as $N \to \infty$, from \eqref{equiv4}, \eqref{con10} and \eqref{con11} we finally conclude that \eqref{equiv5} does hold.

We finally turn to the proof of \eqref{i100} and \eqref{n100}. Similarly to the computation of $i_{\ve}'(U,t)$ in \eqref{trunc}, the computation of $\tilde{i}'_{N}$ is justified as in Step 2 in the proof of Lemma \ref{L:fvei} using the regularity properties of $U_0$, instead of those of $U$. Consequently, we obtain 
\begin{align}\label{comp10}
\tilde{i}_N'(r) & = \frac{a}{r} \tilde{i}_N(r) + \frac{1}{r}\int_{\Rnp }y^a (Z_{t_0} U_0(X, t_0-r^2)^2 \tau_N \Ga(X, -r^2) 
\\
& + r \int_{\Rnp } y^a |\nabla U_0|^2(X, t_0-r^2) <\nabla \tau_N, X> \Ga(X,-r^2)
\notag
\\
& - 2r \int_{\Rnp} y^a Z_{t_0} U_0(X, t_0-r^2) < \nabla \tau_N, \nabla U_0(X, t_0-r^2)> \Ga(X, -r^2).
\notag
\end{align}
Using \eqref{e110}, \eqref{e111} and the homogeneity of $U_0$ and of its derivatives, by an  argument similar to that for $\tilde{h}'$ in  \eqref{fr10}-\eqref{fr13}, for any $\delta > 0$ we have
\[
r \int_{\Rnp } y^a |\nabla U_0|^2 (X, t_0-r^2)<\nabla \tau_N, X> \Ga(X, -r^2)\ \longrightarrow\ 0,
\]
and 
\[
- 2r \int_{\Rnp} y^a Z_{t_0} U_0(X, t_0-r^2) < \nabla \tau_N, \nabla U_0(X, t_0-r^2)> \Ga(X, -r^2)\ \longrightarrow\ 0,
\]
uniformly in $r \in [\delta, r_0]$. Similarly, as $N\to \infty$ we obtain 
\begin{equation*}
\left|\frac{1}{r} \int_{\Rnp }y^a (Z_{t_0} U_0(X, t_0-r^2) )^2 \tau_N \Ga(X, -r^2) - \frac{1}{r} \int_{\Rnp }y^a (Z_{t_0} U_0(X, t_0-r^2) )^2 \Ga(X, -r^2)\right|\ \longrightarrow\ 0,
\end{equation*}
uniformly in $r \in [\delta, r_0]$. This proves that as as $N \to \infty$ the following convergences \[
\tilde{i}_N(r)\ \longrightarrow\ \tilde{i}(r),\ \ \ \tilde{i}_N'(r)\ \longrightarrow\ \frac{a}{r} \tilde{i}(r) + 
\frac{1}{r} \int_{\Rnp }y^a (Z_{t_0} U_0(X, t_0-r^2) )^2 \Ga(X,-r^2),
\]
are uniform in $[\delta, r_0]$. It ensues that for any $r \in [\delta, r_0]$ we have 
\[
\tilde{i}'(r) = \frac{a}{r} \tilde{i}(r) + \frac{1}{r} \int_{\Rnp }y^a (Z_{t_0} U_0(X, t_0-r^2) )^2 \Ga(X, -r^2). 
\]
From the arbitrariness of $\delta>0$ we conclude that \eqref{i100} holds for all $r \in (0,r_0]$.

Finally, from \eqref{h100}, \eqref{i100}, \eqref{equiv5} and the Cauchy-Schwarz inequality, we conclude that \eqref{n100} holds.  

\end{proof}

\end{document}